\documentclass[
reqno,
11pt,
oneside
]{article}

\usepackage[ngerman, english]{babel}
\usepackage[utf8]{inputenc}
\usepackage{enumerate}
\usepackage[normalem]{ulem}
\usepackage[babel,french=guillemets,german=swiss]{csquotes}
\usepackage[center,font=small]{caption}
\usepackage{cite}
\usepackage{bbm}
\usepackage[raggedright]{titlesec}
\usepackage{xcolor}

\usepackage{tocbasic}
\DeclareTOCStyleEntry[
beforeskip=.2em plus 1pt,
]{tocline}{section}

\usepackage[%
left=1in,
right=1.25in,
bottom=1.25in,
top=1in,
footskip=0.5in,
]{geometry}

\usepackage{color}
\definecolor{LinkColor}{rgb}{0,0,1}
\definecolor{LinkColor2}{rgb}{0,0.5,0}
\definecolor{lbcolor}{rgb}{0.85,0.85,0.85}
\definecolor{FrameColor}{rgb}{0.85,0.85,0.85}
\definecolor{rosso}{rgb}{0.8,0,0}
\definecolor{lightgray}{rgb}{0.5,0.5,0.5}

\usepackage{enumitem}

\usepackage{graphicx}
\usepackage{placeins}
\usepackage{overpic}

\usepackage{amsmath}
\usepackage{amssymb}
\usepackage{dsfont}
\usepackage{empheq}
\usepackage{amsthm}
\numberwithin{equation}{section}

\usepackage[%
pdftitle={Titel},%
pdfauthor={Autor},%
pdfcreator={LaTeX, LaTeX with hyperref and KOMA-Script},
pdfsubject={Betreff}, 
pdfkeywords={Keywords}
]{hyperref} 

\hypersetup{%
	colorlinks	=true,
	linkcolor	=LinkColor,%
	anchorcolor	=LinkColor,%
	citecolor	=LinkColor2,%
	filecolor	=LinkColor,%
	menucolor	=LinkColor,%
	urlcolor	=LinkColor,%
}

\newtheorem{theorem}{Theorem}[section]

\newtheorem{proposition}[theorem]{Proposition}

\newtheorem{definition}[theorem]{Definition}

\theoremstyle{definition}
\newtheorem{remark}[theorem]{Remark}

\makeatletter
\renewenvironment{proof}[1][\proofname]{%
	\par\pushQED{\qed}\normalfont%
	\topsep6\p@\@plus6\p@\relax
	\trivlist\item[\hskip\labelsep\bfseries#1\@addpunct{.}]%
	\ignorespaces
}{%
	\popQED\endtrivlist\@endpefalse
}
\makeatother

\makeatletter
\renewcommand\paragraph{\@startsection{paragraph}{4}{\z@}%
	{1ex \@plus1ex \@minus.2ex}%
	{-1em}%
	{\normalfont\normalsize\bfseries}}
\renewcommand\subparagraph{\@startsection{paragraph}{4}{\z@}%
	{1ex \@plus1ex \@minus.2ex}%
	{-1em}%
	{\normalfont\normalsize\itshape}}
\makeatother

\newcommand{\abs}[1]{\left| #1 \right|}
\newcommand{\bigabs}[1]{\big| #1 \big|}

\newcommand{\norm}[1]{\| #1 \|}
\newcommand{\bignorm}[1]{\big\| #1 \big\|}
\newcommand{\ang}[2]{ \langle #1 , #2  \rangle}
\newcommand{\bigang}[2]{ \big< #1 , #2  \big>}
\newcommand{\scp}[2]{ \left( #1 , #2  \right)}
\newcommand{\bigscp}[2]{\big( #1 , #2 \big)}
\newcommand{\mean}[1]{\langle #1 \rangle}
\newcommand{\meano}[1]{{\langle #1 \rangle}_{\Omega}}
\newcommand{\meang}[1]{{\langle #1 \rangle}_{\Gamma}}

\newcommand{\R}{\mathbb R}

\newcommand{\n}{\mathbf{n}}
\renewcommand{\v}{{\boldsymbol v}}
\newcommand{\w}{{\boldsymbol w}}
\newcommand{\p}{\overline{p}}

\newcommand{\D}{\mathbf{D}}

\newcommand{\N}{\mathbb{N}}

\newcommand{\bH}{\mathbf{H}}
\newcommand{\J}{\mathbf{J}}
\newcommand{\K}{\mathbf{K}}
\newcommand{\bL}{\mathbf{L}}
\newcommand{\T}{\mathbf{T}}

\newcommand{\DD}{\mathcal{D}}
\newcommand{\HH}{\mathcal{H}}
\newcommand{\LL}{\mathcal{L}}

\newcommand{\UU}{\mathcal{U}}
\newcommand{\ZZ}{\mathcal{Z}}

\newcommand{\Db}{{\mathcal{D}_{\beta}}}
\newcommand{\Da}{{\mathcal{D}_{\alpha}}}
\newcommand{\Dbp}{{\mathcal{D}_{\beta}^{\prime}}}

\newcommand{\ZZK}{\ZZ_k}

\newcommand{\mom}{m_\Omega}
\newcommand{\mga}{m_\Gamma}
\newcommand{\nga}{n_\Gamma}

\newcommand{\trho}{\tilde{\rho}}

\newcommand{\tJ}{\tilde{\mathbf J}}
\newcommand{\tK}{\tilde{\mathbf K}}

\newcommand{\tT}{\tilde{\mathbf T}}

\newcommand{\Om}{\Omega}
\newcommand{\Ga}{\Gamma}
\newcommand{\Si}{\Sigma}

\newcommand{\intO}{\int_\Omega}

\newcommand{\intG}{\int_\Gamma}

\newcommand{\eps}{\varepsilon}

\newcommand{\dtau}{\;\mathrm d\tau}
\newcommand{\dx}{\;\mathrm dx}

\newcommand{\dt}{\;\mathrm dt}

\newcommand{\dS}{\;\mathrm dS}

\newcommand{\dG}{\;\mathrm d\Ga}

\newcommand{\h}{\mathds{h}}

\newcommand{\ddt}{\frac{\mathrm d}{\mathrm dt}}

\newcommand{\del}{\partial}
\newcommand{\delt}{\partial_{t}}

\newcommand{\deln}{\partial_\n}

\newcommand{\delph}{\partial_{\phi}}
\newcommand{\delgph}{\partial_{\Grad\phi}}
\newcommand{\delps}{\partial_{\psi}}
\newcommand{\delgps}{\partial_{\Gradg\psi}}

\newcommand{\Grad}{\nabla}
\newcommand{\Lap}{\Delta}
\newcommand{\Div}{\textnormal{div}}
\newcommand{\Gradg}{\nabla_\Ga}
\newcommand{\Lapg}{\Delta_\Ga}
\newcommand{\Divg}{\textnormal{div}_\Ga}

\newcommand{\emb}{\hookrightarrow}

\newcommand{\suchthat}{\;\ifnum\currentgrouptype=16 \middle\fi|\;}

\newcommand{\revised}[1]{{\color{black}#1}}

\begin{document}

%
%

\title{\bf
    Two-phase flows with bulk-surface interaction:\\
	thermodynamically consistent Navier--Stokes--Cahn--Hilliard models\\
	with dynamic boundary conditions}	

\author{
Andrea Giorgini\footnote{Research supported by MUR grant Dipartimento di Eccellenza 2023-2027.}\\[7pt]
\small Politecnico di Milano\\
\small Dipartimento di Matematica \\ 
\small Via E. Bonardi 9, 20133 Milano, Italy\\
\small \href{mailto:andrea.giorgini@polimi.it}{andrea.giorgini@polimi.it}
\and 
Patrik Knopf\footnote{Research supported by the DFG (Deutsche Forschungsgemeinschaft): Project 52469428 and RTG 2339.}\\[7pt]
\small Universität Regensburg \\
\small Fakultät für Mathematik\\
\small Universitätsstr.~31, D-93053 Regensburg\\
\small \href{mailto:patrik.knopf@ur.de}{patrik.knopf@ur.de}
}


\date{ }

\renewcommand{\thefootnote}{\fnsymbol{footnote}}

	
\maketitle

\begin{center}
    \vspace{-3ex}
	\scriptsize
	{
		\textit{This is a preprint version of the paper. Please cite as:} \\  
		A. Giorgini and P. Knopf, \textit{J. Math. Fluid Mech.} 25:65 (2023), \\
		\url{https://doi.org/10.1007/s00021-023-00811-w}
	}
 \vspace{1ex}
\end{center}

%
%

\begin{abstract}
We derive a novel thermodynamically consistent Navier--Stokes--Cahn--Hilliard system with dynamic boundary conditions. 
This model describes the motion of viscous incompressible binary fluids with different densities.
In contrast to previous models in the literature, our new model allows for surface diffusion, a variable contact angle between the diffuse interface and the boundary, and mass transfer between bulk and surface. 
In particular, this transfer of material is subject to a mass conservation law including both a bulk and a surface contribution.
The derivation is carried out by means of local energy dissipation laws and the Lagrange multiplier approach.
Next, in the case of fluids with matched densities, we show the existence of global weak solutions in two and three dimensions as well as the uniqueness of weak solutions in two dimensions. 	
\\[1ex]
\textit{Keywords:} 
Two-phase flows, 
Navier--Stokes--Cahn--Hilliard system, 
bulk-surface interaction, 
dynamic boundary conditions
\\[1ex]	
\textit{Mathematics Subject Classification:}
35D30,
35Q35,
76D03,
76D05,
76T06,
76T99.
\end{abstract}

 \begin{small}
 	\setcounter{tocdepth}{2}
 	\hypersetup{linkcolor=black}
\tableofcontents
 \end{small}

\setlength\parindent{0ex}
\setlength\parskip{1ex}
\allowdisplaybreaks

\section{Introduction}

The description of multi-phase flows is a primal topic in modern continuum fluid dynamics with enormous applications in biology, chemistry and engineering. Two major approaches have been developed according to the representation of the interface separating the different components (see \cite{du2020phase, AbelsGarckeReview, giga2017variational,pruss2016moving} and the references therein): Sharp Interface (SI) versus Diffuse Interface (DI) methods (see Figure~\ref{fig:DI-SI}). In the former, the interface is represented by a hypersurface in the surrounding domain, which leads to the formulation of free boundary problems. In the latter, also called phase-field method, the interface is represented by a layer with finite thickness, corresponding to the level sets of the concentration function whose evolution is governed by a macroscopic equation. The main advantage of this approach is the Eulerian formulation of the phase-field equation, which avoids tracking the interface as in free boundary problems. Although conceptually opposites, these approaches are strictly related since a proper scaling of DI systems approximates SI models in the so-called {\it sharp interface limit} (see, e.g., \cite{AbelsGarckeReview}).  

\begin{figure}[h]
\centering
\includegraphics[width=0.8\textwidth]{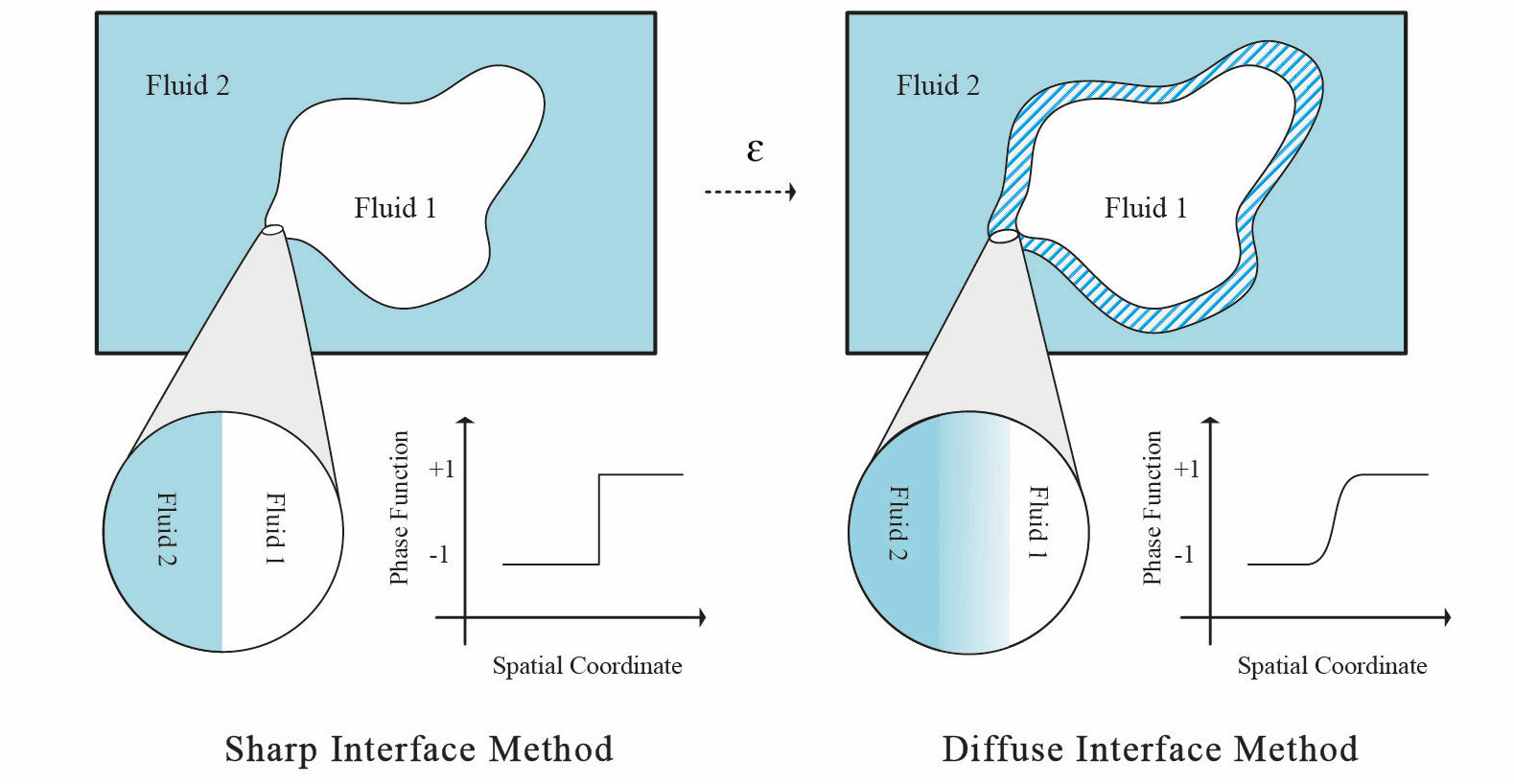}
\caption{Difference between the Sharp Interface (left) and Diffuse Interface  (right) methods.}
\label{fig:DI-SI}
\end{figure}

In the context of the DI theory, the cornerstone system for the motion of two viscous and incompressible fluids with matched (constant) densities is the model H, which was rigorously derived in \cite{GurtinPolignoneVinals}. The model H consists of the following Navier--Stokes--Cahn--Hilliard system
\begin{subequations}
\label{ModelH}
\begin{alignat}{2}
    \label{mH:1}
    &\rho \delt \v  
    + \rho \Div\big( \v \otimes \v \big)
    - \Div\big( 2\nu(\phi) \, \D\v \big)
    + \Grad p
    = - \eps\, \Div\big(\Grad\phi \otimes \Grad\phi \big)
    &&\quad\text{in $Q$}, 
    \\[1ex]
    \label{mH:2}
    &\Div \, \v = 0
    &&\quad\text{in $Q$}, 
    \\[1ex]
    \label{mH:3}
    &\delt \phi + \Div(\phi \v )
    = \Div\big( \mom(\phi) \Grad \mu \big)
    &&\quad\text{in $Q$}, 
    \\[1ex]
    \label{mH:4}
    &\mu = -\eps \Lap \phi + \eps^{-1} F'(\phi)
    &&\quad\text{in $Q$}.
\end{alignat}
\end{subequations}
Here, $Q:=\Omega\times (0,T)$, where $\Omega\subset\R^d$ with $d= 2,3$ denotes a bounded, smooth domain with boundary $\Gamma$ and $T>0$ is a given final time. 
In \eqref{ModelH}, $\rho$ is the constant density of the mixture, $\v$ is the velocity of the mixture, $p$ is the pressure, $\phi$ is the difference of the fluid concentrations, $\mu$ is the chemical potential. In addition, $\nu$ is the concentration depending viscosity, $\D$ is the symmetric gradient operator, $\varepsilon$ is a positive parameter related to the thickness of the interface, $m_\Omega$ is the (bulk) mobility function and $F$ is a double-well potential. 

The fundamental assumption in deriving the model H is that the density of the mixture is constant, meaning that the two fluids have the same constant density $\rho$ (matched densities). In the past two decades, several works investigated suitable Navier--Stokes--Cahn--Hilliard generalizations of the model H aiming to describe both incompressible mixtures with unmatched densities and compressible two-phase flows. Without claiming completeness, we refer the reader to \cite{AGG,lowengrub1998quasi, boyer2002theoretical, ding2007diffuse, shen2013mass,shokrpour2018diffuse,giga2017variational,heida2012development,freistuhler2017phase}.
Among them is the thermodynamically consistent model for viscous incompressible mixtures with unmatched densities proposed by Abels, Garcke and Gr\"{u}n (AGG model) in the seminal work \cite{AGG}, which reads as follows:
\begin{subequations}
\label{AGG}
\begin{alignat}{2}
    \label{AGG:1}
    &\delt \big(\rho(\phi) \v\big) 
    + \Div\big( \v \otimes (\rho(\phi)\v + \J) \big)
    - \Div\big( 2\nu(\phi) \, \D\v \big)
    + \Grad p
    = - \eps\, \Div\big(\Grad\phi \otimes \Grad\phi \big)
    &&\quad\text{in $Q$}, 
    \\[1ex]
    \label{AGG:2}
    &\Div\, \v = 0
    &&\quad\text{in $Q$},
    \\[1ex]
    \label{AGG:3}
    &
    \revised{\rho(\phi)= \trho_1 \frac{1-\phi}{2}+ \trho_2 \frac{1+\phi}{2},}
    \quad 
    \J = - \tfrac 12 (\trho_2 - \trho_1) 
        \mom(\phi) \Grad\mu
    &&\quad\text{in $Q$},
    \\[1ex]
    \label{AGG:4}
    &\delt \phi + \Div(\phi \v )
    = \Div\big( \mom(\phi) \Grad \mu \big)
    &&\quad\text{in $Q$}, 
    \\[1ex]
    \label{AGG:5}
    &\mu = -\eps \Lap \phi + \eps^{-1} F'(\phi)
    &&\quad\text{in $Q$}.
\end{alignat}
\end{subequations}
In \eqref{AGG}, the novel (coupling) terms are the concentration depending density $\rho(\phi)$, where $\trho_1$ and $\trho_2$ are the different (constant) densities of the two fluids, and the additional mass flux $\J$.

In the literature, both the model H and the AGG model have mostly been supplemented with the classical no-slip/homogeneous Neumann boundary conditions
\begin{subequations}
\label{classic-BC}  
\begin{alignat}{2}
    \label{classic-BC:v} 
    \v &= \mathbf{0}
    &&\quad \text{on $\Sigma$},\\
    \label{classic-BC:phi} 
    \partial_\n \phi &= 0
    &&\quad \text{on $\Sigma$},\\
    \label{classic-BC:mu}
    \partial_\n \mu &= 0
    &&\quad \text{on $\Sigma$}.
\end{alignat}
\end{subequations}
Here, $\Sigma:=\Gamma\times (0,T)$ and $\n$ denotes the outward normal vector field on $\Gamma$. 
The conditions \eqref{classic-BC:v} and \eqref{classic-BC:mu} ensure the mass conservation in the bulk, namely $\int_\Omega \phi(t) \dx$ is a conserved quantity.
Sufficiently regular solutions of the system (\eqref{AGG},\eqref{classic-BC}) satisfy the energy dissipation law
\begin{align}
	\label{CONS:ENERGY:AGG}
		\ddt E(\v,\phi)
		= - \intO 2\nu(\phi)\, \abs{\D\v}^2 \dx
		- \intO \mom(\phi) \abs{\Grad \mu}^2 \dx, 
\end{align}
where the \textit{total energy} $E$ is given by
\begin{align}
	\label{DEF:E:AGG}
	E(\v,\phi) 
	&:= E_\text{kin}(\phi,\v) + E_\text{bulk}(\phi),
\end{align}
where
\begin{align}
	\label{DEF:EKIN+EBULK:AGG}
	E_\text{kin}(\v,\phi) 
	:= \intO \frac{\rho(\phi)}{2} \abs{\v}^2 \dx, \quad
	E_\text{bulk}(\phi) := \intO \frac{\eps}{2} \abs{\Grad\phi}^2 + \frac{1}{\eps} F(\phi) \dx
\end{align}
denote the \textit{kinetic energy} and the \textit{free energy in the bulk}, respectively.
We point out that in this model, the fluid concentrations on the surface do not affect the dynamics in the bulk.

Concerning the mathematical analysis, the Cauchy problems corresponding to \eqref{ModelH} and \eqref{AGG}, each endowed with \eqref{classic-BC}, have been studied in \cite{Abels2009, boyer1999, GalGrasselli2010, GMT2019} and in 
\cite{abels2013existence,abels2013incompressible,AbelsWeber2021, giorgini2021well, giorgini2022-3D}, respectively. We also refer the interested reader to \cite{AbelsGarckeReview} (see also \cite{AGG}) for the connection between the above DI systems and the two-phase Navier--Stokes free boundary problem. 

Even though the AGG model (\eqref{AGG},\eqref{classic-BC}) is already very advanced and is capable of describing the complex case where the two fluids have different specific densities, it still inherits some limitations from the underlying (convective) Cahn--Hilliard equation with homogeneous Neumann boundary conditions. The main limitations are:
\begin{enumerate}[label = \textnormal{(L\arabic*)}, ref = \textnormal{L\arabic*}]
	\item\label{L1} The boundary condition \eqref{classic-BC:phi} enforces the diffuse interface separating both fluids to always intersect the boundary $\Gamma$ at a perfect angle of ninety degrees. Of course, this restrictive condition will not always be satisfied in concrete applications as the contact angle of the interface might deviate from ninety degrees and even change dynamically over the course of time.
	Moreover, as discussed in \cite{Qian-Wang-Sheng}, the no-slip boundary condition \eqref{classic-BC:v} is not well-suited for describing general moving contact line phenomena. This is because situations where the velocity field actually contributes to the motion of the contact line of the diffuse interface in the bulk at the boundary of the domain cannot be described.
	\item\label{L2} The boundary condition \eqref{classic-BC:mu} can be regarded as a no-flux boundary condition as it already implies $\J\cdot\n = 0$ on $\Sigma$. Therefore, the system (\eqref{AGG},\eqref{classic-BC}) can merely describe the situation where the mass of both fluids in the bulk $\Omega$ is conserved but it is not capable of describing a transfer of material between bulk and surface which could, for instance, be caused by absorption or adsorption processes as well as chemical reactions taking place on the boundary (see, e.g., \cite{KLLM2021}).
\end{enumerate}
Due to these well-known limitations of the Cahn--Hilliard model, and in order to describe short-range interactions between bulk and surface more precisely, physicists proposed that the total free energy should contain an additional contribution on the surface being also of Ginzburg--Landau type (see, e.g., \cite{binder-frisch}):
\begin{align}
	\label{DEF:EFREE}
	E_\text{free}(\phi,\psi)
	:= E_\text{bulk}(\phi) + E_\text{surf}(\psi)
	\quad\text{where}\quad
	E_\text{surf}(\psi)
	:= \intG \frac{\delta\kappa}{2} \abs{\Gradg\psi}^2 + \frac{1}{\delta} G(\psi) \dS,
\end{align}
and $E_\text{bulk}$ is defined as in \eqref{DEF:EKIN+EBULK:AGG}. Here, $\psi$ is an additional phase-field variable describing the material distribution on the surface, the parameter $\delta>0$ corresponds to the width of the diffuse interface on the surface and the constant $\kappa\ge 0$ acts as a weight for surface diffusion effects.
In most cases, $\psi$ is just assumed to be the trace of the phase-field $\phi$ on the boundary, i.e.,
\begin{align}
    \label{BC:TRACE}
    \phi\vert_\Sigma = \psi \quad\text{on $\Sigma$}.
\end{align}
However, also more general relations between $\phi$ and $\psi$ (so-called transmission conditions) have been investigated in the literature (see, e.g., \cite{CFL,knopf-lam}).
The function $G$ stands for an additional potential on the surface. If phase separating processes on the boundary are to be described, $G$ is also chosen to be double-well shaped.

Based on this free energy, several dynamic boundary conditions for the Cahn--Hilliard (CH) equation%
\begin{subequations}
\label{CH}
    \begin{alignat}{2}
    \label{CH:1}
    &\delt \phi 
    = \Div\big( \mom(\psi) \Grad \mu \big) 
    &&\;\;\text{in $Q$}, 
    \\[1ex]
    \label{CH:2}
    &\mu = - \eps \Lap \phi + \delta^{-1} F'(\phi)
    &&\;\;\text{in $Q$}
\end{alignat}
\end{subequations}
have been introduced in the literature. For instance, the Allen--Cahn type dynamic boundary condition
\begin{align}
    \label{DBC:AC}
    \delta\delt \psi = \kappa\delta\Lapg\psi - \eps\deln\phi - \frac{1}{\delta} G'(\psi)
    \quad\text{on $\Sigma$},
\end{align}
to replace \eqref{classic-BC:phi} was proposed in \cite{kenzler} and further analyzed in \cite{CGM}.
Here, the symbol $\Lapg$ denotes the Laplace--Beltrami operator (see, e.g., \cite{dziuk2013finite} for differential operators on surfaces).

In recent times, dynamic boundary conditions which also exhibit a Cahn--Hilliard type structure have become very popular, especially since they allow for a better description of the transfer of material between bulk and surface (see, e.g., \cite{KLLM2021}). In these models, the boundary condition \eqref{classic-BC:phi} is replaced by a Cahn--Hilliard type equation on the surface, that is 
\begin{subequations}
\label{DBC:CH}
    \begin{alignat}{2}
    \label{DBC:CH:1}
    &\delt \psi 
    = \Divg\big( \mga(\psi) \Gradg \theta \big) - \beta \mom(\psi)\, \deln \mu 
    &&\quad\text{on $\Sigma$}, 
    \\[1ex]
    \label{DBC:CH:2}
    &\theta = - \delta\kappa \Lapg \psi + \delta^{-1} G'(\psi)
    + \eps \deln \phi
    &&\quad\text{on $\Sigma$}.
\end{alignat}
Here, $\theta$ denotes the chemical potential on the boundary, the function $\mga$ describes the mobility and the parameter $\kappa\ge 0$ acts as a weight for surface diffusion effects. Furthermore, the boundary condition \eqref{classic-BC:mu} needs to be replaced by a suitable coupling condition for the bulk chemical potential $\mu$ and the surface chemical potential $\theta$. Recently, for a parameter $\beta\neq 0$, the following choices have been considered:
\begin{align}
    \label{DBC:CH:3}
    \begin{cases}
        L\mom(\psi) \, \deln \mu = \beta \theta - \mu 
        &\text{if $L\in[0,\infty)$},\\
        \mom(\psi)\, \deln \mu = 0 
        &\text{if $L = \infty$}.
    \end{cases}
\end{align}
\end{subequations}
\begin{itemize}
\item If $L=0$, \eqref{DBC:CH:3} is to be interpreted as the Dirichlet type coupling condition $\beta\theta = \mu$ on $\Sigma$. This means that the chemical potentials $\mu$ and $\theta$ are assumed to always remain in chemical equilibrium. In this case, the dynamic boundary conditions \eqref{DBC:CH} were proposed in \cite{Gal} (for $\kappa=0$) and in \cite{GMS2011} (for general $\kappa\ge 0$). Sufficiently regular solutions of the system (\eqref{BC:TRACE},\eqref{CH},\eqref{DBC:CH}) satisfy the mass conservation law
\begin{align}
    \label{GMS:MASS}
    \beta\intO \phi(t) \dx + \intG \psi(t) \dS 
    = \beta\intO \phi_0 \dx + \intG \psi_0 \dS
\end{align}
and the energy dissipation law
\begin{align}
    \label{GMS:ENERGY}
    \ddt E_\text{free}\big(\phi(t),\psi(t)\big) 
    = - \intO \mom(\phi) \abs{\Grad \mu}^2 \dx
        - \intG \mga(\psi) \abs{\Gradg \theta}^2 \dS
\end{align}
for all $t\in[0,T]$.
In the case $\beta>0$, the well-posedness of the system (\eqref{BC:TRACE},\eqref{CH},\eqref{DBC:CH}) was studied in \cite{GMS2011} and its long-time behavior was investigated in \cite{GMS2011,GKY2022}.
\item For $L=\infty$, the dynamic boundary conditions \eqref{DBC:CH} were derived in \cite{Liu-Wu2019} by means of an energetic variational approach based on Onsager's law and the least action principle. Here, due to \eqref{DBC:CH:3}, the chemical potentials are not directly coupled. However, mechanical interactions of the materials between bulk and surface are still taken into account by the trace relation~\eqref{BC:TRACE}. Since \eqref{DBC:CH:3} with $L=\infty$ implies that the mass flux between bulk and surface is zero, the bulk mass and the surface mass are conserved separately. To be precise, sufficiently regular solutions of the system (\eqref{BC:TRACE},\eqref{CH},\eqref{DBC:CH}) satisfy the mass conservation law
\begin{align}
    \label{LW:MASS}
    \intO \phi(t) \dx = \intO \phi_0 \dx
    \quad\text{and}\quad
    \intG \psi(t) \dS = \intG \psi_0 \dS
\end{align}
and the energy dissipation law \eqref{GMS:ENERGY} for all $t\in[0,T]$. The well-posedness of the system (\eqref{BC:TRACE},\eqref{CH},\eqref{DBC:CH}) was established in \cite{Liu-Wu2019,GK} and its long-time behavior was studied in \cite{Liu-Wu2019,MW}.
\item In the case $L\in(0,\infty)$, the boundary conditions \eqref{DBC:CH} were proposed and analyzed in \cite{KLLM2021}. Here, the chemical potentials $\mu$ and $\theta$ are coupled by a Robin type boundary condition. The constant $1/L$ is related to the kinetic rate associated with adsorption/desorption processes or chemical reactions on the boundary. 
Sufficiently regular solutions of the system (\eqref{BC:TRACE},\eqref{CH},\eqref{DBC:CH}) satisfy the mass conservation law \eqref{GMS:MASS} and the energy dissipation law
\begin{equation}
    \label{KLLM:ENERGY}
\begin{split}
    \ddt E_\text{free}\big(\phi(t),\psi(t)\big) 
    &= - \intO \mom(\phi) \abs{\Grad \mu}^2 \dx
        - \intG \mga(\psi) \abs{\Gradg \theta}^2 \dS
    \\
    &\qquad - \frac 1L \intG (\beta\theta - \mu)^2 \dS
    \end{split}
\end{equation}
for all $t\in[0,T]$. 
The weak and strong well-posedness of the system (\eqref{BC:TRACE},\eqref{CH},\eqref{DBC:CH}) was established in \cite{KLLM2021}.
It was further shown that the case $L\in(0,\infty)$ can be understood as an interpolation between $L=0$ and $L=\infty$ as these cases are obtained as asymptotic limits on the level of strong solutions to the system (\eqref{BC:TRACE},\eqref{CH},\eqref{DBC:CH}) as the parameter $L$ is sent to zero or to infinity, respectively.
For the investigation of long-time dynamics, we refer to \cite{GKY2022}.
\end{itemize}
We further point out that a nonlocal variant of the system (\eqref{BC:TRACE},\eqref{CH},\eqref{DBC:CH}) was proposed and investigated in \cite{KS}. Recent reviews of results concerning the Cahn--Hilliard equation with dynamic boundary conditions can be found in \cite{Miranville-Book,Wu-Review}.

In the context of Navier--Stokes--Cahn--Hilliard models for two-phase flows, in order to overcome the aforementioned limitation \eqref{L1},
the authors in \cite{Qian-Wang-Sheng} proposed a variational derivation through Onsager's principle of maximal energy dissipation of a new class of generalized Navier slip boundary conditions for two-phase flows
\begin{subequations}
\label{GNBC}
\begin{alignat}{2}
    \label{GNBC:1}
    &\v \cdot \n=0, \quad 2\nu(\psi) \left( \D \v \cdot \n\right)_\tau+ \gamma \v_\tau= L(\psi)\nabla_\Ga \psi 
    &&\quad\text{on $\Sigma$}, 
    \\
    \label{GNBC:2}
    &\partial_t \psi + \v_\tau \cdot \nabla_\Ga \psi=- \beta  L(\psi),\quad 
    &&\quad\text{on $\Sigma$},
    \\
    \label{GNBC:3}
    &\partial_\n \mu=0 
    &&\quad\text{on $\Sigma$},
\end{alignat}
\end{subequations}
where $\psi = \phi\vert_\Sigma$ as in \eqref{BC:TRACE}, and
$$
L(\psi)= - \kappa \Delta_\Ga \psi
+\varepsilon \partial_\n \psi + G'(\psi). 
$$
Here, the subscript $\tau$ denotes the tangential component of a vector $\w$, i.e., $\w_\tau=\w-(\w\cdot \n)\n$, $\nabla_\Ga$ is the tangential gradient on $\Gamma$ and $\Lapg$ denotes the Laplace--Beltrami operator on $\Gamma$ (see, e.g., \cite{dziuk2013finite} for differential operators on surfaces). The phenomenological parameters $\beta$, $\kappa$ and $\gamma$ are positive, and $G$ represents a surface double-well potential. 
The surface dynamic of the concentration $\psi$ is described by \eqref{GNBC:2} that is a convective Allen--Cahn type equation on the surface. It can be understood as an adaptation of the dynamic boundary condition \eqref{DBC:AC} (with $\delta=1)$ to the situation of an additional volume-averaged velocity field. Notice that the boundary condition \eqref{GNBC:3} (together with either \eqref{ModelH} or \eqref{AGG}) entails
that $\int_\Omega \phi(t) \dx$ is invariant (as for \eqref{classic-BC:mu}), whereas $\int_\Ga \psi(t) \dG$ is not a conserved quantity. 
The model H \eqref{ModelH} endowed with \eqref{GNBC} was first studied in \cite{GGM2016}. The authors proved the existence of global weak solutions assuming that both $F$ and $G$ are polynomial-like functions or $F$ is a singular potential and $G$ is polynomial function. In the former case, they also showed the convergence of any weak solution to a stationary state.  
Later on, the global existence of weak solutions for \eqref{AGG} endowed with \eqref{GNBC} has been achieved in \cite{GGW2019}. More recently, the compressible Navier--Stokes--Cahn--Hilliard system supplemented with \eqref{GNBC} has been studied in \cite{cherfils2019compressible}, and the Navier--Stokes--Allen--Cahn and the Navier--Stokes--Voigt--Allen--Cahn systems with \eqref{GNBC} have been analyzed in \cite{GGP2021}. For the stochastic model H with dynamic boundary conditions \eqref{GNBC}, we mention the recent results in \cite{Qiu-Wang}. It is worth pointing out that none of the aforementioned works established any uniqueness result for weak solutions in either two or three dimensions. In fact, the only uniqueness result established so far concerns {\it quasi-strong} solutions for the Navier--Stokes--Voigt--Allen--Cahn systems with \eqref{GNBC} in \cite{GGP2021}. 

In this paper, we derive and study a new class of thermodynamically consistent Navier--Stokes--Cahn--Hilliard systems with dynamic boundary conditions. Our model derivation is based on local mass balance laws in the bulk and on the surface in which the mass fluxes are still to be determined. Arguing similarly as in \cite{AGG}, after considering local energy dissipation laws, we apply the Lagrange multiplier approach to complete our model derivation by identifying the unknown mass fluxes on the macroscopic level.
The resulting Navier--Stokes--Cahn--Hilliard (NSCH) model reads as follows:
\begin{subequations}
\label{NSCH}
\begin{alignat}{2}
    \label{NSCH:1}
    &\delt \big(\rho(\phi) \v\big) 
    + \Div\big( \v \otimes (\rho(\phi)\v + \J) \big)
    - \Div\big( 2\nu(\phi) \, \D\v \big)
    + \Grad p
    = - \eps\, \Div\big(\Grad\phi \otimes \Grad\phi \big)
    &&\;\;\text{in $Q$}, 
    \\[1ex]
    \label{NSCH:2}
    &\Div \, \v = 0,
    &&\;\;\text{in $Q$}, 
    \\[1ex]
    \label{NSCH:3}
    &\revised{\rho(\phi)= \trho_1 \frac{1-\phi}{2}+ \trho_2 \frac{1+\phi}{2},\quad 
    \J = - \tfrac 12 (\trho_2 - \trho_1) 
        \mom(\phi) \Grad\mu}
    &&\;\;\text{in $Q$}, 
    \\[1ex]
    \label{NSCH:4}
    &\delt \phi + \Div(\phi \v )
    = \Div\big( \mom(\phi) \Grad \mu \big)
    &&\;\;\text{in $Q$}, 
    \\[1ex]
    \label{NSCH:5}
    &\mu = -\eps \Lap \phi + \eps^{-1} F'(\phi)
    &&\;\;\text{in $Q$}, 
    \\[1ex]
    \label{NSCH:6}
    &\delt \psi + \Divg(\psi \v_\tau ) 
    = \Divg\big( \mga(\psi) \Gradg \theta \big) - \beta \mom(\psi)\, \deln \mu 
    &&\;\;\text{on $\Sigma$}, 
    \\[1ex]
    \label{NSCH:7}
    &\theta = - \delta\kappa \Lapg \psi + \delta^{-1} G'(\psi)
    + \eps \deln \phi
    &&\;\;\text{on $\Sigma$}, 
    \\[1ex]
    \label{NSCH:8}
    &\v\cdot \n = 0,
    \qquad
    \phi\vert_\Gamma = \psi, 
    \qquad
    \begin{cases}
        L\mom(\psi) \, \deln \mu = \beta \theta - \mu 
        &\text{if $L\in[0,\infty)$},\\
        \mom(\psi)\, \deln \mu = 0 
        &\text{if $L = \infty$}
    \end{cases}
    &&\;\;\text{on $\Sigma$},
    \\[1ex]
    \label{NSCH:9}
    & \big[2\nu(\psi) (\D\v\,\n) + \gamma(\psi) \v \big]_\tau
    = \big[ -\psi \Gradg \theta + \tfrac 12 \, (\J\cdot\n) \, \v\big]_\tau
    &&\;\;\text{on $\Sigma$}, 
    \\[1ex]
    \label{NSCH:10}
    &\v\vert_{t=0} = \v_0,\quad \phi\vert_{t=0} = \phi_0
    &&\;\;\text{in $\Omega$},
    \\[1ex]
    \label{NSCH:11}
    &\psi\vert_{t=0} = \psi_0
    &&\;\;\text{on $\Gamma$},
\end{alignat}
\end{subequations}
In this model, 
 $\mom$ and $\mga$ are nonnegative functions respresenting the mobilities. The positive parameters $\eps$ and $\delta$ are related to the thickness of the diffuse interfaces in the bulk and on the boundary, respectively. The constant $\kappa \ge 0$ acts as a weight for surface diffusion effects. 
Moreover, $\beta\in\R$ satisfies $\beta\abs{\Omega} + \abs{\Gamma} \neq 0$, and $L \in [0,\infty]$.
In contrast to the dynamic boundary conditions \eqref{GNBC} which are of Allen--Cahn type, the dynamic boundary conditions in system \eqref{NSCH} are of Cahn--Hilliard type. More precisely, they can be understood as a convective variant of the dynamic boundary conditions $\big(\eqref{BC:TRACE}, \eqref{DBC:CH}\big)$ presented above. The sharp interface limit corresponding to the system \eqref{NSCH} is to be studied in a subsequent work. 

The system \eqref{NSCH} is associated with the total energy
\begin{align}
    \label{DEF:E}
    E(\v,\phi,\psi) 
    &:= E_\text{kin}(\phi,\v) + E_\text{free}(\phi,\psi) 
\end{align}
where $E_\text{kin}$ is the kinetic energy introduced in \eqref{DEF:EKIN+EBULK:AGG} and $E_\text{free} = E_\text{bulk} + E_\text{surf}$ is the free energy defined in \eqref{DEF:EFREE}.
Sufficiently regular solutions of system \eqref{NSCH} satisfy the \textit{mass conservation law}
\begin{align}
    \label{CONS:MASS}
    \left\{
    \begin{aligned}
    &\beta\intO \phi(t) \dx + \intG \psi(t) \dS 
    = \beta\intO \phi_0 \dx + \intG \psi_0 \dS
    &&\quad\text{if $L\in[0,\infty)$,}
    \\
    &\intO \phi(t) \dx = \intO \phi_0 \dx 
    \quad\text{and}\quad
    \intG \psi(t) \dS = \intG \psi_0 \dS
    &&\quad\text{if $L=\infty$,}
    \end{aligned}
    \right.
\end{align}
and the \textit{energy dissipation law}
\begin{align}
    \label{CONS:ENERGY}
    \begin{aligned}
    \ddt E(\v,\phi,\psi)
    &= - \intO 2\nu(\phi)\, \abs{\D\v}^2 \dx
    - \intG \gamma(\psi) \abs{\v}^2 \dS
    - \intO \mom(\phi) \abs{\Grad \mu}^2 \dx \\
    &\qquad - \intG \mga(\psi) \abs{\Gradg \theta}^2 \dS
    - \h(L) \intG (\beta\theta - \mu)^2 \dS,
    \end{aligned}
\end{align}
where
\begin{align}
\label{DEF:h:INTRO}
    \h(L) = 
    \begin{cases}
        0 &\text{if $L=0$ or $L=\infty$},\\
        L^{-1} &\text{if $L\in(0,\infty)$},
    \end{cases}
\end{align}
as long as the solution exists.

After presenting the model derivation of the system \eqref{NSCH} in Section~\ref{ModelDerivation}, we will study the case of fluids having matched densities, namely $\rho \equiv \trho_1=\trho_2$. This means that the flux $\J$ vanishes. We rewrite the system \eqref{NSCH} as follows
\begin{subequations}
\label{NSCH:S}
\begin{alignat}{2}
    \label{NSCH:S:1}
    & \rho \, \delt \v
    + \rho \, \Div\big( \v \otimes \v \big)
    - \Div\big( 2\nu(\phi) \, \D\v \big)
    + \Grad p
    =- \eps\, \Div\big(\Grad\phi \otimes \Grad\phi \big)
    &&\quad\text{in $Q$}, 
    \\[1ex]
    \label{NSCH:S:2}
    &\Div \, \v = 0
    &&\quad\text{in $Q$}, 
    \\[1ex]
    \label{NSCH:S:3}
    &\delt \phi + \Div(\phi \v )
    = \Div\big( \mom(\phi) \Grad \mu \big)
    &&\quad\text{in $Q$}, 
    \\[1ex]
    \label{NSCH:S:4}
    &\mu = -\eps \Lap \phi + \eps^{-1} F'(\phi)
    &&\quad\text{in $Q$}, 
    \\[1ex]
    \label{NSCH:S:5}
    &\delt \psi + \Divg(\psi \v_\tau ) 
    = \Divg\big( \mga(\psi) \Gradg \theta \big) - \beta \mom(\psi)\, \deln \mu 
    &&\quad\text{on $\Sigma$}, 
    \\[1ex]
    \label{NSCH:S:6}
    &\theta = - \delta\kappa \Lapg \psi + \delta^{-1} G'(\psi)
    + \eps \deln \phi
    &&\quad\text{on $\Sigma$}, 
    \\[1ex]
    \label{NSCH:S:8}
    &\v\cdot \n = 0,
    \qquad
    \phi\vert_\Gamma = \psi, 
    \qquad
    \begin{cases}
        L\mom(\psi) \, \deln \mu = \beta \theta - \mu 
        &\text{if $L\in[0,\infty)$},\\
        \mom(\psi)\, \deln \mu = 0 
        &\text{if $L = \infty$}
    \end{cases}
    &&\quad\text{on $\Sigma$},
    \\[1ex]
    \label{NSCH:S:9}
    & \big[2\nu(\psi) (\D\v\,\n) + \gamma(\psi) \v \big]_\tau
    = \big[ -\psi \Gradg \theta \big]_\tau
    &&\quad\text{on $\Sigma$}, 
    \\[1ex]
    \label{NSCH:S:10}
    &\v\vert_{t=0} = \v_0,\quad \phi\vert_{t=0} = \phi_0
    &&\quad\text{in $\Omega$},
    \\[1ex]
    \label{NSCH:S:11}
    &\psi\vert_{t=0} = \psi_0
    &&\quad\text{on $\Gamma$}.
\end{alignat}
\end{subequations}


The model \eqref{NSCH:S} can be regarded as a generalization of the model H with dynamic boundary conditions of Cahn--Hilliard type.
We will focus on the mathematical analysis of system \eqref{NSCH:S} in the case $L=0$ and $\beta>0$ (see system \eqref{NSCH:MD} in Section~\ref{MathAnalysis}). In Theorem \ref{THM:WS}, we first prove the existence of global weak solutions to system \eqref{NSCH:MD} with regular (polynomial-like) potentials $F$ and $G$ in bounded Lipschitz domains (cf. \eqref{ass:const}--\eqref{ass:vis-fric} below for the specific assumptions). Weak solutions satisfy the {\it variational formulation} of the problem \eqref{NSCH:MD} and a {\it weak energy dissipation law} (see \eqref{WS:ENERGY}). A weak solution is obtained as the limit of smooth solutions to suitably regularized problems based on a Faedo--Galerkin scheme, where appropriate compactness properties are derived on the basis of uniform and global estimates gained from the  energy dissipation law \eqref{CONS:ENERGY}. 
However, an ad hoc construction of the solution ansatz is required since both $(\phi,\psi)$ and $(\mu,\theta)$ need to be approximated in the same finite-dimensional subspace in order to show the validity of \eqref{CONS:ENERGY} for such regularized solutions (cf. the change of variables \eqref{CofV}). In comparison with the works in literature \cite{GGM2016, GGW2019, GGP2021}, our argument is rather simple and relies only on a one-level Faedo--Galerkin approximation. Furthermore, it is shown that $(\mu,\theta)\in L^4(0,T; \LL^2)$, which is a novel regularity for NSCH systems (and even CH equations) with regular potentials (cf., for instance, \cite{boyer1999, GalGrasselli2010}).
Next, in Theorem \ref{THM:UNI}, we then establish the uniqueness of the weak solutions in the two dimensional case. This is the first uniqueness result in the literature for a NSCH system with dynamic boundary conditions. It is important to underline that the nonlinear coupling in the NSCH system with dynamic boundary conditions \eqref{NSCH:S} is much stronger than in the model H \eqref{ModelH}. Roughly speaking, this is due to the elastic (surface) term $[\psi \Grad \theta]_\tau$ in \eqref{NSCH:S:9},
in addition to the classical capillarity (bulk) term $-\varepsilon \Div (\nabla \phi \otimes \nabla \phi)$ (also referred to as \textit{Korteweg force}).
More precisely, gradient terms arising from equivalent formulations of the \textit{Korteweg force} (cf. \eqref{Korteweg}) disappear in bulk integrals in view of the incompressibility constraint $\Div \, \v=0$.
Instead, surface integrals involving the product between a gradient function and $\v$ does not vanish since $\Div_\tau \, \v \neq 0$ on $\Sigma$. For this reason, 
the term $[\psi \Grad \theta]_\tau$ cannot be handled as in \cite{GMT2019} through dual estimates. 
Therefore, the appropriate functional to control the difference of two solutions corresponds to the total energy. This functional is capable of reproducing the same cancellation of the nonlinear terms as in the derivation of \eqref{CONS:ENERGY}.
In order to rigorously justify this fact, two additional assumptions need to be imposed: the domain $\Omega$ has a $C^3$-boundary, which is needed to gain further spatial regularity of $(\phi,\psi)$ and to recover the relations of the bulk/surface chemical potentials $(\mu,\theta)$ almost everywhere; the relation between bulk and surface potentials $F(s)=\beta G(s)$ for all $s \in \mathbb{R}$, which enables us to apply a particular chain formula (see Appendix \ref{A}).

\section{Model derivation}\label{ModelDerivation}

In this section we will derive the novel Navier--Stokes--Cahn--Hilliard system with dynamic boundary conditions \eqref{NSCH}.

\subsection{Considerations based on local mass balance laws}
We consider the time evolution of two fluids (indexed with $i=1,2$) in a (sufficiently smooth) domain $\Omega\subset\R^d$ with $d\in\{2,3\}$ on a time interval $[0,T]$ with $T>0$. We write $Q=\Omega\times(0,T)$ and $\Sigma=\Gamma\times(0,T)$, where $\Gamma=\partial\Omega$.
Let $\rho_i:Q\cup\Sigma\to\R$, $i=1,2$, denote the mass densities of these two fluids. 
The local mass balance equations in the bulk read as
\begin{align}
    \label{BL:MASS:B}
    \delt \rho_i + \Div \, \widehat\J_i = 0,
    \quad i=1,2,
    \quad\text{in $Q$.}
\end{align}
Here, $\widehat\J_i$ are the mass fluxes. As we want to allow for a transfer of material between bulk and surface, the local mass balance equation on the boundary is given by
\begin{align}
    \label{BL:MASS:S}
    \delt \rho_i + \Divg \, \widehat\K_i = \beta \widehat\J_i\cdot \n,
    \quad i=1,2,
    \quad\text{on $\Sigma$.}
\end{align}
In this relation, $\widehat\K_i$ stand for the mass fluxes on the boundary. The right-hand side $\beta \widehat\J_i$ describes the transfer of mass between bulk and surface. Here, the constant $\beta\in\R$ acts as a weight of this mass transfer.
Introducing the individual velocities $\v_i:Q\cup\Sigma\to\R^d$, $i=1,2$, which can be expressed as
\begin{align*}
    \v_i = \frac{\widehat\J_i}{\rho_i} \quad\text{in $Q$},
    \qquad
    \v_i = \frac{\widehat\K_i}{\rho_i} \quad\text{on $\Sigma$},
    \quad
    i=1,2,
\end{align*}
the mass balances \eqref{BL:MASS:B} and \eqref{BL:MASS:S} can be rewritten as
\begin{alignat}{2}
    \label{BL:MASS:B'}
    \delt \rho_i + \Div(\rho_i\v_i) &= 0,
    &&\quad i=1,2,
    \quad\text{in $Q$,}
    \\
    \label{BL:MASS:S'}
    \delt \rho_i + \Divg(\rho_i\v_i) &= \beta \widehat\J_i \cdot \n,
    &&\quad i=1,2,
    \quad\text{on $\Sigma$.}
\end{alignat}
Let now $\phi_i:Q\to\R$ and $\psi_i:\Sigma\to\R$, $i=1,2$, denote the volume fractions of the two fluids in the bulk and on the boundary, respectively. 
Provided that each fluid has a constant density $\trho_i$, $i=1,2$, the volume fractions can be identified as
\begin{align}
\label{DEF:PHIPSI*}
    \phi_i = \frac{\rho_i}{\trho_i} \quad\text{in $Q$,}
    \qquad
    \psi_i = \frac{\rho_i}{\trho_i} \quad\text{on $\Sigma$,}
    \quad
    i=1,2.
\end{align}
In particular, this naturally entails the trace relation 
\begin{align}
    \label{BC:PHIPSI}
    \phi\vert_\Sigma = \psi \quad\text{on $\Sigma$.}
\end{align}
Under the assumption that the excess volume is zero, we have
\begin{align}
    \label{CONS:PHIPSI}
    \phi_1 + \phi_2 = 1 \quad\text{in $Q$,}
    \qquad
    \psi_1 + \psi_2 = 1 \quad\text{on $\Sigma$.}
\end{align}
We further set
\begin{align}
\label{DEF:PHIPSI}
    \phi = \phi_2 - \phi_1 \quad\text{in $Q$,}
    \qquad
    \psi = \psi_2 - \psi_1 \quad\text{on $\Sigma$.}
\end{align}
Furthermore, let $\v:Q\cup\Sigma\to\R^d$ denote the volume averaged velocity field associated with the two fluids.
It can be expressed as
\begin{alignat}{2}
    \label{DEF:VAVG:B}
    \v &= \phi_1 \v_1 + \phi_2 \v_2,
    &&\quad i=1,2,
    \quad\text{in $Q$,}
    \\
    \label{DEF:VAVG:S}
    \v &= \psi_1 \v_1 + \psi_2 \v_2,
    &&\quad i=1,2,
    \quad\text{on $\Sigma$.}
\end{alignat}
We assume that both fluids cannot permeate the boundary $\Gamma$, which leads to the condition
\begin{align}
    \label{BC:NP}
    \v_\n = \v\cdot\n = 0
    \quad\text{on $\Sigma$.}
\end{align}
Hence, $\v$ is identical to its tangential component, that is, $\v_\tau = \v - (\v\cdot\n)\n = \v$ on $\Sigma$. We further write
\begin{alignat}{2}
    \J_i &= \widehat\J_i - \rho_i \v,
    &&\quad i=1,2,
    \quad\text{in $Q$,}
    \\
    \K_i &= \widehat\K_i - \rho_i \v,
    &&\quad i=1,2,
    \quad\text{on $\Sigma$,}
\end{alignat}
to denote the mass fluxes relative to the volume averaged velocity. The mass balance equations \eqref{BL:MASS:B'} and \eqref{BL:MASS:S'} can thus be expressed as
\begin{alignat}{2}
    \label{BL:MASS:B*}
    \delt \rho_i + \Div(\rho_i\v) + \Div \, \J_i &= 0,
    &&\quad i=1,2,
    \quad\text{in $Q$,}
    \\
    \label{BL:MASS:S*}
    \delt \rho_i + \Divg(\rho_i\v) + \Divg \, \K_i  &= \beta \J_i \cdot \n,
    &&\quad i=1,2,
    \quad\text{on $\Sigma$.}
\end{alignat}
In this context, $\J_i$ and $\K_i$, $i=1,2$, can be regarded as diffusive flow rates. We now define the total mass density as
\begin{alignat}{2}
    \label{DEF:RHO}
    \rho &= \rho_1 + \rho_2,
    &&\quad i=1,2,
    \quad\text{in $Q\cup\Sigma$.}
\end{alignat}
This leads to the relations
\begin{alignat}{2}
    \label{BL:MASS:B**}
    \delt \rho + \Div(\rho\v) + \Div(\J_1+\J_2)&= 0
    &&\quad\text{in $Q$,}
    \\
    \label{BL:MASS:S**}
    \delt \rho + \Divg(\rho\v) + \Divg(\K_1+\K_2) &= \beta (\J_1 + \J_2) \cdot \n
    &&\quad\text{on $\Sigma$.}
\end{alignat}
Next, assuming a conservation of linear momentum with respect to the velocity field $\v$, we have
\begin{align}
    \label{BL:CLM}
    \delt (\rho\v) + \Div(\rho\v\otimes\v) = \Div \, \tT 
    \quad\text{in $Q$.}
\end{align}
Here, $\tT$ denotes the stress tensor that needs to be specified by constitutive assumptions. In the following, we write 
\begin{align}
\label{DEF:JK}
    \J = \J_1 + \J_2 \quad\text{in $Q$,}
    \qquad
    \K = \K_1 + \K_2 \quad\text{on $\Sigma$.}
\end{align}
Using \eqref{BL:MASS:B**} to rewrite \eqref{BL:CLM}, we obtain
\begin{align}
\label{BL:CLM*}
\begin{aligned}
    \rho(\delt\v + \v\cdot\Grad\v) 
    &= \Div \, \tT + ( \Div \, \J )\v
    \\
    &= \Div(\tT + \v\otimes \J) - \Grad\v \, \J
    \quad\text{in $Q$.}
    \end{aligned}
\end{align}
This allows us to define the objective tensor
\begin{align}
    \label{DEF:TT}
    \T = \tT + \v \otimes \J
    \quad\text{in $Q$,}
\end{align}
which is frame indifferent. Hence, \eqref{BL:CLM*} can be rewritten as
\begin{align}
\label{BL:CLM**}
    \rho \delt\v + \Grad\v \, (\rho\v + \J) 
    = \Div\, \T 
    \quad\text{in $Q$.}
\end{align}
We now multiply \eqref{BL:MASS:B'} by ${1}/{\trho_i}$, $i=1,2$, and we add the resulting equations. Recalling \eqref{CONS:PHIPSI} and \eqref{DEF:VAVG:B}, this yields
\begin{align}
    \label{PDE:DIVV}
    \begin{aligned}
    \Div \, \v &= \delt(\phi_1 + \phi_2) + \Div \, \v  
    \\
    &= \delt \left(\frac{\rho_1}{\trho_1} + \frac{\rho_2}{\trho_2}\right)
        + \Div \left(\frac{\rho_1}{\trho_1}\v_1 + \frac{\rho_2}{\trho_2}\v_2 \right)
    =0
    &&\quad\text{in $Q$,}
    \end{aligned}
\end{align}
which means that the volume averaged velocity field is divergence-free. 
Multiplying \eqref{BL:MASS:S'} by ${1}/{\trho_i}$, $i=1,2$, and adding the resulting equations, we infer
\begin{align}
    \label{PDE:DIVV:S}
    \begin{aligned}
    \Divg \, \v &= \delt(\psi_1 + \psi_2) + \Divg \, \v 
    \\
    &= \delt \left(\frac{\rho_1}{\trho_1} + \frac{\rho_2}{\trho_2}\right)
        + \Divg \left(\frac{\rho_1}{\trho_1}\v_1 + \frac{\rho_2}{\trho_2}\v_2 \right)
    \\
    &= \beta\left(\frac{\J_1}{\trho_1} + \frac{\J_2}{\trho_2}\right)\cdot\n 
    &&\quad\text{on $\Sigma$.}
    \end{aligned}
\end{align}
We now define 
\begin{alignat}{4}
    \label{DEF:TJ}
    \tJ_i &= \frac{\J_i}{\trho_i}, 
    &&\quad i=1,2,
    &\qquad \J_\phi &= \tJ_2-\tJ_1
    &&\quad\text{in $Q$,}
    \\[1ex]
    \label{DEF:TK}
    \tK_i &= \frac{\K_i}{\trho_i},
    &&\quad i=1,2,
    &\qquad \K_\psi &= \tK_2-\tK_1
    &&\quad\text{on $\Sigma$.}
\end{alignat}
We next multiply \eqref{BL:MASS:B*} by ${1}/{\trho_i}$, $i=1,2$. Subtracting the resulting equations, we deduce that
\begin{align}
    \delt \phi + \Div(\phi\v) + \Div \, \J_\phi = 0
    \quad\text{in $Q$.}
\end{align}
Moreover, multiplying \eqref{BL:MASS:B*} by ${1}/{\trho_i}$, adding the resulting equations, and recalling \eqref{CONS:PHIPSI} and \eqref{PDE:DIVV}, we conclude that
\begin{align}
    \label{ID:DIVJ}
    \Div(\tJ_1 + \tJ_2) 
    = \delt(\phi_1+\phi_2) + \Div\big((\phi_1+\phi_2)\v \big) + \Div(\tJ_1 + \tJ_2) 
    = 0
    \quad\text{in $Q$.}
\end{align}
Proceeding similarly with \eqref{BL:MASS:S*}, and recalling \eqref{CONS:PHIPSI} and \eqref{PDE:DIVV:S}, we obtain
\begin{align}
    \delt \psi + \Divg(\psi\v) + \Divg \, \K_\psi = \beta \J_\phi \cdot \n
    \quad\text{on $\Sigma$,}
\end{align}
and
\begin{align}
    \label{ID:DIVK}
    \begin{aligned}
    \Divg(\tK_1 + \tK_2) 
    &= \delt(\psi_1+\psi_2) + \Divg\big((\psi_1+\psi_2)\v \big)
    \\
    &\quad + \Divg(\tK_1 + \tK_2) - \beta (\tJ_1+\tJ_2)\cdot \n
    = 0
    \quad\text{on $\Sigma$.}
    \end{aligned}
\end{align}
Using \eqref{DEF:TJ} and \eqref{ID:DIVJ}, we infer that
\begin{align}
    \label{ID:DIVJ**}
    \begin{aligned}
    \Div \, \J 
    &= \frac 12 \bigg[ \trho_2 \Div \, \tJ_2 
            + \trho_1 \Div \, \tJ_1 \bigg]
        + \frac 12\bigg[ \trho_2 \Div \, \tJ_2 
            + \trho_1 \Div \, \tJ_1 \bigg]
    \\[1ex]
    &= \frac{\trho_2-\trho_1}{2} \Div \, \tJ_2
        - \frac{\trho_2-\trho_1}{2} \Div \, \tJ_1
    = \frac{\trho_2-\trho_1}{2} \Div \, \J_\phi
    &&\quad\text{in $Q$.}
    \end{aligned}
\end{align}
This means that $\J$ and $\tfrac 12 (\trho_2-\trho_1) \J_\phi $ can merely differ by an additive divergence-free function. As in \cite{AGG}, we thus assume that
\begin{align}
    \label{ID:J}
    \J = \frac{\trho_2-\trho_1}{2} \J_\phi 
    \quad\text{in $Q$.}
\end{align}
In the same fashion, we derive the identity
\begin{align}
    \Divg \, \K = \frac{\trho_2-\trho_1}{2} \Divg \, \K_\psi
    \quad\text{on $\Sigma$.}
\end{align}
From \eqref{CONS:PHIPSI}, \eqref{DEF:PHIPSI} and \eqref{DEF:RHO}, assuming that $\phi$ and $\psi$ only attain values in $[-1,1]$, we further infer that the density $\rho = \rho(\phi)$ is given as
\begin{align}
    \label{DEF:RHO*}
    \rho(\phi) = \frac{\trho_2-\trho_1}{2} \phi + \frac{\trho_2+\trho_1}{2}
    \quad\text{in $Q$,}
\end{align}
and due to \eqref{BC:PHIPSI}, the trace $\rho(\phi)\vert_\Sigma$ can be expressed as
\begin{align}
    \label{DEF:RHO:S*}
    \rho(\phi)\vert_\Sigma =
    \rho(\psi) = \frac{\trho_2-\trho_1}{2} \psi + \frac{\trho_2+\trho_1}{2}
    \quad\text{on $\Sigma$.}
\end{align}
Plugging \eqref{DEF:RHO*} into \eqref{BL:MASS:B**}, and \eqref{DEF:RHO:S*} into \eqref{BL:MASS:S**}, and using the identities \eqref{ID:DIVJ} and \eqref{ID:DIVK}, we eventually derive the equations
\begin{alignat}{2}
    \label{EQ:PHI}
    \delt \phi + \Div(\phi\v) + \Div \, \J_\phi&= 0
    &&\quad i=1,2,
    \quad\text{in $Q$,}
    \\
    \label{EQ:PSI}
    \delt \psi + \Divg(\psi\v) + \Divg \, \K_\psi &= \beta \J_\phi \cdot \n
    &&\quad i=1,2,
    \quad\text{on $\Sigma$.}
\end{alignat}

\subsection{Local energy dissipation laws}

Proceeding as in \cite{AGG}, we introduce the following energy density in the bulk:
\begin{align}
    \label{DEF:END:B}
    e_\Om(\v,\phi,\Grad\phi) 
    = \frac{\rho(\phi)}{2}\abs{\v}^2
        + f(\phi,\Grad\phi)
    \quad\text{in $Q$.}
\end{align}
The first summand on the right-hand side stands for the kinetic energy density whereas the second summand denotes the free energy density in the bulk.
On the boundary, we introduce the additional energy density
\begin{align}
    \label{DEF:END:S}
    e_\Ga(\psi,\Gradg\psi) 
    = g(\psi,\Gradg\psi)
    \quad\text{on $\Sigma$,}
\end{align}
where $g$ stands for the free energy density on the surface. 

We now consider an arbitrary test volume $V(t)\subset \Omega$,
$t\in[0,T]$, that is transported by the flow associated with $\v$. Let $\nu$ denote the outer unit normal vector field of $V(t)$.
In an isothermal situation, the second law of thermodynamics leads to the dissipation inequality
\begin{align}
    \label{IEQ:DISS}
    \begin{aligned}
    0 &\ge \ddt \left[ \int_{V(t)} e_\Om(\v,\phi,\Grad\phi) \dx
        + \int_{\del V(t)\cap \Ga} e_\Ga(\psi,\Gradg\psi) \dS \right]
    \\
    &\qquad + \int_{\del V(t)\cap \Omega} \J_e \cdot \nu \dS
    + \int_{\del_\Ga (\del V(t)\cap \Gamma)} \K_e \cdot \nu_\Ga \dS_\Ga.
    \end{aligned}
\end{align}
Here $\J_e$ and $\K_e$ are energy fluxes that will be specified later, \revised{and we assume that $\K_e\cdot \n = 0$.} As we consider a closed system, there is no transfer of energy over the boundary $\Ga$ and thus, the domain of the first integral in the second line is just $\del V(t)\cap \Omega$ instead of $\del V(t)$. We further point out that $\del_\Ga (\del V(t)\cap \Gamma) \subseteq \Ga$ is to be understood as the relative boundary of the set $\del V(t)\cap \Gamma$ within the submanifold $\Ga$, and $\nu_\Ga$ stands for the corresponding outer unit normal vector field. Applying Gau\ss's divergence theorem on both integrals in the second line, and recalling that $\nu=\n$ on $\del V(t)\cap \Gamma$, we reformulate \eqref{IEQ:DISS} as
\begin{align}
    \label{IEQ:DISS*}
    \begin{aligned}
    0 &\ge \ddt \left[ \int_{V(t)} e_\Om(\v,\phi,\Grad\phi) \dx
        + \int_{\del V(t)\cap \Ga} e_\Ga(\psi,\Gradg\psi) \dS \right]
    \\
    &\qquad + \int_{V(t)} \Div \, \J_e \dx
    + \int_{\del V(t)\cap \Gamma} \Divg \, \K_e - \J_e\cdot\n \dS.
    \end{aligned}
\end{align}
Applying the Reynolds transport theorem, we find that
\begin{align}
    \label{TRANSP:B}
    \ddt \int_{V(t)} e_\Om(\v,\phi,\Grad\phi) \dx
    = \int_{V(t)} \delt e_\Om(\v,\phi,\Grad\phi) 
        + \Div\big(e_\Om(\v,\phi,\Grad\phi) \v \big) \dx.
\end{align}
Similarly, using the transport theorem for evolving hypersurfaces (see, e.g., \cite[Theorem~32]{BGN-book}), we obtain
\begin{align}
    \label{TRANSP:S}
    \ddt \int_{\del V(t)\cap \Ga} e_\Ga(\psi,\Gradg\psi) \dS
    = \int_{V(t)\cap \Ga} \delt e_\Ga(\psi,\Gradg\psi) 
        + \Divg\big(e_\Ga(\psi,\Gradg\psi) \v \big) \dS.
\end{align}
Combining \eqref{IEQ:DISS*}, \eqref{TRANSP:B} and \eqref{TRANSP:S}, we thus get
\begin{align}
    \label{IEQ:DISS**}
    \begin{aligned}
    0 &\ge \int_{V(t)} 
        \delt e_\Om(\v,\phi,\Grad\phi) 
        + \Div\big(e_\Om(\v,\phi,\Grad\phi) \v \big)
        + \Div \, \J_e \dx
    \\
    &\quad+ \int_{\del V(t)\cap \Gamma} 
        \delt e_\Ga(\psi,\Gradg\psi) 
        + \Divg\big(e_\Ga(\psi,\Gradg\psi) \v \big)
        +\Divg \, \K_e - \J_e\cdot\n \dS.
        \end{aligned}
\end{align}
In particular, the above inequality holds true for all test volumes $V(t)\subset \Om$ with $\del V(t) \cap \Ga = \emptyset$. We thus infer the local dissipation law in the bulk, which reads as
\begin{align}
    \label{IEQ:DISS:LOC:B}
    0 \ge -\mathcal D_\Om 
    := \delt e_\Om(\v,\phi,\Grad\phi) 
        + \Div\big(e_\Om(\v,\phi,\Grad\phi) \v \big)
        + \Div \, \J_e 
    \quad\text{in $Q$.}
\end{align}
Let now $\alpha>0$ be arbitrary and suppose that $V(t)\subset \Om$ is a test volume with $\abs{V(t)}$ being sufficiently small such that 
\begin{align*}
    \int_{V(t)} \mathcal D_\Om \dx < \alpha.
\end{align*}
Then, we use \eqref{IEQ:DISS**} to infer that
\begin{align*}
    \int_{\del V(t)\cap \Gamma} 
        \delt e_\Ga(\psi,\Gradg\psi) 
        + \Divg\big(e_\Ga(\psi,\Gradg\psi) \v \big)
        +\Divg \, \K_e - \J_e\cdot\n \dS
    <\alpha.
\end{align*}
Since $\alpha>0$ and the test volume $V(t)$ were arbitrary (except for the above restriction on $V(t)$), we conclude the following local dissipation law on the boundary:
\begin{align}
    \label{IEQ:DISS:LOC:S}
    0 \ge -\mathcal D_\Ga
    := \delt e_\Ga(\psi,\Gradg\psi) 
        + \Divg\big(e_\Ga(\psi,\Gradg\psi) \v \big)
        +\Divg \, \K_e - \J_e\cdot\n 
    \quad\text{on $\Sigma$.}
\end{align}

\subsection{Completion of the model derivation by the Lagrange multiplier approach}

We now complete the model derivation by means of the \textit{Lagrange multiplier approach}. We introduce the functions $\mu$ and $\theta$ that will be fixed in the subsequent approach. In view of the identities \eqref{EQ:PHI} and \eqref{EQ:PSI}, the local energy dissipation laws \eqref{IEQ:DISS:LOC:B} and \eqref{IEQ:DISS:LOC:S} can be written as
\begin{alignat}{2}
    \label{IEQ:DISS:B:0}
    0 \ge -\mathcal D_\Om 
    &= \delt e_\Om(\v,\phi,\Grad\phi) 
        + \Div\big(e_\Om(\v,\phi,\Grad\phi) \v \big)
        + \Div \, \J_e 
    \notag\\
    &\quad - \mu \big( 
        \delt \phi 
        + \Div(\phi\v) 
        + \Div \, \J_\phi \big)
    &&\quad\text{in $Q$,}
    \\[1ex]
    \label{IEQ:DISS:S:0}
    0 \ge - \mathcal D_\Ga
    &= \delt e_\Ga(\psi,\Gradg\psi) 
        + \Divg\big(e_\Ga(\psi,\Gradg\psi) \v \big)
        +\Divg \, \K_e - \J_e\cdot\n
    \notag\\
    &\quad- \theta \big( 
        \delt \psi 
        + \Divg(\v \psi)
        + \Divg \, \K_\psi 
        - \beta \J_\phi\cdot\n 
        \big)
    &&\quad\text{on $\Sigma$.}
\end{alignat}
Here, the functions $\mu$ and $\theta$ can be understood as Lagrange multipliers.

In the following, for brevity, we will just write $\rho$, $f$ and $g$ instead of $\rho(\phi)$, $f(\phi,\Grad\phi)$ and $g(\psi,\Gradg\psi)$.
By the definition of the energy densities $e_\Om$ and $e_\Ga$ (see \eqref{DEF:END:B} and \eqref{DEF:END:S}), we rewrite \eqref{IEQ:DISS:B:0} and \eqref{IEQ:DISS:S:0} as
\begin{alignat}{2}
    \label{IEQ:DISS:B:1}
    0 \ge -\mathcal D_\Om 
    &= \delt \big(\tfrac 12 \rho \abs{\v}^2 \big)
        + \Div\big(\tfrac 12 \rho \abs{\v}^2 \v \big)
        + \delt f + \Div(f\v)
        + \Div \, \J_e
    \notag\\
    &\quad - \mu \big( 
        \delt \phi 
        + \Div(\phi\v) 
        + \Div \, \J_\phi \big)
    &&\quad\text{in $Q$,}
    \\[1ex]
    \label{EST:DISS:S:1}
    0 \ge - \mathcal D_\Ga
    &= \delt g 
        + \Divg(g \v)
        + \Divg \, \K_e 
        - \J_e \cdot \n
    \notag\\
    &\quad- \theta \big( 
        \delt \psi 
        + \Divg(\psi \v)
        + \Divg \, \K_\psi 
        - \beta \J_\phi\cdot\n 
        \big)
    &&\quad\text{on $\Sigma$.}
\end{alignat}
By the chain rule, the derivatives 
$\delt f$ and $\delt g$
can be expressed as
\begin{alignat}{2}
    \label{ID:DDTF}
    \delt f
    &= \delph f \, \delt \phi + \delgph f\, \delt \Grad\phi
    &&\quad\text{in $Q$,}
    \\
    \label{ID:DDTG}
    \delt g
    &= \delps g \, \delt \psi + \delgps g\, \delt \Gradg\psi
    &&\quad\text{on $\Sigma$.}
\end{alignat}
In the following, for any function $h:Q\to\R$, we use the notation
\begin{align*}
    D_t h = \delt h + (v\cdot \Grad) h
\end{align*}
to denote its material derivative.
In the bulk, we proceed exactly as in \cite{AGG} to reformulate \eqref{IEQ:DISS:B:1} as
\begin{align} 
    \label{EST:DISS:B:2}
    \begin{aligned}
    0 & \ge \Div\big[\J_e 
            - \tfrac 12 \abs{\v}^2 \J 
            + \T^\top\v 
            - \mu \J_\phi
            + \delgph f D_t \phi
        \big]
    \\
    & \quad + \big[ \delph f 
            - \Div(\delgph f) 
            - \mu 
        \big]
        D_t \phi    
    \\
    & \quad - \big[ \T 
            + \Grad\phi \otimes \delgph f 
            \big] : \Grad\v
        + \Grad\mu \cdot \J_\phi
    &&\quad\text{in $Q$.}
    \end{aligned}
\end{align}
To ensure \eqref{EST:DISS:B:2}, we now choose the chemical potential $\mu$ and the energy flux $\J_e$ as
\begin{alignat}{2}
    \label{DEF:MU*}
    \mu &= \delph f 
        - \Div\big( \delgph f \big)
    &&\quad\text{in $Q$,}
    \\
    \label{DEF:JE}
    \J_e &= \tfrac 12 \abs{\v}^2 \J 
            - \T^\top\v 
            + \mu \J_\phi
            - \delgph f D_t \phi
    &&\quad\text{in $Q$.}
\end{alignat}
By these choices, the first two lines of the right-hand side in \eqref{EST:DISS:B:2} vanish.
We further assume the mass flux $\J_\phi$ to be of Fick's type, that is
\begin{align}
    \label{DEF:JPHI}
    \J_\phi = -\mom(\phi) \Grad\mu \quad\text{in $Q$,}
\end{align}
where $\mom = \mom(\phi)$ is a nonnegative function representing the mobility.
Hence, \eqref{EST:DISS:B:2} reduces to
\begin{align}
    \label{EST:DISS:B:3}
    0  \ge  - \big[ \T 
            + \Grad\phi \otimes \delgph f 
            \big] : \Grad\v
        - \mom(\phi) \abs{\Grad\mu}^2
    \quad\text{in $Q$.}
\end{align}
We now define the tensor
\begin{align*}
    \mathbf S = \T + p \mathbf I 
        + \Grad\phi \otimes \delgph f
    \quad\text{in $Q$.}
\end{align*}
Here, the variable $p$ denotes the pressure, and $\mathbf I$ denotes the identity matrix. $\mathbf S$ is the viscous stress tensor that corresponds to irreversible changes of the energy due to friction. For Newtonian fluids, $\mathbf S$ is usually chosen as
\begin{align}
    \mathbf S = 2 \nu(\phi) \D\v,
\end{align}
where $\nu=\nu(\phi)$ is a nonnegative function representing the viscosity of the fluids, and $\D\v$ is the symmetric gradient of $\v$. By this choice, \eqref{EST:DISS:B:3} is satisfied since we obtain 
\begin{align}
    \label{EST:DISS:B:4}
    0  \ge  - \nu(\phi) \abs{\D\v}^2
        - \mom(\phi) \abs{\Grad\mu}^2
    \quad\text{in $Q$,}
\end{align}
by means of the identity $p\mathbf I:\Grad \v = p\, \Div \, \v  = 0$ in $Q$.

We now assume that the energy density $f$ is of Ginzburg--Landau type, that is
\begin{align}
    \label{DEF:GL:B}
    f(\phi,\Grad\phi)
    = \frac{\eps}{2}\abs{\Grad\phi}^2 + \frac 1\eps F(\phi)
    \quad\text{in $Q$.}
\end{align}
Here, the parameter $\eps>0$ is related to the thickness of the diffuse interface that separates the two fluids, and $F$ is a potential that usually exhibits a double-well structure. Hence, the chemical potential $\mu$ reads as
\begin{align}
    \label{PDE:MU}
    \mu =  - \eps \Lap\phi + \frac 1\eps F'(\phi)
    \quad\text{in $Q$,}
\end{align}
and the total stress tensor $\T$ is given by
\begin{align}
    \label{DEF:T}
    \begin{aligned}
    \T &= \mathbf S - p\mathbf I - \Grad\phi \otimes \delgph f
    \\
    &= 2 \nu(\phi) \D\v 
    - p \mathbf I 
    - \eps \Grad\phi \otimes \Grad\phi
    \quad\text{in $Q$.}
    \end{aligned}
\end{align}
Plugging \eqref{DEF:JPHI} into \eqref{EQ:PHI}, we further obtain the equation
\begin{align}
    \label{PDE:PHI}
    \delt\phi + \Div(\phi\v) = \Div\big( \mom(\phi) \Grad\mu \big)
    \quad\text{in $Q$.}
\end{align}
Eventually, recalling \eqref{DEF:TT} and \eqref{DEF:T}, we use \eqref{BL:CLM} to derive the equation
\begin{align}
    \label{PDE:NS}
    \delt (\rho\v) 
    + \Div\big(\v\otimes(\rho(\phi)\v + \J )\big) 
    - \Div\big(2\nu(\phi)\D\v\big)
    + \Grad p
    = - \eps\Div(\Grad\phi\otimes\Grad\phi)
    \quad\text{in $Q$,}
\end{align}
where
\begin{align}
    \label{PDE:J}
    \J = - \frac{\trho_2-\trho_1}{2} \mom(\phi) \Grad\mu 
    \quad\text{in $Q$.}
\end{align}

We now consider the local energy dissipation law \eqref{EST:DISS:S:1} on the boundary.
Recalling the formulas for $\J_e$ (see \eqref{DEF:JE}) and $\T$ (see \eqref{DEF:T}), we infer that
\begin{align} \label{EST:DISS:S:2}
\begin{aligned}
    0 & \ge \delps g\, \delt \psi
        + \Divg\big( \delgps g \delt \psi \big)
        - \Divg\big(\delgps g\big) \delt \psi
        + \Divg(g\v) + \Divg \, \K_e
    \\
    &\quad - \tfrac 12 (\J \cdot \n) \, (\v\cdot\v) 
        + 2\nu(\phi) 
        \big(\D\v \,\n \big)\cdot \v
        - \eps (\Grad\phi\otimes \Grad\phi)\n \cdot \v
        - \mu \J_\phi\cdot \n
    \\
    &\quad  
        + (\delgph f \cdot \n)  \delt \psi    
        + (\delgph f \cdot \n)  (\v \cdot \Gradg\psi)
        - \theta \delt \psi - \Divg(\theta \psi \v)
        + \psi \Gradg \theta \cdot \v
    \\
    &\quad 
        - \Divg(\K_\psi \theta)
        + \Gradg \theta\, \K_\psi 
        + \beta\theta \J_\phi\cdot \n
    \\[1ex]
    &= \Divg\big[ \K_e
        + \delgps g \delt \psi
        + g\v 
        - \theta \psi \v 
        - \theta \K_\psi 
    \big]
    \\
    &\quad 
    + \big[
        \delps g 
        - \Divg\big(\delgps g\big) 
        + \delgph f \cdot \n
        - \theta
    \big] \delt\psi
    \\
    &\quad + \big[
        2\nu(\psi) (\D\v \, \n)
        + \psi \Gradg \theta
        - \tfrac 12 (\J \cdot \n) \v
    \big] \cdot \v 
    \\
    &\quad 
        + (\delgph f \cdot \n)  (\v \cdot \Gradg\psi)
        - \eps (\Grad\phi\otimes \Grad\phi)\n \cdot \v
    \\
    &\quad
        + \Gradg \theta\, \K_\psi 
        + (\beta\theta - \mu) \J_\phi\cdot \n
        \end{aligned}
\end{align}
on $\Sigma$. 
Here, we have used that 
\begin{align}
\label{CONV-EQ}
    \v\cdot \Gradg \psi 
    = \v\cdot \Gradg \phi
    = \v\cdot \big[ \Grad\phi - \n(\Grad\phi\cdot \n) \big]
    = \v\cdot \Grad\phi 
    \quad\text{on $\Sigma$}
\end{align}
due to the definition of the surface gradient and the boundary condition \eqref{BC:NP}.
In order to ensure that the inequality \eqref{EST:DISS:S:2} is fulfilled, we choose the chemical potential $\theta$, the mass flux $\K_\psi$, and the energy flux $\K_e$ as follows:
\begin{alignat}{2}
    \label{DEF:THETA*}
    \theta &= \delps g 
        - \Divg\big( \delgps g \big)
        + \delgph f \cdot \n
    &&\quad\text{on $\Sigma$,}
    \\
    \label{DEF:KPSI}
    \K_\psi &= - \mga(\psi) \Gradg \theta
    &&\quad\text{on $\Sigma$,}
    \\
    \label{DEF:KE}
    \K_e &= - \delgps g D_t \psi
        + \theta \psi \v
        + \theta \K_\psi
        - g \v
    &&\quad\text{on $\Sigma$.}
\end{alignat}
\revised{To ensure the condition $K_e\cdot\n = 0$ on $\Sigma$, we need to assume at this point that $\delgps g \cdot \n = 0$ on $\Sigma$.}
In \eqref{DEF:KPSI}, $\mga=\mga(\psi)$ is a nonnegative function representing the mobility.
This means that the first two lines of the right-hand side in \eqref{EST:DISS:S:2} vanish. 
Moreover, since $\v_\n=0$ (see \eqref{BC:NP}), we have $\v=\v_\tau$ and thus,
\begin{align*}
    &\big[
        2\nu(\psi) (\D\v \, \n)
        + \psi \Gradg \theta
        - \tfrac 12 (\J \cdot \n) \v
    \big] \cdot \v 
    \\
    & = \big[
        2\nu(\psi) (\D\v \, \n)
        + \psi \Gradg \theta
        - \tfrac 12 (\J \cdot \n) \v
    \big]_\tau \cdot \v_\tau 
\end{align*}
Therefore, in order to ensure that \eqref{EST:DISS:S:2} is satisfied, we make the constitutive assumption
\begin{align} \label{BC:NAV}
    \big[
        2\nu(\psi) (\D\v \, \n)
        + \psi \Gradg \theta
        - \tfrac 12 (\J \cdot \n) \v
    \big]_\tau 
    = - \gamma(\psi) \v_\tau .
\end{align}
where $\gamma=\gamma(\psi)$ is a nonnegative function. 
This equation can be regarded as an inhomogeneous Navier slip boundary condition.
We now assume that the energy density $g$ is also of Ginzburg--Landau type, that is
\begin{align}
    \label{DEF:GL:S}
    g(\psi,\Gradg \psi) = \frac{\kappa\delta}{2}\abs{\Gradg\psi}^2
    + \frac 1\delta G(\psi)
    \quad\text{on $\Sigma$.}
\end{align}
Here, $\delta>0$ is related to the thickness of the diffuse interface on the boundary, $\kappa\ge 0$ acts as a weight for surface diffusion effects, and $G$ is a potential that usually exhibits a double-well structure. Hence, recalling the definition of the energy density $f$ (see \eqref{DEF:GL:B}), the chemical potential $\theta$ is given as
\begin{align} 
\label{PDE:THETA}
    \theta = -\kappa\delta \Lapg \psi + \frac 1\delta G'(\psi) - \eps\deln \phi
    \quad\text{on $\Sigma$.}
\end{align}
Thanks to \eqref{CONV-EQ}, we further have
\begin{align}
    (\delgph f \cdot \n)  (\v \cdot \Gradg\psi)
    = \eps (\Grad\phi \cdot \n) (\v \cdot \Grad \phi)
    = \eps (\Grad\phi \otimes \Grad\phi)\n \cdot \v
    \quad\text{on $\Sigma$.}
\end{align}
Hence, \eqref{EST:DISS:S:2} reduces to 
\begin{align} \label{EST:DISS:S:3}
    0 &\ge 
        - \gamma(\psi)\abs{\v}^2 
        - \mga(\psi) \abs{\Gradg \theta}^2
        - (\beta\theta - \mu)\,  \mom(\phi) \Grad\mu \cdot \n
    \quad\text{on $\Sigma$.}
\end{align}
The first two terms on the right-hand side are clearly nonpositive as $\gamma(\psi)$ and $\mga(\psi)$ are nonnegative. This means that \eqref{EST:DISS:S:3} is fulfilled if one of the following boundary conditions holds:
\begin{subequations} 
\label{BC:MUTHETA}
\begin{alignat}{2}
    \beta\theta - \mu &= 0
    &&\quad\text{on $\Sigma$,}
    \\
    \mom(\phi) \Grad\mu \cdot \n &= \tfrac 1L (\beta\theta - \mu)
    &&\quad\text{on $\Sigma$ for a constant $L\in(0,\infty)$,}
    \\
    \mom(\phi) \Grad\mu \cdot \n &= 0
    &&\quad\text{on $\Sigma$.}
\end{alignat}
\end{subequations}
Finally, by substituting \eqref{DEF:JPHI} and \eqref{DEF:KPSI} into \eqref{EQ:PSI}, we obtain
\begin{align}
    \label{PDE:PSI}
    \delt \psi + \Divg(\psi\v) = \Divg(\mga(\psi) \Gradg \theta) - \beta \mom(\psi) \Grad\mu\cdot\n
    \quad\text{on $\Sigma$.}
\end{align}

Collecting \eqref{PDE:NS}, \eqref{PDE:DIVV}, \eqref{PDE:J}, \eqref{PDE:PHI}, \eqref{PDE:MU}, \eqref{PDE:PSI}, \eqref{PDE:THETA}, \eqref{BC:NP}, \eqref{BC:PHIPSI}, \eqref{BC:MUTHETA} and \eqref{BC:NAV}, and imposing the initial conditions
\begin{align}
    \label{IC:VPHIPSI}
    \v\vert_{t=0} = \v_0,
    \quad
    \phi\vert_{t=0} = \phi_0
    \quad\text{in $\Om$},
    \qquad
    \psi\vert_{t=0} = \psi_0
    \quad\text{on $\Sigma$,}
\end{align}
for prescribed initial data $\v_0$, $\phi_0$ and $\psi_0$, we have thus derived the system \eqref{NSCH}.

\section{Mathematical analysis: the case of matched densities}
\label{MathAnalysis}

\subsection{Notation and preliminaries}

We fix some notation and assumptions that are supposed to hold throughout the remainder of this paper.

\paragraph{Notation.}
\begin{enumerate}[label=$(\mathrm{N \arabic*})$, ref = $\mathrm{N \arabic*}$]
\item The set of natural numbers excluding zero is denoted by $\N$, and we write $\N_0 = \N\cup\{0\}$.
\item Let $\Omega$ be a bounded Lipschitz domain in $\R^d$ with $d\in \left\lbrace 2,3 \right\rbrace$. We write $\Gamma:=\partial\Omega$.
For any real numbers $k\geq 0$ and $p \in [1, \infty]$, the Lebesgue and Sobolev spaces for functions defined on $\Omega$ with values in $\R$  are denoted as $L^p(\Omega)$ and $W^{k,p}(\Omega)$. We write $\norm{\cdot}_{L^p(\Omega)}$ and $\norm{\cdot}_{W^{k,p}(\Omega)}$ to denote the standard norms on these spaces. If $p = 2$, we use the notation $H^k(\Omega) = W^{k,2}(\Omega)$. We point out that $H^0(\Omega)$ can be identified with $L^2(\Omega)$. Analogously, the 
Lebesgue and Sobolev spaces on $\Gamma$ are denoted by $L^p(\Gamma)$ and $W^{k,p}(\Gamma)$ with corresponding norms $\norm{\cdot}_{L^p(\Gamma)}$ and $\norm{\cdot}_{W^{k,p}(\Gamma)}$, respectively. 
In the case of vector fields on $\Omega$ with values in $\R^n$ for some $n\in\N$ with $n>1$, we use the notation $\bL^p(\Omega)$, $\mathbf{W}^{k,p}(\Omega)$ and $\bH^k(\Omega)$. For simplicity, their norms are denoted as in the scalar case by $\norm{\cdot}_{L^p(\Omega)}$,  $\norm{\cdot}_{W^{k,p}(\Omega)}$ and $\norm{\cdot}_{H^{k}(\Omega)}$, respectively.

Let $I$ be a closed interval in $\R$ and $X$ be a Banach space. The space $C(I;X)$ denotes the set of continuous functions from $I$ to $X$ and, for $k\in\N$, $C^k(I;X)$ denotes the space of $k$-times continuously differentiable functions from $I$ to $X$. In particular, we simply write $C(I)$ and $C^k(I)$ if $X=\R$. Moreover, $C_w(I;X)$ denotes the space of functions mapping from $I$ to $X$ which are continuous on $I$ with respect to the weak topology on $X$. This means that for any function $f\in C_w(I;X)$ and every sequence $(t_k)_{k\in\N}$ in $I$ with $t_k\to t$ in $I$ as $k\to\infty$, it holds $f(t_k) \to f(t)$ weakly in $X$ as $k\to\infty$. Furthermore, for any real numbers $k \geq 0$ and $p \in [1, \infty]$, the Bochner spaces of functions defined on an interval $I$ in $\R$ with values in $X$ are denoted by $L^p(I;X)$ and $W^{k,p}(I;X)$. 

\item For any Banach space $X$, we write $X'$ to denote its dual space. The associated duality pairing of elements $\phi\in X'$ and $\zeta\in X$ is denoted as $\ang{\phi}{\zeta}_X$. If $X$ is a Hilbert space, we write $(\cdot, \cdot)_X$ to denote its inner product. 
\item For any bounded Lipschitz domain $\Omega$, we define
\begin{align*}
\mean{u}_\Omega := \begin{cases}
\frac{1}{\abs{\Omega}} \ang{u}{1}_{H^1(\Omega)} & \text{ if } u \in H^1(\Omega)', \\
\frac{1}{\abs{\Omega}} \int_\Omega u \dx & \text{ if } u \in L^1(\Omega)
\end{cases}
\end{align*}
to denote the (generalized) spatial mean of $u$. Here, $\abs{\Omega}$ denotes the $d$-dimensional Lebesgue measure of $\Omega$. 
The spatial mean of a function $v \in H^1(\Gamma)'$ (or $v \in L^1(\Gamma)$, respectively) is denoted as $\mean{v}_\Ga$ and defined analogously.

\item For any bounded Lipschitz domain $\Omega$ and $d\in\{2,3\}$, we introduce the space 
$$
\bL^2_\Div(\Omega):= 
\left\lbrace \v \in L^2(\Omega; \R^d): \ \Div \, \v =0 \ \text{in} \ \Omega, \ \v \cdot \n =0 \ \text{on} \ \Gamma\right\rbrace
=
\overline{C_{0,\sigma}^\infty(\Omega;\R^d)}^{L^2(\Omega)},
$$
where $C_{0,\sigma}^\infty(\Omega;\R^d) := \left\lbrace \v \in C_c^\infty(\Omega; \R^d): \ \Div \, \v =0 \ \text{in} \ \Omega \right\rbrace$. Moreover, we define 
$$
\bH^1_\Div(\Omega)=
\bH^1(\Omega) \cap \bL^2_\Div(\Omega).
$$

\end{enumerate}

\paragraph{Assumptions.}%
To prove the existence of weak solutions in the case of matched densities, we make the following assumptions:
\begin{enumerate}[label=$(\mathrm{A \arabic*})$, ref = $\mathrm{A \arabic*}$]
    \item \label{ass:dom} We consider a bounded domain $\Omega\subset \R^d$ with $d\in\{2,3\}$ with Lipschitz boundary $\Gamma=\partial\Omega$ and a final time $T>0$. We further use the notation
	\begin{align*}
		Q:=\Omega\times(0,T),\quad \Si:=\Ga\times(0,T).
	\end{align*}

    \item \label{ass:const} The constants occurring in the system \eqref{NSCH:S} satisfy $\eps, \delta, \beta, \kappa>0$ and $L = 0$. 
	Since the choice of $\delta$, $\eps$ and $\kappa$ has no impact on the mathematical analysis, we will simply set $\delta=\eps=\kappa=1$ without loss of generality.

	\item  \label{ass:pot}
	The potentials $F:\R\to[0,\infty)$ and $G:\R\to[0,\infty)$ are twice continuously differentiable 
	and there exist two exponents $p$ and $q$ satisfying
    \begin{align*}
        p\in\begin{cases}
            [2,\infty), &\text{if $d=2$,}\\
            [2,6], &\text{if $d=3$,}
        \end{cases}
        \quad\text{and}\quad
        q\in[2,\infty)
    \end{align*}
    as well as constants $c_{F''},c_{G''}\ge 0$ such that the second-order derivatives satisfy the growth conditions
    \begin{align}
    \label{Ass:F''}
    \abs{F''(s)} &\le
        c_{F''}(1+\abs{s}^{p-2}),
    \\
    \label{Ass:G''}
    \abs{G''(s)} &\le  
    c_{G''}(1+\abs{s}^{q-2}),
    \end{align}
    for all $s\in\R$.

	These assumptions already entail that there exist constants $c_F,c_G,c_{F'},c_{G'}\ge 0$ such that $F'$, $G'$, $F$ and $G$ satisfy the growth conditions
	\begin{align}
		\label{GR:F'}
		\abs{F'(s)} &\le c_{F'}(1+\abs{s}^{p-1}), \\
		\label{GR:G'}
		\abs{G'(s)} &\le c_{G'}(1+\abs{s}^{q-1}), \\
		\label{GR:F}
		F(s) &\le c_{F}(1+\abs{s}^p), \\
		\label{GR:G}
		G(s) &\le c_{G}(1+\abs{s}^q),
	\end{align}
	for all $s\in\R$.
	
	\item  \label{ass:mob}
	The mobility functions $m_\Om:\R\to\R$ and $m_\Ga:\R\to\R$
	are continuous, bounded and uniformly positive. This means that there exist positive constants $m_\Om^*$, $M_\Om^*$, $m_\Ga^*$, $M_\Ga^*$ such that for all $s \in \R$,
	\begin{align}
	     0< m_\Om^* \leq m_\Om(s) \leq M_\Om^*
	    \quad\text{and}\quad
	     0< m_\Ga^* \leq m_\Ga(s) \leq M_\Ga^*.
	\end{align}

    \item \label{ass:vis-fric}
	The viscosity of the mixture $\nu: \R \to \R $ and the friction parameter $\gamma: \R \to \R$ are continuous, bounded and uniformly positive.
    This means that there exist positive constants $\nu_0$, $\nu_1$, $\gamma_0$, $\gamma_1$ such that for all $s \in \R$,
	\begin{align}
	    0<\nu_0\leq \nu(s)\leq \nu_1, \quad\text{and}\quad
	    0<\gamma_0\leq \gamma(s) \leq \gamma_1.
	\end{align}

\end{enumerate}

\bigskip

\begin{remark}\label{REM:ASS}\normalfont 
		We point out that the polynomial double-well potential
		\begin{align*}
		W_\text{dw}(s)=\tfrac 1 4 (s^2-1)^2,\quad s\in\R,
		\end{align*} 
		is a suitable choice for $F$ and $G$ as it satisfies \eqref{ass:pot} with $p = 4$ and $q = 4$. However, singular potentials like the logarithmic Flory--Huggins potential or the double-obstacle potential are not admissible in this setting as they do not even satisfy \eqref{Ass:F''}--\eqref{Ass:G''}.
\end{remark}

\paragraph{Preliminaries.} We next introduce several function spaces, inner products, norms and operators that will be used throughout this paper.
\begin{enumerate}[label=$(\mathrm{P \arabic*})$, ref = $\mathrm{P \arabic*}$]
	\item For any real numbers $k\geq 0$ and $p\in[1,\infty]$, we set
	\begin{align*}
	\LL^p := L^p(\Omega)\times L^p(\Gamma), 
	\quad\text{and}\quad
	\HH^k := H^k(\Omega) \times H^k(\Gamma),
	\end{align*}
	and we identify $\LL^2$ with $\HH^0$. Note that $\HH^k$ is a Hilbert space with respect to the inner product
	\begin{align*}
	\bigscp{(\phi,\psi)}{(\zeta,\xi)}_{\HH^k} := \bigscp{\phi}{\zeta}_{H^k(\Omega)} + \bigscp{\psi}{\xi}_{H^k(\Gamma)}
	\quad\text{for all $(\phi,\psi),(\zeta,\xi)\in\HH^k$,}
	\end{align*}
	and its induced norm $\norm{\cdot}_{\HH^k}:= \scp{\cdot}{\cdot}_{\HH^k}^{1/2}$.
	
	\item \label{pre:D} 
	For $\beta>0$, we introduce the subspace
	\begin{align*}
		\Db := \left\{ (\phi,\psi) \in \HH^1 \;\big\vert\; \phi\vert_\Ga = \beta\psi \;\; \text{a.e. on}\; \Ga \right\} \subset \HH^1,
	\end{align*}
	endowed with the inner product
	$	\scp{\cdot}{\cdot}_\Db := \scp{\cdot}{\cdot}_{\HH^1}$
	and its induced norm.
	The space $\Db$ is a Hilbert space. Moreover, we define the product
	\begin{align*}
		\bigang{(\phi,\psi)}{(\zeta,\xi)}_{\Db} := \scp{\phi}{\zeta}_{L^2(\Omega)} + \scp{\psi}{\xi}_{L^2(\Gamma)}
	\end{align*}
	for all $(\phi,\psi), (\zeta,\xi)\in \LL^2$. By means of the Riesz representation theorem, this product can be extended to a duality pairing on $\Dbp\times \Db$, which will also be denoted as $\ang{\cdot}{\cdot}_{\Db}$. 
	
	In particular, the spaces $\big(\Db,\LL^2,\Dbp\big)$ form a Gelfand triplet, and the operator norm on $\Dbp$ is given by
	\begin{align*}
		\norm{(\phi,\psi)}_{\Dbp} := \sup\Big\{\, \big|\ang{(\phi,\psi)}{(\zeta,\xi)}_{\Db}\big| \;\Big\vert\; (\zeta,\xi)\in\Db \text{ with } \norm{(\zeta,\xi)}_{\Db} = 1 \Big\},
	\end{align*}
	for all $(\phi,\psi) \in \Dbp$. 
	
\item \label{pre:poin} 
We further recall the following \textit{bulk-surface Poincar\'e inequalities}:

\begin{enumerate}[label=$(\mathrm{P 3.\arabic*})$, ref = $\mathrm{P 3.\arabic*}$]

\item \label{poincare:1}
For any $\alpha>0$, there exists a constant $C_P$ depending only on $\alpha$ and $\Omega$ such that
\begin{equation}
    \label{bs-Poincare}
    \norm{(\phi,\psi)}_{\LL^2}\leq C_P \norm{(\Grad\phi,\Gradg\psi)}_{\LL^2} 
\end{equation}
for all $(\phi,\psi)\in \Da$ with 
$\alpha\abs{\Omega}\mean{\phi}_\Om + \abs{\Gamma}\mean{\psi}_\Ga = 0$.
\\[1ex]
This bulk-surface Poincar\'e inequality is a special case of the one established in \cite[Lemma~A.1]{knopf-liu} (with the parameters therein being chosen as $K=0$ and $\alpha,\beta>0$).

\item \label{poincare:2}
There exists a constant $C_P'$ depending only on $\Omega$ such that 
\begin{equation}
    \label{bs-Poincare:2}
    \norm{u}_{L^2(\Omega)}
    \leq C_P' \left( \norm{\nabla u}_{L^2(\Omega)} 
    + \norm{u}_{L^2(\Gamma)} \right)
\end{equation}
for all $u \in H^1(\Omega)$. 
\\[1ex]
This bulk-surface Poincar\'e inequality is established in \cite[Chapter II, Section 1.4]{temam2012infinite}.

\end{enumerate}

\end{enumerate}

\subsection{Existence and uniqueness of weak solutions}

In this section, we consider the case of matched densities, that is $\rho\equiv\trho_1=\trho_2$. This means we consider the system \eqref{NSCH:S} supplemented with the boundary condition \eqref{NSCH:S:8} with $L=0$ and $\beta>0$. For the analysis, without loss of generality, we set $\rho = \kappa = \eps = \delta = 1$. We further employ the usual decomposition
\begin{align}
\label{Korteweg}
    - \, \Div\big(\Grad\phi \otimes \Grad\phi \big)
    = - \Grad\left(\frac 1 2 \abs{\Grad\phi}^2 + F(\phi)\right) + \mu \Grad\phi. 
\end{align}
Hence, replacing the pressure $p$ by
\begin{align*}
    \p := p + \left(\frac 1 2 \abs{\Grad\phi}^2 + F(\phi)\right),
\end{align*}
the system \eqref{NSCH:S} can be restated as follows
\begin{subequations}
\label{NSCH:MD}
\begin{alignat}{2}
    \label{NSCH:MD:1}
    & \delt \v
    + (\v\cdot\nabla) \v
    - \Div\big( 2\nu(\phi) \, \D\v \big)
    + \Grad \p
    = \mu \Grad\phi 
    &&\quad\text{in $Q$}, 
    \\[1ex]
    \label{NSCH:MD:2}
    &\Div \, \v  = 0,
    &&\quad\text{in $Q$}, 
    \\[1ex]
    \label{NSCH:MD:3}
    &\delt \phi + \Div(\phi \v )
    = \Div\big( m_\Om(\phi) \Grad \mu \big)
    &&\quad\text{in $Q$}, 
    \\[1ex]
    \label{NSCH:MD:4}
    &\mu = -\Lap \phi + F'(\phi)
    &&\quad\text{in $Q$}, 
    \\[1ex]
    \label{NSCH:MD:5}
    &\delt \psi + \Divg(\psi \v_\tau ) 
    = \Divg\big( m_\Ga(\psi) \Gradg \theta \big) - \beta m_\Om(\psi)\, \deln \mu 
    &&\quad\text{on $\Sigma$}, 
    \\[1ex]
    \label{NSCH:MD:6}
    &\theta = - \Lapg \psi + G'(\psi)
    + \deln \phi
    &&\quad\text{on $\Sigma$}, 
    \\[1ex]
    \label{NSCH:MD:7}
    &\v\cdot \n = 0,
    \qquad
    \phi\vert_\Gamma = \psi, 
    \qquad
    \mu\vert_\Gamma = \beta\theta
    &&\quad\text{on $\Sigma$},
    \\[1ex]
    \label{NSCH:MD:8}
    & \big[2\nu(\psi) (\D\v\,\n) + \gamma(\psi) \v \big]_\tau
    = \big[ -\psi \Gradg \theta \big]_\tau
    &&\quad\text{on $\Sigma$}, 
    \\[1ex]
    \label{NSCH:MD:9}
    &\v\vert_{t=0} = \v_0,\quad \phi\vert_{t=0} = \phi_0
    &&\quad\text{in $\Omega$},
    \\[1ex]
    \label{NSCH:MD:10}
    &\psi\vert_{t=0} = \psi_0
    &&\quad\text{on $\Gamma$}.
\end{alignat}
\end{subequations}
Here, we recall that in the case of matched densities, the term $\J$ vanishes. Therefore, in the Navier--Stokes equation \eqref{NSCH:MD:1} and in the corresponding boundary condition \eqref{NSCH:MD:8}, the terms related to $\J$ do not appear any more. 
The total energy associated with this system is
\begin{align}
    \label{DEF:ENERGY}
    E(\v,\phi,\psi) 
    = \intO \frac{1}{2} \abs{\v}^2 \dx
        + \intO \frac 12 \abs{\Grad\phi}^2 + F(\phi) \dx
        + \intG \frac 12 \abs{\Gradg\psi}^2 + G(\psi) \dS.
\end{align}

The notion of weak solutions to system~\eqref{NSCH:MD} is defined as follows:

\begin{definition}[Weak solutions of system \eqref{NSCH:MD}]
    \label{DEF:WS}
    Suppose that the assumptions
     \eqref{ass:dom}--\eqref{ass:vis-fric} hold.
    Let $\v_0 \in \bL^2_\Div(\Omega)$ and $(\phi_0,\psi_0) \in \DD_1$ 
    be arbitrary.
    The quintuplet $(\v,\phi,\psi,\mu,\theta)$ is called a weak solution of the system \eqref{NSCH:MD} on $[0,T]$ if the following properties hold:
    \begin{enumerate}[label = \textnormal{(\roman*)}, leftmargin=*, topsep=0pt]
        \item \label{DEF:WS:REG}
        The functions $\v$, $\phi$, $\psi$, $\mu$ and $\theta$ have the following regularity:
        \begin{subequations}
        \label{WS:REG}
       \begin{align}
            &\v \in 
            L^\infty\big(0,T;\bL_\Div^2(\Omega)\big)
                \cap L^2\big(0,T;\bH^1_\Div(\Omega) \big), 
            \\
            &\left\{
            \begin{aligned}
                &\v \in W^{1,2} \big(0,T; \bH^1_\Div(\Omega)' \big) 
                 \cap C\big([0,T];\bL^2_\Div(\Omega)\big)  
                    \ &&\text{if } \ d=2,
                \\   
                &\v \in W^{1,\frac43} \big(0,T; \bH^1_\Div(\Omega)' \big) 
                 \cap C_w\big([0,T];\bL^2_\Div(\Omega)\big)
                    \ &&\text{if } \ d=3,
            \end{aligned} 
            \right.
            \\
            &(\phi,\psi) \in C([0,T];\LL^2)
                \cap H^1(0,T;\Dbp)
                \cap L^\infty(0,T;\DD_1),
            \\
            &(\mu,\theta) \in L^4(0,T;\LL^2)\cap L^2(0,T;\DD_\beta).
        \end{align}
        \end{subequations}
        \item \label{DEF:WS:INI}
        The functions $\v$, $\phi$ and $\psi$ satisfy the initial conditions
        \begin{align}
            \label{INI}
            \v\vert_{t=0} = \v_0,\quad \phi\vert_{t=0} = \phi_0
            \quad\text{a.e.~in $\Omega$},
            \quad\text{and}\quad
            \psi\vert_{t=0} = \psi_0
            \quad\text{a.e.~on $\Gamma$}.
        \end{align}
        \item \label{DEF:WS:WF}
        The functions $\v$, $\phi$, $\psi$, $\mu$ and $\theta$ satisfy the weak formulation
        \begin{subequations}
        \label{WF}
        \begin{align}
        \label{WF:1}
        &\begin{aligned}
        &\ang{\delt \v}{\w}_{\bH_\Div^1(\Omega)} 
            - \intO (\v \otimes \v) : \Grad\w \dx
            + \intO 2\nu(\phi) \D\v : \D\w \dx
        \\
        &\quad
        = \intO \mu \Grad\phi \cdot \w \dx
        - \intG \gamma(\psi) \v\cdot \w 
            + \psi\Gradg\theta \cdot \w 
            \dS,
        \end{aligned}    
        \\[1ex]
        \label{WF:2}
        &\begin{aligned}
        &\bigang{(\delt \phi,\delt \psi)}{(\zeta,\xi)}_{\Db}
            - \intO \phi\v \cdot \Grad\zeta \dx
            - \intG \psi\v \cdot \Gradg\xi \dS
        \\     
        &\quad= 
        - \intO m_\Om(\phi) \Grad \mu \cdot \Grad \zeta \dx
        	-\intG m_\Ga(\psi) \Gradg \theta \cdot \Gradg \xi \dS,
        \end{aligned}	
        \\[1ex]
        \label{WF:3}
        &\begin{aligned}
        &\intO \mu \, \eta \dx     
            + \intG \theta \, \vartheta \dS 
        \\
        &\quad 
        = \intO \Grad\phi \cdot \Grad \eta 
            + F'(\phi)\eta\dx
            + \intG \Gradg\psi \cdot \Gradg\vartheta 
            + G'(\psi)\vartheta \dS
            \end{aligned}
        \end{align}
        \end{subequations}
        for all test functions $\w\in \bH^1_\Div(\Omega)$, $(\zeta,\xi)\in\Db$, $(\eta,\vartheta)\in\DD_1$.
    \item \label{DEF:WS:MASS}
    The functions $\phi$ and $\psi$ satisfy the mass        conservation law
        \begin{align}
            \label{WS:MASS}
            \beta \intO \phi(t) \dx + \intG \psi(t) \dS 
            = \beta \intO \phi_0 \dx + \intG \psi_0 \dS
            =: m
            \quad\text{for all $t\in[0,T]$}.
        \end{align}
    \item \label{DEF:WS:ENERGY}
    The functions $\v$, $\phi$, $\psi$, $\mu$ and $\theta$ satisfy the weak energy dissipation law
        \begin{align}
            \label{WS:ENERGY}
            \begin{aligned}
          &E\big(\v(t),\phi(t),\psi(t)\big)
            + \int_0^t\intO 2\nu(\phi)\, \abs{\D\v}^2 \dx\dt
            + \int_0^t\intG \gamma(\psi) \abs{\v}^2 \dS\dt
            \\
            &\quad + \int_0^t\intG m_\Ga(\psi) \abs{\Gradg \theta}^2 \dS\dt
            + \int_0^t\intO m_\Om(\phi) \abs{\Grad \mu}^2 \dx\dt 
            \le E(\v_0,\phi_0,\psi_0)
            \end{aligned}
        \end{align}
        for all $t\in[0,T]$.
    \end{enumerate}
\end{definition}

The existence of such a weak solution is ensured by the following theorem.

\begin{theorem}
    \label{THM:WS}
    Suppose that the assumptions \eqref{ass:dom}--\eqref{ass:vis-fric} hold and
    let $\v_0 \in \bL^2_\Div(\Omega)$ and $(\phi_0,\psi_0) \in \DD_1$ 
    be arbitrary. Then, there exists a weak solution of system \eqref{NSCH:MD} in the sense of Definition~\ref{DEF:WS}.
    If we additionally suppose that the domain $\Omega$ is of class $C^k$ with $k\in\{2,3\}$, and in the case $d=3$, we further assume that assumption \eqref{ass:pot} holds with $p< 6$, then $(\phi,\psi) \in L^2(0,T;\HH^k)$
    as well as
    \begin{subequations}
    \label{EQ:MUTH:STRG}
    \begin{alignat}{2}
        \mu &= -\Lap \phi + F'(\phi) &&\quad\text{a.e.~in $Q$},\\
        \theta &= -\Lapg \psi + G'(\psi) + \deln\phi &&\quad\text{a.e.~on $\Sigma$}.
    \end{alignat}
    \end{subequations}
\end{theorem}

\bigskip

In two dimensions, we show the uniqueness of the weak solutions under additional assumptions on the domain $\Omega$, the viscosity $\nu$, the mobilities $\mom$ and $\mga$ and the potentials $F$ and $G$.

\begin{theorem}
    \label{THM:UNI}
    Suppose that the assumptions \eqref{ass:dom}--\eqref{ass:vis-fric} hold with $d=2$.
    Without loss of generality, we assume $p\ge 3$.
    We additionally assume that $\Omega$ is of class $C^3$, $\gamma$ is differentiable and $\gamma'$ is locally bounded,
    the functions $\nu$, $\mom$ and $\mga$ are constant, the potentials $F$ and $G$ satisfy the relation
    \begin{align}
    \label{REL:POT}
        F(s) = \beta G(s) \quad\text{for all $s\in\R$,}
    \end{align}
    \revised{and there exists a constant $c_{F'''}\ge 0$ such that}
    \begin{align*}
        \revised{|F'''(s)| \le c_{F'''}(1+|s|^{p-3}) \quad\text{for all $s\in\R$.} }
    \end{align*}
    Then, the weak solution of system \eqref{NSCH:S} given by Theorem~\ref{THM:WS} is unique. In addition, given two weak solutions $(\v_1,\phi_1,\psi_1,\mu_1,\theta_1)$ and $(\v_2,\phi_2,\psi_2,\mu_2,\theta_2)$ corresponding to the 
initial data $(\v_{0}^1, (\phi_0^1, \psi_0^1))$, $(\v_{0}^2, (\phi_0^2, \psi_0^2)) \in \bL^2_\Div(\Omega)\times\DD_1$, the following continuous dependence estimate
\begin{equation}
\label{cont-dep*}
\begin{split}
& \norm{ \v_1(t)-\v_2(t)}_{L^2(\Omega)}^2 +  \norm{ \phi_1(t)-\phi_2(t)}_{H^1(\Omega)}^2
 + \norm{\psi_1(t)-\psi_2(t)}_{H^1(\Ga)}^2 \\
 &\leq 
 C\mathrm{e}^{C\int_0^T \mathcal{F}(t) \dt}
 \left( \norm{ \v_0^1-\v_0^2}_{L^2(\Omega)}^2 +  \norm{ \phi_0^1-\phi_0^2}_{H^1(\Omega)}^2
 + \norm{\psi_0^1-\psi_0^2}_{H^1(\Ga)}^2 \right),
\end{split}
 \end{equation}
holds for all $t\in [0,T]$, where 
$$
\mathcal{F}:= 
       1+ \norm{\v_1}_{L^4(\Omega)}^4+
       \norm{\v_2}_{H^1(\Omega)}^2 +
       \norm{\mu_2}_{H^1(\Omega)}^2+
       \norm{\theta_2}_{H^1(\Ga)}^2+
       \norm{\phi_1}_{W^{1,4}(\Omega)}^2,
$$
and the constant $C>0$ depends only on $\Omega$, the parameters of the system and the norms of the initial data.
\end{theorem}

\subsection{Proofs}

\begin{proof}[Proof of Theorem~\ref{THM:WS}]
We construct a weak solution to system \eqref{NSCH:MD} by discretizing the weak formulation \eqref{WF} through a Faedo--Galerkin scheme. 

In order to obtain suitable {\it uniform} estimates for the Galerkin approximation, it is necessary that the phase-fields $(\phi,\psi)$ and the chemical potentials $(\mu,\theta)$ are approximated in the same finite-dimensional subspace. In turn this is only possible if the phase-fields and the chemical potentials are coupled by the same Dirichlet type boundary condition (cf. \eqref{NSCH:MD:7}), which is not the case for general $\beta>0$. However, by a change of variables, we can reformulate system \eqref{NSCH:MD} in such a way that the Dirichlet type boundary conditions for the redefined variables are exactly the same. To this end, we set $\alpha:=\sqrt{\beta}>0$, and we introduce the functions
\begin{alignat}{2}
\label{CofV}
\left\{\;
\begin{aligned}
    \Psi &:= \alpha^{-1} \psi,
    &&\quad\text{i.e.,}\quad
    \psi = \alpha\Psi,
    \\
    \Psi_0 &:= \alpha^{-1} \psi_0,
    &&\quad\text{i.e.,}\quad
    \psi_0 = \alpha\Psi_0,
    \\
    \Theta &:= \alpha \theta,
    &&\quad\text{i.e.,}\quad
    \theta = \alpha^{-1} \Theta,
    \\
    \nga(\Psi) &:= \alpha^{-2} \mga(\alpha \Psi) 
    = \alpha^{-2} \mga(\psi).
    \end{aligned}
\right.
\end{alignat}
By this change of variables, system \eqref{NSCH:MD} is equivalent to
\begin{subequations}
\label{NSCH:MD*}
\begin{alignat}{2}
    \label{NSCH:MD*:1}
    & \delt \v
    + (\v\cdot\nabla) \v
    - \Div\big( 2\nu(\phi) \, \D\v \big)
    + \Grad \p
    = \mu \Grad\phi 
    &&\quad\text{in $Q$}, 
    \\[1ex]
    \label{NSCH:MD*:2}
    &\Div \, \v = 0,
    &&\quad\text{in $Q$}, 
    \\[1ex]
    \label{NSCH:MD*:3}
    &\delt \phi + \Div(\phi \v )
    = \Div\big( m_\Om(\phi) \Grad \mu \big)
    &&\quad\text{in $Q$}, 
    \\[1ex]
    \label{NSCH:MD*:4}
    &\mu = -\Lap \phi + F'(\phi)
    &&\quad\text{in $Q$}, 
    \\[1ex]
    \label{NSCH:MD*:5}
    &\delt \Psi + \Divg(\Psi \v_\tau ) 
    = \Divg\big( \nga(\Psi) \Gradg \Theta \big) - \alpha\, m_\Om(\alpha\Psi)\, \deln \mu 
    &&\quad\text{on $\Sigma$}, 
    \\[1ex]
    \label{NSCH:MD*:6}
    &\Theta = - \alpha^2 \, \Lapg \Psi + \alpha G'(\alpha\Psi)
    + \alpha \deln \phi
    &&\quad\text{on $\Sigma$}, 
    \\[1ex]
    \label{NSCH:MD*:7}
    &\v\cdot \n = 0
    \qquad
    \phi\vert_\Gamma = \alpha \Psi, 
    \qquad
    \mu\vert_\Gamma = \alpha \Theta
    &&\quad\text{on $\Sigma$},
    \\[1ex]
    \label{NSCH:MD*:8}
    & \big[2\nu(\alpha\Psi) (\D\v\,\n) + \gamma(\alpha\Psi) \v \big]_\tau
    = \big[ -\Psi \Gradg \Theta \big]_\tau
    &&\quad\text{on $\Sigma$}, 
    \\[1ex]
    \label{NSCH:MD*:9}
    &\v\vert_{t=0} = \v_0,\quad \phi\vert_{t=0} = \phi_0
    &&\quad\text{in $\Omega$},
    \\[1ex]
    \label{NSCH:MD*:10}
    &\Psi\vert_{t=0} = \Psi_0
    &&\quad\text{on $\Gamma$}.
\end{alignat}
\end{subequations}

In view of Definition~\ref{DEF:WS}, a weak solution $(\v, \phi, \Psi, \mu, \Theta)$ of system \eqref{CofV} needs to have the regularity
\begin{subequations}
\label{WS:REG:alpha}
\begin{align}
    &\v \in 
    L^\infty\big(0,T;\bL_\Div^2(\Omega)\big)
        \cap L^2\big(0,T;\bH^1_\Div(\Omega) \big), 
    \\
    &\left\{
    \begin{aligned}
        &\v \in W^{1,2} \big(0,T; \bH^1_\Div(\Omega)' \big) 
         \cap C\big([0,T];\bL^2_\Div(\Omega)\big)  
            \ &&\text{if } \ d=2,
        \\   
        &\v \in W^{1,\frac43} \big(0,T; \bH^1_\Div(\Omega)' \big) 
         \cap C_w\big([0,T];\bL^2_\Div(\Omega)\big)
            \ &&\text{if } \ d=3,
    \end{aligned} 
    \right.
    \\
    &(\phi,\Psi) \in C([0,T];\LL^2)
        \cap H^1(0,T;\DD_\alpha')
        \cap L^\infty(0,T;\DD_\alpha),
    \\
    &(\mu,\Theta) \in L^4(0,T;\LL^2)\cap L^2(0,T;\DD_\alpha),
\end{align}
\end{subequations}
and satisfy the weak formulation which reads as
\begin{subequations}
\label{WF:alpha}
\begin{align}
\label{WF:1:alpha}
&\begin{aligned}
&\ang{\delt \v}{\w}_{\bH_\Div^1(\Omega)} 
    - \intO (\v \otimes \v) : \Grad\w \dx
    + \intO 2\nu(\phi) \D\v : \D\w \dx
\\
&\quad
= \intO \mu \Grad\phi \cdot \w \dx
- \intG \gamma(\alpha \Psi) \v\cdot \w 
    + \Psi\Gradg\Theta \cdot \w \dS,
\end{aligned}    
\\[1ex]
\label{WF:2:alpha}
&\begin{aligned}
&\bigang{(\delt \phi,\delt \Psi)}{(\zeta,\xi)}_{\DD_\alpha}
- \intO \phi\v \cdot \Grad\zeta \dx
    - \intG \Psi \v \cdot \Gradg\xi \dS
\\     
&\quad= 
- \intO m_\Om(\phi) \Grad \mu \cdot \Grad \zeta \dx
	-\intG \eta_\Ga(\Psi) \Gradg \Theta \cdot \Gradg \xi \dS,
\end{aligned}	
\\[1ex]
\label{WF:3:alpha}
&\begin{aligned}
&\intO \mu \, \eta \dx     
    + \intG \Theta \, \vartheta \dS 
\\
&\quad 
= \intO \Grad\phi \cdot \Grad \eta 
    + F'(\phi)\eta\dx
    + \intG \alpha^2 \Gradg\Psi \cdot \Gradg\vartheta 
    + \alpha G'(\alpha \Psi)\vartheta \dS
    \end{aligned}
\end{align}
\end{subequations}
for all test functions $\w\in \bH^1_\Div(\Omega)$ and $(\zeta,\xi),\,(\eta,\vartheta)\in\DD_\alpha$.

We now prove the existence of a weak solution of system \eqref{NSCH:MD*} (in the aforementioned sense) and afterwards,
we use the change of variables \eqref{CofV} to infer the existence of a weak solution to the original model \eqref{NSCH:MD}.
The weak formulation \eqref{WF:alpha} can indeed be discretized by a suitable Faedo--Galerkin scheme.

\noindent\textbf{Step 1: Discretization via a bulk-surface Faedo--Galerkin scheme.}
It has already been established in \cite[Theorem~4.4]{knopf-liu} that the second-order elliptic eigenvalue problem
\begin{align}
    \label{EIG}
    \left\{
    \begin{aligned}
        -\Lap \zeta &= \lambda \zeta &&\text{in $\Omega$},\\
        -\Lapg \xi + \alpha \deln \zeta &= \lambda \xi &&\text{on $\Gamma$},\\
        \zeta\vert_\Gamma &= \alpha \xi &&\text{on $\Gamma$}
    \end{aligned}
    \right.
\end{align}
has countably many eigenvalues $\{\lambda_i\}_{i\in\N}$ with $\lambda_i\ge 0$ for all $i\in\N$, and the corresponding eigenfunctions $\{(\zeta_k,\xi_k)\}_{k\in\N} \subset \DD_\alpha$
can be chosen such that they form an orthonormal basis of $\LL^2$.
Here, we fix the eigenfunction associated with the first eigenvalue $\lambda_1=0$ as 
\begin{align}
    \label{1STEIG}
    (\zeta_1,\xi_1) = \frac{1}{\sqrt{\alpha^2|\Omega|+|\Gamma|}}(\alpha,1).
\end{align}
Moreover, in view of their construction (see \cite{knopf-liu}), the eigenfunctions $(\zeta_k,\xi_k)$ with $k \ge 2$ satisfy the mean value constraint
\begin{align}
    \label{EIG:MEAN}
    \alpha |\Omega| \meano{\zeta_k} + |\gamma| \meang{\xi_k} = 0.
\end{align} 
For any $k\in\N$, we introduce the finite dimensional subspace 
\begin{align}
    \label{DEF:ZK}
    \ZZK := \text{span}\,\{(\zeta_1,\xi_1),...,(\zeta_k,\xi_k)\} \subset \Da,
\end{align}
and we write
\begin{align*}
    \mathbb{P}_{\ZZK}(\zeta,\xi)
    = \big(\mathbb{P}^\Om_{\ZZK}(\zeta,\xi),\mathbb{P}^\Ga_{\ZZK}(\zeta,\xi)\big)
    \quad\text{for any $(\zeta,\xi)\in \DD_\alpha$},
\end{align*}
to denote the $\LL^2$-orthogonal projection onto $\ZZK$.

Next, we consider the Stokes operator with Navier boundary conditions (see, e.g., \cite[Appendix~A]{Abels2012})
\begin{equation}
\mathbf{A}: \mathcal{D}(\mathbf{A}) \subset \bL^2_\Div(\Omega) \rightarrow \bL^2_\Div(\Omega): \quad 
\v \mapsto \mathbf{A}\v:=
- \mathbb{P} \Div (2 D \v),
\end{equation}
where $\mathbb{P}$ is the Leray projection operator, with domain 
\begin{equation}
\mathcal{D}(\mathbf{A})
=
\Big\{ \v\in \bH^2(\Omega) \cap \bL^2_\Div(\Omega): \big[2 (\D\v\,\n) + \v \big]_\tau=0 \;\text{on $\Gamma$} \Big\}.
\end{equation}
Since $\mathbf{A}$ is a positive and self-adjoint operator on $\bL^2_\Div(\Omega)$ with compact inverse, and 
$$
\left( \bL^2_\Div(\Omega), \mathcal{D}(\mathbf{A}) \right)_{\frac12, 2}= \bH^1_\Div (\Omega),
$$ 
there exists a sequence of countably many positive eigenvalues $\big\lbrace \widetilde{\lambda}_i \big\rbrace_{i \in \mathbb{N}}$ and the corresponding eigenfunctions $\lbrace \w_i \rbrace_{i \in \mathbb{N}}\subset \bH^1_\Div(\Omega)$, 
which are determined by
\begin{equation}
    \int_{\Omega} 2 D \w_i : \nabla \v \dx + \int_{\Gamma} \w_i \cdot \v \dS = \widetilde{\lambda}_i \int_{\Omega}  \w_i \cdot \v \dx \quad \text{for all $\v \in \bH^1_\Div (\Omega)$},
\end{equation}
can be chosen in such a way that they form an orthonormal basis of $\bL^2_\Div(\Omega)$. 
For any $k \in \mathbb{N}$, we define the finite dimensional subspace 
$$
\UU_k = \text{span}\,\{\w_1,...,\w_k\} \subset \bH^1_\Div(\Omega),
$$
and the $\bL^2_\Div(\Omega)$-orthogonal projection onto $\UU_k$ is denoted by $\mathbb{P}_{\UU_k}$.

Next, for any $k\in\N$ and $t\in[0,T]$, we make the ansatz
\begin{subequations}
\label{DEF:APPROX}
\begin{alignat}{2}
    \label{DEF:VK}
    \v_k(t) 
    &:= \sum_{i=1}^k a^k_i(t)\, \w_i,
    &&\quad\text{a.e. in $\Omega$},
    \\
    \label{DEF:PHIPSIK}
    \big(\phi_k(t),\Psi_k(t)\big) 
    &:= \sum_{i=1}^k b^k_i(t)\, (\zeta_i,\xi_i),
    &&\quad\text{a.e. in $\Omega\times\Gamma$},
    \\
    \label{DEF:MUTHK}
    \big(\mu_k(t),\Theta_k(t)\big) 
    &:= \sum_{i=1}^k c^k_i(t)\, (\zeta_i,\xi_i),
    &&\quad\text{a.e. in $\Omega\times\Gamma$}.
\end{alignat}
\end{subequations}
Here, the scalar, time-dependent coefficients $a^k_i$, $b^k_i$ and $c^k_i$, $i=1,...,k$ are assumed to be continuously differentiable functions that are still to be determined. They need to be designed in such a way that the discretized weak formulation 
\begin{subequations}
    \label{WFD}
    \begin{align}
    \label{WFD:1}
    &\begin{aligned}
    &\ang{\delt \v_k}{\w}_{\bH_\Div^1(\Omega)} 
        - \intO (\v_k \otimes \v_k) \Grad\w \dx 
        + \intO 2\nu(\phi_k) \D\v_k : \D\w \dx
    \\
    &\quad
    = \intO \mu_k \Grad\phi_k \cdot \w \dx
    - \intG \gamma(\alpha\Psi_k) \v_k\cdot \w 
        + \Psi_k\Gradg\Theta_k \cdot \w \dS,
        \end{aligned}
    \\[1ex]
    \label{WFD:2}
    &\begin{aligned}
    &\bigang{(\delt \phi_k,\delt \Psi_k)}{(\zeta,\xi)}_{\Da} 
        -  \intO \phi_k\v_k \cdot \Grad\zeta \dx
        - \intG \Psi_k\v_k \cdot \Gradg\xi \dS
    \\     
    &\quad= -  \intO m_\Om(\phi_k) \Grad \mu_k : \Grad \zeta \dx
    	-\intG n_\Ga(\Psi_k) \Gradg \Theta_k : \Gradg \xi \dS,
    	\end{aligned}
    \\[1ex]
    \label{WFD:3}
    &\begin{aligned}
    &\intO \mu_k \cdot \zeta \dx     
        + \intG \Theta_k \cdot \xi \dS
    \\
    &\quad 
    = \intO \Grad\phi_k \cdot \Grad \zeta 
        + F'(\phi_k) \zeta \dx
        + \intG \alpha^2 \Gradg\Psi_k \cdot \Gradg \xi 
        + \alpha G'(\alpha\Psi_k)\xi \dS
    \end{aligned}
    \end{align}
\end{subequations}
holds for all $\w\in \UU_k$ and all $(\zeta,\xi)\in\ZZK$,
supplemented with the initial conditions
\begin{alignat}{2}
    \label{WFD:INI:1}
    \v_k\vert_{t=0}
    &= \mathbb{P}_{\UU_k}(\v_0) \in \UU_k
    &&\quad\text{in $\Omega$,}
    \\
    \label{WFD:INI:2}
    (\phi_k,\Psi_k)\vert_{t=0} 
    &= \mathbb{P}_{\ZZK}(\phi_0,\Psi_0) \in \ZZK
    &&\quad\text{in $\Omega\times\Gamma$},
\end{alignat}
where $\Psi_0:=\alpha^{-1}\psi_0$.

Let now $\mathbf{a}^k := (a_1^k,...,a_{k}^k)$, $\mathbf{b}^k := (b_1^k,...,b_{k}^k)$ and $\mathbf{c}^k := (c_1^k,...,c_{k}^k)$ denote the coefficient vectors. 
Testing \eqref{WFD:1} with $\w_1,...,\w_k$ and \eqref{WFD:2} with $(\zeta_1,\xi_1),...,(\zeta_k,\xi_k)$, we conclude that $(\mathbf{a}^k,\mathbf{b}^k)^\top$ is determined by a system of $2k$ ordinary differential equations whose right-hand side depends continuously on $\mathbf{a}^k$, $\mathbf{b}^k$ and $\mathbf{c}^k$. Due to \eqref{WFD:INI:1} and \eqref{WFD:INI:2}, this ODE system is subject to the initial condition
\begin{alignat*}{2}
    [\mathbf a^k]_i(0) &= a_i^k(0) = \bigscp{\v_0}{\w_i}_{\bL^2(\Omega)},
    &&\quad i\in\{1,...,k\},
    \\
    [\mathbf b^k]_i(0) &= b_i^k(0) = \bigscp{(\phi_0,\Psi_0)}{(\zeta_i,\xi_i)}_{\LL^2},
    &&\quad i\in\{1,...,k\}.
\end{alignat*}
Moreover, testing \eqref{WFD:3} with $(\zeta_1,\xi_1),...,(\zeta_k,\xi_k)$, we infer that $\mathbf{c}^k$ is explicitly given by a system of $k$ algebraic equations whose right-hand side depends continuously on $\mathbf{b}^k$. We now replace the variable $\mathbf{c}^k$ appearing in the right-hand side of the aforementioned ODE system to obtain a closed $2k$-dimensional ODE system describing $(\mathbf{a}^k,\mathbf{b}^k)^\top$ whose right-hand side depends continuously on $(\mathbf{a}^k,\mathbf{b}^k)^\top$. We can thus apply the Cauchy--Peano theorem to infer the existence of at least one local solution
$(\mathbf{a}^k,\mathbf{b}^k)^\top: [0,T_k^*)\cap [0,T] \to \R^{2k}$
with $T_k^*>0$ to the corresponding initial value problem. Without loss of generality, we assume that $T_k^*\le T$ and that $(\mathbf{a}^k,\mathbf{b}^k)^\top$ is right-maximal, meaning that $T_k^*$ is chosen as large as possible. Now, we can reconstruct $\mathbf{c}^k:[0,T_k^*) \to \R^{k}$ by the aforementioned system of $k$ algebraic equations. In view of \eqref{DEF:APPROX}, we thus obtain functions
\begin{align*}
    \v_k \in C^1\big( [0,T_k^*) ; \bH^1_{\Div}(\Omega) \big),
    \quad
    (\phi_k,\Psi_k) 
    \in C^1\big( [0,T_k^*) ; \Da \big),
    \quad
    (\mu_k,\Theta_k)
    \in C^1\big( [0,T_k^*) ; \Da \big),
\end{align*}
which fulfill the discretized weak formulation \eqref{WFD} on the time interval $[0,T_k^*)$. 

\noindent\textbf{Step 2: Uniform estimates.}
We establish suitable estimates for each approximate solution $(\v_k,\phi_k,\psi_k,\mu_k,\theta_k)$ which are uniform with respect to the index $k$.
In particular, let $T_k<T_k^*$ be arbitrary. In the following, let $C$ denote generic positive constants depending only on the initial data and the constants introduced in \eqref{ass:const}--\eqref{ass:vis-fric} including the final time $T$, but is independent of $k$ and $T_k$.

Testing \eqref{WFD:1} with $\v_k$, \eqref{WFD:2} with $(\mu_k,\theta_k)$, and \eqref{WFD:3} with $-(\delt\phi_k,\delt\Psi_k)$, adding the resulting equations, and integrating with respect to time from $0$ to $t$,
we derive the discretized energy identity
\begin{align}
\label{DISC:ENERGY}
\begin{aligned}
&E\big(\v_k(t),\phi_k(t),\alpha\Psi_k(t))
+ \int_0^t \intO 2\nu(\phi_k)\, \abs{\D\v_k}^2 \dx \dt
+ \int_0^t \intG \gamma(\alpha\Psi_k) \abs{\v_k}^2 \dS \dt
 \\
&\qquad + \int_0^t \intO \mom(\phi_k) \abs{\Grad \mu_k}^2 \dx \dt
+ \int_0^t \intG n_\Ga(\psi_k) \abs{\Gradg \Theta_k}^2 \dS \dt \\[1ex]
&
= E\big(\v_k(0),\phi_k(0),\alpha\Psi_k(0)) 
\end{aligned}
\end{align}
for all $t\in[0,T_k]$. Here, $E$ stands for the total energy introduced in \eqref{DEF:ENERGY}. 
Due to the growth conditions on $F$ and $G$ (see \eqref{ass:pot}), the Sobolev embeddings $H^1(\Omega)\hookrightarrow L^6(\Omega)$ and $H^1(\Ga)\hookrightarrow L^p(\Ga)$ for any $1\leq p<\infty$, and the properties of the projections $\mathbb{P}_{\UU_k}$ and $\mathbb{P}_{\ZZK}$,
we use the initial conditions \eqref{WFD:INI:1} and \eqref{WFD:INI:2} to infer that
\begin{align*}
   &E\big(\v_k(0),\phi_k(0),\alpha\Psi_k(0)) 
   \notag\\
   &\leq \frac12 \norm{\mathbb{P}_{\UU_k}(\v_0)}_{L^2(\Omega)}^2
   +\frac12 \norm{\nabla \mathbb{P}^\Om_{\ZZK}(\phi_0,\Psi_0)}_{L^2(\Omega)}^2
   +\frac{\alpha^2}2 \norm{\nabla \mathbb{P}^\Ga_{\ZZK}(\phi_0,\Psi_0)}_{L^2(\Ga)}^2
   \notag\\
   &\quad + c_F \left( |\Omega|+ \norm{\mathbb{P}^\Om_{\ZZK}(\phi_0,\Psi_0)}_{L^p(\Omega)}^p\right)
   + c_G \left( |\Ga|+ \norm{\mathbb{P}^\Ga_{\ZZK}(\phi_0,\Psi_0)}_{L^q(\Ga)}^q\right) \notag
   \\[1ex]
   & \leq \frac12 \norm{\v_0}_{L^2(\Omega)}^2
   + C \norm{(\phi_0,\Psi_0)}_{\HH^1}^2
   \notag\\
   &\quad + C \left( 1 + \norm{\mathbb{P}^\Om_{\ZZK}(\phi_0,\Psi_0)}_{H^1(\Omega)}^p\right)
   + C \left( 1 + \norm{\mathbb{P}^\Ga_{\ZZK}(\phi_0,\Psi_0)}_{H^1(\Ga)}^q\right) \notag\\[1ex]
   & \leq \frac12 \norm{\v_0}_{L^2(\Omega)}^2
   + C \norm{(\phi_0,\Psi_0)}_{\HH^1}^2
   \notag\\
   &\quad + C \left( 1 + \norm{(\phi_0,\Psi)}_{\HH^1}^p\right)
   + C \left( 1 + \norm{(\phi_0,\Psi_0)}_{\HH^1}^q\right), \notag
\end{align*}
which, in turn, entails that
\begin{align}
\label{EST:EK}
    E\big(\v_k(0),\phi_k(0),\alpha\Psi_k(0)) \le C.
\end{align}
Recalling that $F,G \ge 0$ and that the functions $\nu$, $\gamma$, $\mom$ and $\mga$ (and thus also $n_\Ga$) are uniformly positive (see \eqref{ass:mob}--\eqref{ass:vis-fric}), we conclude the uniform bounds
\begin{align}
    \label{EST:AP:1}
    \norm{\v_k}_{L^\infty(0,T_k;L^2(\Om))} 
    + \norm{\Grad \v_k}_{L^2(0,T_k;L^2(\Om))} 
    + \norm{\v_k}_{L^2(0,T_k;L^2(\Ga))}
    &\le C,
    \\
    \label{EST:AP:2}
    \norm{\Grad\phi_k}_{L^\infty(0,T_k;L^2(\Om))} 
    + \norm{\Gradg\Psi_k}_{L^\infty(0,T_k;L^2(\Ga))}
    &\le C,
    \\
    \label{EST:AP:3}
    \norm{\Grad\mu_k}_{L^2(0,T_k;L^2(\Om))} 
    + \norm{\Gradg\Theta_k}_{L^2(0,T_k;L^2(\Ga))}
    &\le C.
\end{align}
Testing \eqref{WFD:1} with the first eigenfunction $(\zeta_1,\xi_1)$ (see \eqref{1STEIG}), we obtain
\begin{align*}
\ddt \left( \alpha \int_\Omega \phi_k(t) \dx + \int_\Ga \Psi_k(t) \dS \right)=0,
\end{align*}
which implies
\begin{equation}
    \label{cons-mass}
    \alpha \intO \phi_k(t) \dx + \intG \Psi_k(t) \dS
        = \alpha \intO \phi_0 \dx + \intG \Psi_0 \dS,
\end{equation}
for all $t\in[0,T_k]$.
Here, we have used that 
\begin{align*}
&\alpha \abs{\Omega}\mean{\mathbb{P}^\Om_{\ZZK}(\phi_0,\Psi_0)}_\Om + \abs{\Gamma}\mean{\mathbb{P}^\Ga_{\ZZK}(\phi_0,\Psi_0)}_\Ga
= \bigscp{\mathbb{P}_{\ZZK}(\phi_0,\Psi_0)}{(\alpha,1)}_{\LL^2} 
\\
&\quad = \bigscp{(\phi_0,\Psi_0)}{(\alpha,1)}_{\LL^2} 
= \alpha \abs{\Omega}\mean{\phi_0}_\Om + \abs{\Gamma}\mean{\Psi_0}_\Ga,
\end{align*}
which follows by means of \eqref{1STEIG} and \eqref{EIG:MEAN}.
Defining the constant
$$
    c_m := \frac{\alpha\abs{\Om}\mean{\phi_0}_\Om + \abs{\Ga}\mean{\Psi_0}_\Ga}{\alpha^2\abs{\Om} + \abs{\Ga}},
$$
we have $(\alpha c_m,c_m) \in \Da$ and invoking \eqref{cons-mass}, a straightforward computation yields  
\begin{align*}
    \big(\phi_k(t)-\alpha c_m, \Psi_k(t) - c_m\big) \in \Da
    \quad\text{and}\quad
    \alpha \abs{\Om}\mean{\phi_k(t)-\alpha c_m}_\Om + \abs{\Ga}\mean{\Psi_k(t) - c_m}_\Ga
    = 0
\end{align*}
for all $t\in[0,T_k]$.
Hence, applying the bulk-surface variant of the Poincar\'e inequality \eqref{bs-Poincare} (with $\beta=\alpha^2$), we deduce that
\begin{align*}
    \norm{(\phi_k,\Psi_k)}_{\LL^2}
    &\leq \norm{(\phi_k - \alpha c_m,\Psi_k - c_m)}_{\LL^2} 
        + \norm{(\alpha c_m,c_m)}_{\LL^2}
    \\[1ex]
    &\leq C_P \norm{(\Grad \phi_k, \Grad_\Ga \Psi_k)}_{\LL^2}
    +C
\end{align*}
for all $t\in[0,T_k]$. 
In view of \eqref{EST:AP:2}, we thus conclude
\begin{align}
    \label{EST:AP:4}
    \norm{\phi_k}_{L^\infty(0,T_k;H^1(\Om))} 
    + \norm{\Psi_k}_{L^\infty(0,T_k;H^1(\Ga))}
    &\le C.
\end{align}

Next, we derive a uniform estimate for $(\mu_k,\Theta_k)$ in the full $\HH^1$-norm. For this aim, we test \eqref{WFD:3} by
$(\zeta,\xi):=\mathbb{P}_{\ZZK}(\zeta^*,\xi^*)$, where $(\zeta^*,\xi^*)\in \Da$ is an arbitrary test function. 
By the growth conditions \eqref{GR:F'}--\eqref{GR:G'} from \eqref{ass:pot} and the Sobolev embeddings $H^1(\Omega)\hookrightarrow L^6(\Omega)$ and $H^1(\Ga)\hookrightarrow L^p(\Ga)$ for any $1\leq p<\infty$, we have
\begin{align}
\begin{aligned}
  \abs{\bigang{(\mu_k,\Theta_k)}{(\zeta^*,\xi^*)}_{\Da}}
  &= \abs{\bigang{(\mu_k,\Theta_k)}{(\zeta,\xi)}_{\Da}}
    \\[1ex]
  &\leq 
  \norm{\Grad\phi_k}_{L^2(\Omega)} \norm{\Grad \zeta}_{L^2(\Omega)} 
  +\norm{F'(\phi_k)}_{L^\frac65(\Omega)} \norm{\zeta}_{L^6(\Omega)}
   \\
  &\quad 
  + \alpha^2 \norm{\Gradg\Psi_k}_{L^2(\Gamma)} \norm{\Gradg \xi}_{L^2(\Gamma)} 
  + \alpha\norm{G'(\alpha\Psi_k)}_{L^2(\Gamma)} \norm{\xi}_{L^2(\Gamma)}   
  \\
  &\leq C\left( 1+ \norm{\Grad\phi_k}_{L^2(\Omega)} 
  + \norm{\phi_k}_{L^6(\Omega)}^5
  \right) \norm{(\zeta^*,\xi^*)}_{\HH^1}
   \\
  &\quad 
  +C \left(1+ \norm{\Gradg \Psi_k}_{L^2(\Gamma)} 
  + \norm{\Psi_k}_{L^{2q}(\Gamma)}^q
  \right)\norm{(\zeta^*,\xi^*)}_{\HH^1}
   \\
  &\leq 
  C\left( 1+ \norm{\phi_k}_{H^1(\Omega)}^5 
  + \norm{\Psi_k}_{H^1(\Ga)}^q
  \right) \norm{(\zeta^*,\xi^*)}_{\HH^1},
  \end{aligned}
  \end{align}
 for all $t\in[0,T_k]$.
 Taking the supremum over all $(\zeta^*,\xi^*) \in \Da$ with $\norm{(\zeta^*,\xi^*)}_{\Da} \le 1$, 
and exploiting \eqref{EST:AP:4}, we find that 
\begin{equation}
\label{EST:AP:5-2}
\norm{(\mu_k,\Theta_k)}_{L^\infty(0,T_k;\DD_\alpha')}\leq C.
\end{equation}
By a duality argument, we infer the existence of a positive constant $C_{\Da}$ such that 
\begin{equation}
\label{DUAL-EST}
\norm{(\bar\zeta,\bar\xi)}_{\LL^2}^2 
= \bigang{(\bar\zeta,\bar\xi)}{(\bar\zeta,\bar\xi)}_{\Da}
\leq 
C_{\Da} \norm{(\bar\zeta,\bar\xi)}_{\DD_\alpha'} \norm{(\Grad \bar\zeta, \Gradg \bar\xi)}_{\LL^2} + C_{\Da} \norm{(\bar\zeta,\bar\xi)}_{\DD_\alpha'}^2
\end{equation}
for all $(\bar\zeta,\bar\xi) \in \Da$.
Hence, applying \eqref{DUAL-EST} with $(\bar\zeta,\bar\xi)=(\mu_k,\Theta_k)$ and using \eqref{EST:AP:3} and \eqref{EST:AP:5-2}, we arrive at
\begin{equation}
\label{EST:AP:5}
\norm{(\mu_k,\Theta_k)}_{L^4(0,T_k;\LL^2)}\leq C.
\end{equation}
In addition, we infer from \eqref{EST:AP:3} and \eqref{EST:AP:5} that
\begin{align}
    \label{EST:AP:6}
    \norm{(\mu_k,\Theta_k)}_{L^2(0,T_k;\HH^1)} 
    &\le C.
\end{align}

Lastly, we derive uniform estimates on the time derivatives. Therefore, let $\w^*\in \bH^1_\Div(\Omega)$ and $(\zeta^*,\xi^*) \in \Da$ be arbitrary test functions. We now test \eqref{WFD:1} with $\w:=\mathbb{P}_{\UU_k}(\w^*)$. 

In three dimensions, recalling that the functions $\nu$ and $\gamma$ are bounded, we use Hölder's inequality, the continuous embeddings $H^1(\Om) \emb L^6(\Om)$ and $H^1(\Om) \emb L^4(\Ga)$ as well as the Gagliardo-Nirenberg interpolation inequality to obtain the estimate
\begin{align*}
\begin{aligned}
    \abs{ \bigang{\delt \v_k}{\w^*}_{\bH_\Div^1(\Omega)} }
    &= \abs{ \bigang{\delt \v_k}{\w}_{\bH_\Div^1(\Omega)} }
    \\[1ex]&
    \le \norm{\v_k}_{L^4(\Om)}^2 
        \norm{\nabla \w}_{L^2(\Om)}
    + C \norm{\Grad\v_k}_{L^2(\Om)} 
        \norm{\Grad\w}_{L^2(\Om)}
    \\&\qquad
    + \norm{\mu_k}_{L^6(\Om)}
        \norm{\Grad\phi_k}_{L^2(\Om)}
        \norm{\w}_{L^3(\Om)}
    + C \norm{\v_k}_{L^2(\Ga)} 
        \norm{\w}_{L^2(\Ga)}
    \\&\qquad
    + \norm{\Psi_k}_{L^4(\Ga)}
        \norm{\Grad \Theta_k}_{L^2(\Ga)}
        \norm{\w}_{L^4(\Ga)}
    \\[1ex]&
    \le C \norm{\v_k}_{L^2(\Om)}^\frac12 
      \norm{\v_k}_{H^1(\Om)}^\frac32
        \norm{\w^*}_{H^1(\Om)}
    + C \norm{\Grad\v_k}_{L^2(\Om)} 
        \norm{\w^*}_{H^1(\Om)}
    \\&\qquad
    + C \norm{\mu_k}_{H^1(\Om)}
        \norm{\Grad\phi_k}_{L^2(\Om)}
        \norm{\w^*}_{H^1(\Om)}
    + C \norm{\v_k}_{H^1(\Om)} 
        \norm{\w^*}_{H^1(\Om)}
    \\&\qquad
    + \norm{\Psi_k}_{H^1(\Ga)}
        \norm{\Theta_k}_{H^1(\Ga)}
        \norm{\w^*}_{H^1(\Om)}
\end{aligned}
\end{align*}
for all $t\in[0,T_k]$.
Taking the supremum over all $\w^*\in \bH^1_\Div(\Omega)$ with $\norm{\w^*}_{\bH^1_\Div(\Omega)} \le 1$, taking the power $\frac43$ and integrating over $[0,T_k]$, we now deduce that 
\begin{align*}
\begin{aligned}
    \norm{\delt\v_k}_{L^\frac43(0,T_k;\bH^1_\Div(\Omega)')}^\frac43 
    & \le C \norm{\v_k }_{L^\infty(0,T_k; L^2(\Om))}^\frac23 
      \norm{\v_k }_{L^2(0,T_k;H^1(\Om))}^2 
     \\&\quad
      + C\, T_k^\frac13 \norm{\v_k }_{L^2(0,T_k;H^1(\Om))}^\frac43 
     \\&\quad
     + C\, T_k^\frac13 \| \phi_k\|_{L^\infty(0,T_k; H^1(\Om))}^\frac43 \| \mu_k\|_{L^2(0,T_k;H^1(\Om))}^\frac43
     \\&\quad
     + C\, T_k^\frac13 \| \Psi_k\|_{L^\infty(0,T_k; H^1(\Ga))}^\frac43 \|\Theta_k\|_{L^2(0,T_k; H^1(\Ga))}^\frac43. 
     \end{aligned}
\end{align*}
By exploiting \eqref{EST:AP:1}, \eqref{EST:AP:4} and \eqref{EST:AP:6}, and recalling that $T_k\leq T$, we simply reach
 \begin{align}
    \label{EST:AP:7}
    \norm{\delt\v_k}_{L^\frac43(0,T_k;\bH^1_\Div(\Omega)')} \le C, \quad \text{if $d=3$}.
\end{align}
In two dimensions, arguing as above and using the Ladyzhenskaya inequality for the convective term, we first obtain
\begin{align*}
    \begin{aligned}
    \abs{ \bigang{\delt \v_k}{\w^*}_{\bH_\Div^1(\Omega)} }
    &= \abs{ \bigang{\delt \v_k}{\w}_{\bH_\Div^1(\Omega)} }
    \\[1ex]&
    \le C \norm{\v_k}_{L^2(\Om)}
      \norm{\v_k}_{H^1(\Om)}
        \norm{\w^*}_{H^1(\Om)}
    + C \norm{\Grad\v_k}_{L^2(\Om)} 
        \norm{\w^*}_{H^1(\Om)}
    \\&\qquad
    + C \norm{\mu_k}_{H^1(\Om)}
        \norm{\Grad\phi_k}_{L^2(\Om)}
        \norm{\w^*}_{H^1(\Om)}
    + C \norm{\v_k}_{H^1(\Om)} 
        \norm{\w^*}_{H^1(\Om)}
    \\&\qquad
    + \norm{\Psi_k}_{H^1(\Ga)}
        \norm{\Theta_k}_{H^1(\Ga)}
        \norm{\w^*}_{H^1(\Om)}.
\end{aligned}
\end{align*}
Thus, after taking the supremum over all $\w^*\in \bH^1_\Div(\Omega)$ with $\norm{\w^*}_{\bH^1_\Div(\Omega)} \le 1$, taking the square and integrating over $[0,T_k]$, we find
\begin{align*}
\begin{aligned}
    \norm{\delt\v_k}_{L^2(0,T_k;\bH^1_\Div(\Omega)')}^2 
    & \le C \norm{\v_k }_{L^\infty(0,T_k; L^2(\Om))}^2
      \norm{\v_k }_{L^2(0,T_k;H^1(\Om))}^2 
      + C \norm{\v_k }_{L^2(0,T_k;H^1(\Om))}^2 
     \\&\quad
     + C  \| \phi_k\|_{L^\infty(0,T_k; H^1(\Om))}^2 \| \mu_k\|_{L^2(0,T_k;H^1(\Om))}^2
     \\&\quad
     + C \| \Psi_k\|_{L^\infty(0,T_k; H^1(\Ga))}^2 \|\Theta_k\|_{L^2(0,T_k; H^1(\Ga))}^2,
     \end{aligned}
\end{align*}
which immediately entails that 
\begin{align}
    \label{EST:AP:7-2D}
    \norm{\delt\v_k}_{L^2(0,T_k;\bH^1_\Div(\Omega)')} \le C, \quad \text{if $d=2$}.
\end{align}

Concerning the bulk and surface phase-fields, we test \eqref{WFD:2} with $(\zeta,\xi):=\mathbb{P}_{\ZZK}(\zeta^*,\xi^*)$.
Recalling that the functions $\mom$ and $\mga$ (and thus also $n_\Ga$) are bounded, we use Hölder's inequality along with the continuous embeddings $H^1(\Omega) \emb L^4(\Omega)$, $H^1(\Gamma) \emb L^4(\Gamma)$ and $H^1(\Omega) \emb L^4(\Gamma)$ to infer that
\begin{align*}
\begin{aligned}
    \abs{\bigang{(\delt \phi_k,\delt \Psi_k)}{(\zeta^*,\xi^*)}_{\Da}}
    &= \abs{\bigang{(\delt \phi_k,\delt \Psi_k)}{(\zeta,\xi)}_{\Da}}
    \\[1ex]&
    \le \norm{\phi_k}_{H^1(\Om)} 
        \norm{\v_k}_{H^1(\Om)}
        \norm{(\zeta^*,\xi^*)}_{\HH^1}
        + C \norm{\mu_k}_{H^1(\Om)}
        \norm{(\zeta^*,\xi^*)}_{\HH^1}
    \\&\quad 
        + \norm{\Psi_k}_{H^1(\Ga)} 
        \norm{\v_k}_{H^1(\Omega)}
        \norm{(\zeta^*,\xi^*)}_{\HH^1}
        + C \norm{\Theta_k}_{H^1(\Ga)}
        \norm{(\zeta^*,\xi^*)}_{\HH^1}.
    \end{aligned}
\end{align*}
Hence, after taking the supremum over all $(\zeta^*,\xi^*) \in \Da$ with $\norm{(\zeta^*,\xi^*)}_{\Da} \le 1$, we take the square of the resulting inequality and integrate over $[0,T_k]$. This yields
\begin{align*}
    \norm{(\delt\phi_k,\delt\Psi_k)}_{L^2(0,T_k;\DD_\alpha')}^2
        &\le \norm{\phi_k}_{L^\infty(0,T;H^1(\Om))}^2 
        \norm{\v_k}_{L^2(0,T;H^1(\Om))}^2
        + C \norm{\mu_k}_{L^2(0,T;H^1(\Om))}^2
    \notag\\&\qquad
        + \norm{\Psi_k}_{L^\infty(0,T;H^1(\Ga))}^2 
        \norm{\v_k}_{L^2(0,T;H^1(\Om))}^2
        + C \norm{\Theta_k}_{L^2(0,T;H^1(\Ga))}^2.
\end{align*}
In view of the uniform estimates \eqref{EST:AP:1}, \eqref{EST:AP:4} and \eqref{EST:AP:6}, we thus conclude
\begin{align}
    \label{EST:AP:8}
    \norm{(\delt\phi_k,\delt\Psi_k)}_{L^2(0,T_k; \DD_\alpha')} \le C.
\end{align}

\noindent\textbf{Step 3: Extension of the approximate solution onto $[0,T]$.}
We recall from Step~1 that the coefficient vector 
$(\mathbf a^k,\mathbf b^k)^\top$ is determined as a solution of a nonlinear ODE system. Using the definition of the approximate solutions given in \eqref{DEF:APPROX} as well as the uniform a priori estimates \eqref{EST:AP:1} and \eqref{EST:AP:4}, we obtain that for any $T_k\in[0,T_k^*)$, all $t\in [0,T_k]$, and all $i\in\{1,...,k\}$, 
\begin{align*}
    |a_i^k(t)| + |b_j^k(t)| 
    &= \bigabs{ \scp{\v_k(t)}{\w_i}_{L^2(\Omega)} } 
    + \bigabs{\scp{\big(\phi_k(t),\Psi_k(t)\big)}{(\zeta_i,\xi_i)}_{\LL^2}} \\
    &\le \norm{\v_k}_{L^\infty(0,T_k;L^2(\Omega))}
    + \norm{(\phi_k,\Psi_k)}_{L^\infty(0,T_k;\LL^2)} 
    \le C.
\end{align*}
This means that the solution $(\mathbf a^k,\mathbf b^k)^\top$ is bounded on the time interval $[0,T_k^*)$ by a constant that does neither depend on $T_k$ nor on $k$.
Consequently, by the classical ODE theory, the solution can thus be extended beyond the time $T_k^*$. However, as $(\mathbf a^k,\mathbf b^k)^\top$ was assumed to be a right-maximal solution, meaning that $T_k$ was assumed to be chosen as large as possible, this is a contradiction.
We thus conclude that the solution $(\mathbf a^k,\mathbf b^k)^\top$ actually exists on the whole time interval $[0,T]$. As the coefficients $\mathbf c^k$ can be reconstructed from $(\mathbf a^k,\mathbf b^k)^\top$ via the vector-valued algebraic equation mentioned in Step~1, we further infer that the vector-valued coefficient function $(\mathbf c^k)^\top$ also exists on $[0,T]$. This directly entails that the approximate solution $(\v_k,\phi_k,\Psi_k,\mu_k,\Theta_k)$ actually exists on $[0,T]$ and satisfies the discretized weak formulation \eqref{WFD} everywhere in $[0,T]$.
Moreover, as the choice of $T_k$ did not play any role in the derivation of the uniform estimates, we further conclude that the estimates \eqref{EST:AP:1}, \eqref{cons-mass}, \eqref{EST:AP:4}, \eqref{EST:AP:5}, \eqref{EST:AP:6}, \eqref{EST:AP:7}, \eqref{EST:AP:7-2D} and \eqref{EST:AP:8} remain true when $T_k$ is replaced by $T$. In summary, we thus know that the approximate weak solution $(\v_k,\phi_k,\Psi_k,\mu_k,\Theta_k)$ satisfies the uniform estimates 
\begin{subequations}
\label{EST:AP:FINAL}
\begin{align}
    \label{EST:AP:3D}
    \begin{aligned}
    &\norm{\delt\v_k}_{L^\frac43(0,T;\bH^1_\Div(\Omega)')}
    + \norm{\v_k}_{L^\infty(0,T;L^2(\Om))} 
    + \norm{\v_k}_{L^2(0,T;H^1(\Om))} 
    + \norm{\v_k}_{L^2(0,T;L^2(\Ga))}
    \\&\;\; 
    + \norm{(\delt\phi_k,\delt\Psi_k)}_{L^2(0,T;\DD_\alpha')}
    + \norm{(\phi_k,\Psi_k)}_{L^\infty(0,T;\DD_\alpha)} 
    \\&\;\;
    + \norm{(\mu_k,\Theta_k)}_{L^4(0,T;\LL^2)}
    + \norm{(\mu_k,\Theta_k)}_{L^2(0,T;\HH^1)}
    \le C, \quad \text{if $d=3$},
    \end{aligned}
    \\[1ex]
    \label{EST:AP:2D}
    \begin{aligned}
    &\norm{\delt\v_k}_{L^2(0,T;\bH^1_\Div(\Omega)')}
    + \norm{\v_k}_{L^\infty(0,T;L^2(\Om))} 
    + \norm{\v_k}_{L^2(0,T;H^1(\Om))} 
    + \norm{\v_k}_{L^2(0,T;L^2(\Ga))}
    \\&\;\; 
    + \norm{(\delt\phi_k,\delt\Psi_k)}_{L^2(0,T;\DD_\alpha')}
    + \norm{(\phi_k,\Psi_k)}_{L^\infty(0,T;\DD_\alpha)} 
    \\&\;\;
    + \norm{(\mu_k,\Theta_k)}_{L^4(0,T;\LL^2)}
    + \norm{(\mu_k,\Theta_k)}_{L^2(0,T;\HH^1)}
    \le C, \quad \text{if $d=2$}.
    \end{aligned}
\end{align}
\end{subequations}

\noindent\textbf{Step 4: Convergence to a weak solution.}
In view of the uniform estimates \eqref{EST:AP:3D} and \eqref{EST:AP:2D}, the Banach--Alaoglu theorem and the Aubin--Lions--Simon lemma (see, e.g., \cite[Theorem~II.5.16]{boyer}) imply the existence of functions $\v$, $\phi$, $\Psi$, $\mu$ and $\theta$ 
such that
\begin{alignat}{2}
    \label{CONV:1}
    \delt\v_k &\to \delt\v
    &&\quad\text{weakly in $L^\frac43(0,T;\bH^1_\Div(\Omega)')$ if $d=3$,}
    \\
    &&&\quad\text{weakly in $L^2(0,T;\bH^1_\Div(\Omega)')$ if $d=2$,}
    \\[1ex]
    \label{CONV:2}
    \v_k &\to \v
    &&\quad\text{weakly-star in $L^\infty(0,T;\bL_\Div^2(\Omega))$,}
    \notag \\
    &&&\quad\text{weakly in $L^2(0,T;\bH^1(\Omega)) \cap L^2(0,T;\bL^2(\Gamma))$,}
    \notag \\
    &&&\quad\text{strongly in $L^2(0,T; \bL_\Div^2(\Omega))\cap C([0,T];\bH^1_\Div(\Omega)')$,}
    \\[1ex]
    \label{CONV:3}
    (\delt\phi_k,\delt\Psi_k) &\to (\delt\phi,\delt\Psi) 
    &&\quad\text{weakly in $L^2(0,T;\DD_\alpha')$,}
    \\[1ex]
    \label{CONV:4}
    (\phi_k,\Psi_k) &\to (\phi,\Psi) 
    &&\quad\text{weakly-star in $L^\infty(0,T;\Da)$,}
    \notag \\
    &&&\quad\text{strongly in $C([0,T];\HH^s)$ for all $s\in [0,1)$,}
    \\[1ex]
    \label{CONV:5}
    (\mu_k,\Theta_k) &\to (\mu,\Theta)
    &&\quad\text{weakly in $L^4(0,T;\LL^2)\cap L^2(0,T;\Da)$,}
\end{alignat}
as $k\to\infty$, along a non-relabeled subsequence. In view of \eqref{CofV}, we set $\psi:=\alpha\Psi$ and $\theta:=\alpha^{-1}\Theta$.
Using Sobolev's embedding theorem, we infer from \eqref{CONV:4} that
\begin{alignat}{2}
    \label{CONV:6}
    \phi_k &\to \phi
    &&\quad\text{strongly in $C([0,T];L^r(\Omega))$ for all $r\in [2,6)$, and a.e.~in $Q$,}
    \\[1ex]
    \label{CONV:7}
    \Psi_k &\to \Psi
    &&\quad\text{strongly in $C([0,T];L^r(\Gamma))$ for all $r\in [2,\infty)$, and a.e.~on $\Sigma$,}
\end{alignat}
as $k\to\infty$, after another subsequence extraction. Recalling the growth assumptions on $G$ and $G'$ imposed in \eqref{ass:pot}, we use the above convergences along with Lebesgue's general convergence theorem (see \cite[Section~3.25]{Alt}) to deduce that
\begin{alignat}{2}
    \label{CONV:9}
    G(\alpha\Psi_k) &\to G(\alpha\Psi)
    &&\quad\text{strongly in $L^1(\Sigma)$, and a.e.~on $\Sigma$,}
    \\
    \label{CONV:10}
    G'(\alpha\Psi_k) &\to G'(\alpha\Psi)
    &&\quad\text{strongly in $L^2(\Sigma)$, and a.e.~on $\Sigma$,}
\end{alignat}
as $k\to\infty$. Moreover, using the growth assumptions on $F'$ from \eqref{ass:pot}, the uniform estimate \eqref{EST:AP:4} and the continuous embedding $H^1(\Omega)\emb L^6(\Omega)$, we infer that $F'(\phi_k)$ is bounded in $L^{6/5}(Q)$ uniformly in $k\in\N$. Hence, there exists a function $f\in L^{6/5}(Q)$  such that $F'(\phi_k) \to f$ weakly in $L^{6/5}(Q)$ as $k\to\infty$. In view of \eqref{CONV:6}, we also have $F'(\phi_k) \to F'(\phi)$ a.e.~in $Q$ as $k\to\infty$. By a convergence principle based on Egorov's theorem (see \cite[Proposition~9.2c]{DiBenedetto}), we infer $f=F'(\phi)$. We have thus shown that
\begin{alignat}{2}
    \label{CONV:11}
    F'(\phi_k) &\to F'(\phi)
    &&\quad\text{weakly in $L^{6/5}(Q)$, and a.e.~in $Q$,}
\end{alignat}
as $k\to\infty$.
Furthermore, it follows from \eqref{CONV:6} and \eqref{CONV:7} that
\begin{alignat}{3}
    \label{CONV:12}
    \nu(\phi_k) &\to \nu(\phi)
    &\quad\text{and}\quad
    \mom(\phi_k) &\to \mom(\phi)
    &&\quad\text{a.e. in $Q$,}
    \\
    \label{CONV:13}
    \gamma(\alpha\Psi_k) &\to \gamma(\alpha\Psi)
    &\quad\text{and}\quad
    \nga(\Psi_k) &\to \nga(\Psi)
    &&\quad\text{a.e. in  $\Sigma$,}
\end{alignat}
as $k\to\infty$. Moreover, we also have from \eqref{ass:mob} and \eqref{ass:vis-fric} that
 \begin{alignat}{3}
    \label{CONV:14}
    \nu(\phi_k) &\to \nu(\phi)
    &\quad\text{and}\quad
    \mom(\phi_k) &\to \mom(\phi)
    &&\quad\text{in $L^r(Q)$,}
    \\
    \label{CONV:15}
    \gamma(\alpha\Psi_k) &\to \gamma(\alpha\Psi)
    &\quad\text{and}\quad
    \nga(\Psi_k) &\to \nga(\Psi)
    &&\quad\text{in $L^r(\Sigma)$,}
\end{alignat}
for all $r\in [2,\infty)$, as $k\to\infty$. Lastly, by \eqref{ass:mob} and \eqref{ass:vis-fric}, the weak-strong convergence principle and the Lebesgue convergence theorem, we infer that
\begin{alignat}{3}
    \label{CONV:DISS:1}
    \nu(\phi_k)\; \D\v_k &\to \nu(\phi)\; \D\v
    &&\quad\text{weakly in $L^2(Q)$,}
    \\
    \label{CONV:DISS:2}
    \mom(\phi_k)\; \Grad\phi_k &\to \mom(\phi)\; \Grad\phi
    &&\quad\text{weakly in $L^2(Q)$,}
    \\
    \label{CONV:DISS:3}
    \gamma(\alpha \Psi_k)\; \v_k &\to \gamma(\alpha \Psi)\; \v
    &&\quad\text{weakly in $L^2(\Sigma)$,}
    \\
    \label{CONV:DISS:4}
    \nga(\Psi_k)\; \Gradg\Psi_k &\to \nga(\Psi)\;\Gradg\Psi
    &&\quad\text{weakly in $L^2(\Sigma)$,}
\end{alignat}
as $k\to\infty$.

We now multiply all equations of the discretized weak formulation \eqref{WFD} by an arbitrary test function $\sigma\in C_0^\infty([0,T])$ and integrate the resulting equations with respect to time from $0$ to $T$.
Using the convergences \eqref{CONV:1}--\eqref{CONV:DISS:4}, we can then pass to the limit $k\to\infty$ to conclude that
\begin{subequations}
    \label{WFD*}
    \begin{align}
    \label{WFD*:1}
    &\begin{aligned}
    &\int_0^T \ang{\delt \v}{\w_i}_{\bH_\Div^1(\Omega)} \;\sigma\dt
        - \int_Q (\v \otimes \v) \Grad\w_i \;\sigma\dx\dt 
        + \int_Q 2\nu(\phi) \D\v : \D\w_i \;\sigma\dx\dt
    \\
    &\quad
    = \int_Q \mu \Grad\phi \cdot \w_i \;\sigma\dx\dt
    - \int_\Sigma \gamma(\alpha\Psi) \v\cdot \w_i \;\sigma
        + \Psi\Gradg\Theta \cdot \w_i \;\sigma\dS\dt,
    \end{aligned}
    \\[1ex]
    \label{WFD*:2}
    &\begin{aligned}
    &\int_0^T \bigang{(\delt \phi,\delt \Psi)}{(\zeta_i,\xi_i)}_{\Da} \;\sigma\dt
        -  \int_Q \phi\v \cdot \Grad\zeta_i \;\sigma\dx\dt
        - \int_\Sigma \Psi\v \cdot \Gradg\xi_i \;\sigma\dS\dt
    \\     
    &\quad= -  \int_Q m_\Om(\phi) \Grad \mu \cdot \Grad \zeta_i \;\sigma\dx\dt
    	-\int_\Sigma \nga(\Psi) \Gradg \Theta \cdot \Gradg \xi_i \;\sigma\dS\dt,
    \end{aligned}	
    \\[1ex]
    \label{WFD*:3}
    &\begin{aligned}
    &\int_Q \mu \, \zeta_i \;\sigma\dx\dt     
        + \int_\Sigma \Theta \, \xi_i \;\sigma\dS\dt
    \\
    &\quad 
    = \int_Q \Grad\phi \cdot \Grad \zeta_i \;\sigma
        + F'(\phi) \zeta_i \;\sigma\dx\dt
        + \int_\Sigma \alpha^2 \Gradg\Psi \cdot \Gradg \xi_i \;\sigma 
        + \alpha G'(\alpha \Psi)\xi_i \;\sigma\dS\dt
    \end{aligned}
    \end{align}
\end{subequations}
hold for all $i\in\N$ and all test functions $\sigma \in C_0^\infty([0,T])$. Since $\mathrm{span}\{\w_i \,\vert\, i\in\N \}$ is dense in $\bH^1_\Div(\Omega)$ and $\mathrm{span}\{(\zeta_i,\xi_i) \,\vert\, i\in\N \}$ is dense in $\Da$, this proves that the quintuplet $(\v,\phi,\Psi,\mu,\Theta)$ satisfies the weak formulation \eqref{WF} for all test functions $\w\in \bH^1_\Div(\Omega)$ and $(\zeta,\xi),(\eta,\vartheta)\in\Da$. Reversing the change of variables \eqref{CofV}, we conclude that Definition~\ref{DEF:WS}\ref{DEF:WS:WF} is fulfilled.

Integrating the weak formulation \eqref{WF:2} with respect to time from $0$ to an arbitrary $t\in[0,T]$ and choosing $(\zeta,\xi)\equiv (\alpha,1)$, we obtain (cf. \eqref{cons-mass})
\begin{align*}
    & \alpha \intO \phi(t) \dx + \intG \Psi(t) \dS 
        =\alpha \intO \phi_0 \dx + \intG \Psi_0 \dS 
\end{align*}
By the change of variables \eqref{CofV}, this proves the mass conservation law \eqref{WS:MASS} meaning that the condition in Definition~\ref{DEF:WS}\ref{DEF:WS:MASS} is fulfilled.

Furthermore, it follows from Strauss' lemma (see \cite[Corollary~2.1]{Strauss}) that 
\begin{equation}
    \label{CONT:WEAK}
    \v \in C_w([0,T];\bL_\Div^2(\Omega)),\quad \text{and}\quad (\phi,\Psi)\in C_w([0,T]; \DD_\alpha).
\end{equation} 
In the two dimensional case, in view of the enhanced time integrability of $\partial_t \v$, we even obtain $\v \in C([0,T];\bL_\Div^2(\Omega))$. In summary, we conclude that the regularity condition \eqref{WS:REG:alpha} is satisfied and thus, by the change of variables \eqref{CofV}, Definition~\ref{DEF:WS}\ref{DEF:WS:REG} is fulfilled.

In view of \eqref{WFD:INI:1} and \eqref{WFD:INI:2}, we have
\begin{alignat}{2}
    \v_k(0) &\to \v_0 
    &&\quad\text{strongly in $\bL^2_\Div(\Omega)$},\\
    \big(\phi_k(0),\Psi_k(0)\big) &\to (\phi_0,\Psi_0)
    &&\quad\text{strongly in $\LL^2$},
\end{alignat}
as $k\to\infty$,
due to the convergence properties of the projections $\mathbb{P}_{\ZZK}$ and $\mathbb{P}_{\mathcal U_k}$.
On the other hand, we infer from \eqref{CONV:2} and \eqref{CONV:4} that
\begin{alignat}{2}
    \v_k(0) &\to \v(0) 
    &&\quad\text{strongly in $\bH^1_\Div(\Omega)'$},\\
    \big(\phi_k(0),\Psi_k(0)\big) &\to \big(\phi(0),\Psi(0)\big)
    &&\quad\text{strongly in $\LL^2$},
\end{alignat}
as $k\to\infty$. In particular, we know from \eqref{CONT:WEAK} that $\v(0) \in \bL^2_\Div(\Omega)$ and $(\phi(0),\Psi(0)) \in \DD_\alpha$. We thus deduce $\v(0) = v_0$ a.e.~in $\Omega$, $\phi(0) = \phi_0$ a.e.~in $\Omega$ and $\Psi(0) = \Psi_0$ a.e.~on $\Gamma$. By the change of variables \eqref{CofV}, we conclude that Definition~\ref{DEF:WS}\ref{DEF:WS:INI} is fulfilled.

To verify the weak energy dissipation law, let $\sigma \in C_0^\infty(0,T)$ be an arbitrary non-negative test function.
Multiplying \eqref{DISC:ENERGY} by $\sigma$ and integrating on $[0,T]$, we obtain
\begin{align}
\label{DISC:EN:test}
\begin{aligned}
    &\int_0^T E\big(\v_k(t),\phi_k(t),\alpha\Psi_k(t)) \, \sigma(t) \, \dt
    \\
    &\qquad
    + \int_0^T \int_0^t \intO \left( 2\nu(\phi_k)\, \abs{\D\v_k}^2 + m_\Om(\phi_k) |\Grad \mu_k|^2 \right) \, \sigma(t) \dx \dtau \dt
     \\
    &\qquad + \int_0^T \int_0^t \intG \left( \gamma(\alpha\Psi_k) \abs{\v_k}^2 + \nga(\Psi_k) \abs{\Gradg \Theta_k}^2\right) \, \sigma(t) \dS \dtau \dt 
     \\[1ex]
    &= \int_0^T E\big(\v_k(0),\phi_k(0),\alpha\Psi_k(0)) \, \sigma(t) \dt.
    \end{aligned}
\end{align}
By the strong convergence in \eqref{CONV:2}, it easily follows that 
\begin{align}
\label{CONV:EN:1}
     \int_0^T \left(\frac12 \norm{\v_k}_{L^2(\Omega)}^2\right)\sigma(t) \dt
     \to
     \int_0^T \left(\frac12 \norm{\v}_{L^2(\Omega)}^2\right)\sigma(t) \dt,
\end{align}
as $k \to \infty$.
Similarly, by \eqref{CONV:9}, we have
\begin{align}
\label{CONV:EN:2}
    \int_0^T \left(\int_\Ga G(\alpha\Psi_k)\dS \right) \sigma(t) \dt
    \to
    \int_0^T \left( \int_\Ga G(\alpha\Psi)\dS \right) \sigma(t) \dt,
\end{align}
as $k \to \infty$.
By means of \eqref{CONV:4}, we also infer 
    \begin{align}
    \label{CONV:EN:3}
    \begin{aligned}
    &\int_0^T \left( \frac12\norm{\nabla \phi}_{L^2(\Omega)}^2 +\frac12\norm{\alpha \nabla_\Ga \Psi}_{L^2(\Ga)}^2 \right) \sigma(t) \dt
     \\
    &\quad \leq \underset{k\to\infty}{\lim\inf}\;
    \int_0^T \left( \frac12\norm{\nabla \phi_k}_{L^2(\Omega)}^2 +\frac12\norm{\alpha \nabla_\Ga \Psi_k}_{L^2(\Ga)}^2 \right) \sigma(t) \dt. 
    \end{aligned}
    \end{align}
In addition, in light of \eqref{ass:pot} and \eqref{CONV:6}, Fatou's lemma implies 
    \begin{align}
    \label{CONV:EN:4}
    \int_0^T \left( \intO F(\phi) \dx \right) \sigma(t) \dt
    \leq \underset{k\to\infty}{\lim\inf}\;
    \int_0^T \left( \intO F(\phi_k) \dx \right) \sigma(t) \dt. 
    \end{align}
Combining \eqref{CONV:EN:1}--\eqref{CONV:EN:4}, we obtain 
\begin{align}
\label{LIMINF:E}
    \int_0^T E\big(\v(t),\phi(t),\alpha\Psi(t)\big)\, \sigma(t) \dt
    \le \underset{k\to\infty}{\lim\inf}\;  \int_0^T E\big(\v_k(t),\phi_k(t),\alpha\Psi_k(t)\big) \, \sigma(t) \dt.
\end{align}
Thanks to \eqref{CONV:DISS:1}--\eqref{CONV:DISS:4} and the weak lower-semicontinuity of norms with respect to weak convergence, another application of Fatou's lemma entails that 
\begin{align}
\label{LIMINF:D}
\begin{aligned}
&\int_0^T \int_0^t \intO \left( 2\nu(\phi)\, \abs{\D\v}^2 + m_\Om(\phi) |\Grad \mu|^2 \right) \, \sigma(t) \dx \dtau \dt
 \\
&\quad + \int_0^T \int_0^t \intG \left( \gamma(\alpha \Psi) \abs{\v}^2 + \nga(\Psi) \abs{\Gradg \Theta}^2\right) \, \sigma(t) \dS \dtau \dt 
 \\
&\leq \underset{k\to\infty}{\lim\inf}
\left(\int_0^T \int_0^t \intO \left( 2\nu(\phi_k)\, \abs{\D\v_k}^2 + \mom(\phi_k) |\Grad \mu_k|^2 \right) \, \sigma(t) \dx \dtau \dt \right.
 \\
&\quad \left. + \int_0^T \int_0^t \intG \left( \gamma(\alpha \Psi_k) \abs{\v_k}^2 + \nga(\Psi_k) \abs{\Gradg \Theta_k}^2\right) \, \sigma(t) \dS \dtau \dt \right).
\end{aligned}
\end{align}
Recalling the growth conditions on $F$ and $G$ (see \eqref{ass:pot}), we use the initial conditions \eqref{WFD:INI:1} and \eqref{WFD:INI:2} as well as the convergence properties of the projections $P_{\UU_k}$ and $P_{\ZZK}$ along with  Lebesgue's general convergence theorem to infer
\begin{align}
    \label{LIM:E0}
    E\big(\v_0,\phi_0,\alpha \Psi_0\big)  
    = \underset{k\to\infty}{\lim}\; E\big(\v_k(0),\phi_k(0),\alpha \Psi_k(0)\big). 
\end{align}
Therefore, we conclude from \eqref{EST:EK}, \eqref{DISC:EN:test}, \eqref{LIMINF:E}, \eqref{LIMINF:D} and \eqref{LIM:E0} that
\begin{align}
\label{EN:DISS:1}
\begin{aligned}
&\int_0^T E\big(\v(t),\phi(t),\alpha \Psi(t)\big) \, \sigma(t) \, \dt
\\
&\qquad + \int_0^T \int_0^t \intO \left( 2\nu(\phi)\, \abs{\D\v}^2 + m_\Om(\phi) |\Grad \mu|^2 \right) \, \sigma(t) \dx \dtau \dt
\\
&\qquad + \int_0^T \int_0^t \intG \left( \gamma(\alpha \Psi) \abs{\v}^2 + \mga(\Psi) \abs{\Gradg \Theta}^2\right) \, \sigma(t) \dS \dtau \dt 
 \\[1ex]
& \leq \int_0^T E\big(\v_0,\phi_0,\alpha \Psi_0\big)  \, \sigma(t) \dt,
\end{aligned}
\end{align}
which, in turn, implies that
\begin{align}
\label{EN:DISS:2}
\begin{aligned}
&E\big(\v(t),\phi(t),\alpha\Psi(t)\big)  
+ \int_0^t \intO \left( 2\nu(\phi)\, \abs{\D\v}^2 + m_\Om(\phi) |\Grad \mu|^2 \right)  \dx \dtau 
 \\
&\quad + \int_0^t \intG \left( \gamma(\alpha\Psi) \abs{\v}^2 + \nga(\Psi) \abs{\Gradg \Theta}^2\right) \dS \dtau
\leq  E\big(\v_0,\phi_0,\alpha \Psi_0\big) 
\end{aligned}
\end{align}
for almost all $t \in [0,T]$. 

It remains to show that the energy inequality \eqref{EN:DISS:2} actually holds true for all $t \in [0,T]$. First, we notice that the both integral terms in \eqref{EN:DISS:2} depend continuously on time. 
Recalling the growth assumptions on $G'$, we further derive the estimate
\begin{align*}
    \label{G:CONT}
    \begin{aligned}
    \left| \intG \big(G(\alpha\Psi(t))-G(\alpha\Psi(s) \big) \right|
    &= \intG \int_0^1 G'(\tau \alpha \Psi(t)+(1-\tau)\alpha\Psi(s)) \dtau\, 
    \alpha \left(\Psi(t)-\Psi(s)\right) \dS
     \\
    &\leq C \intG \left(1+ |\Psi(t)|^{q-1}+|\Psi(s)|^{q-1}\right) |\Psi(t)-\Psi(s)| \dS
     \\
    &\leq C \left( 1+ \norm{\Psi(t)}_{H^1(\Gamma)}^{q-1}+\norm{\Psi(s)}_{H^1(\Gamma)}^{q-1}
    \right) \norm{\Psi(t)-\Psi(s)}_{L^2(\Gamma)},
    \end{aligned}
\end{align*}
for all $s,t\in [0,T]$. This implies that the function $t\mapsto \intG G(\alpha\Psi(t))\,\mathrm dS$ is continuous on $[0,T]$.
Furthermore, recalling $\phi \in C([0,T]; L^2(\Omega))$ as well as \eqref{CONT:WEAK}, we deduce that the functions
\begin{align*}
    t \mapsto \norm{\v(t)}_{L^2(\Omega)}^2,
    \quad
    t \mapsto \norm{\nabla \phi(t)}_{L^2(\Omega)}^2,
    \quad
    t \mapsto \norm{\alpha \Gradg \Psi(t)}_{L^2(\Ga)}^2,
    \quad
    t \mapsto \intO F\big(\phi(t)\big) \dx
\end{align*}
are lower semicontinuous on $[0,T]$. Here, for the first three functions, we used the weak lower semicontinuity of the respective norms. For the fourth function, we applied Fatou's lemma.

Combining all these properties, it follows that the weak solution $(\v,\phi,\Psi,\mu,\Theta)$
satisfies the weak energy dissipation law \eqref{EN:DISS:2} for all times $t\in[0,T]$. Hence, reversing the change of variables \eqref{CofV}, we conclude that Definition~\ref{DEF:WS}\ref{DEF:WS:ENERGY} is fulfilled. 

We have thus shown that the quintuplet $(\v,\phi,\psi,\mu,\theta)$ is a weak solution of system \eqref{NSCH:MD} in the sense of Definition~\ref{DEF:WS}.

\noindent\textbf{Step 5: Higher regularity.}
We now assume that $\Omega$ is of class $C^2$ and, if $d=3$, we further assume $p<6$. 
Thanks to the regularity theory for second order eigenvalue problems in \cite[Proposition~4.1]{knopf-liu}, we deduce from the ansatz \eqref{DEF:PHIPSIK} that $(\phi_k,\Psi_k)\in L^2(0,T;\Da\cap \HH^2)$ for all $k\in \mathbb{N}$.
In view of \eqref{WFD:3}, the pair $\big(\phi_k(t),\Psi_k(t)\big)$ is a weak solution of the elliptic problem
\begin{subequations}
\label{EQ:ELPR}
\begin{alignat}{2}
    -\Lap \phi_k(t) &= f_k(t)
    &&\quad\text{in $\Omega$,}\\
    -\alpha^2 \Lapg \Psi_k(t) + \alpha \deln \phi_k(t) &= g_k(t)
    &&\quad\text{on $\Gamma$,}\\
    \phi_k(t)\vert_\Gamma &= \alpha \Psi_k(t)
    &&\quad\text{on $\Gamma$,}
\end{alignat}
\end{subequations}
in the sense of \cite[Definition~3.1]{knopf-liu}, where
\begin{align*}
    \big(f_k(t),g_k(t)\big) 
    := \big(\mu_k(t),\Theta_k(t)\big) 
    - \mathbb{P}_{\ZZK} \big( F'(\phi_k(t)) , \alpha G'(\alpha \Psi_k(t)) \big)
\end{align*}
for almost all $t\in [0,T]$. We now aim to prove uniform estimates that are independent of $k$ by exploiting \eqref{EST:AP:3D} and \eqref{EST:AP:2D}. For this reason, the symbol $C$ will stand for a generic constant independent of $k$.

We first observe that 
$$
\bignorm{\mathbb{P}_{\ZZK} \big( F'(\phi_k(t)) , \alpha G'(\alpha \Psi_k(t)) \big) }_{\LL^2}^2
\leq 
\bignorm{ \big( F'(\phi_k(t)) , \alpha G'(\alpha \Psi_k(t)) \big) }_{\LL^2}^2.
$$
Recalling \eqref{ass:pot} and using Sobolev's embedding theorem, we derive the estimate
\begin{align}
	\label{EST:REG:G'}
	\bignorm{\alpha G'\big(\alpha \Psi_k(t)\big)}_{L^2(\Gamma)}
	&\le C + C \norm{\Psi_k(t)}_{H^1(\Gamma)}^{q-1} \le C,
\end{align}
for almost all $t\in [0,T]$.
Let now $\epsilon>0$ be arbitrary. Without loss of generality, we assume $p\in (4,6)$.
Using the Gagliardo--Nirenberg inequality and Young's inequality, we observe that 
\begin{align}
\label{EST:F':L2} 
\begin{aligned}
   \bignorm{F'\big(\phi_k(t)\big)}_{L^2(\Omega)}
   &\leq C + C \norm{\phi_k(t)}_{L^{2(p-1)}(\Omega)}^{p-1}
   \\
   &\leq C + C \norm{\phi_k(t)}_{L^6(\Omega)}^{\frac{p+2}{2}} \norm{\phi_k(t)}_{H^2(\Omega)}^{\frac{p-4}{2}}
   \\
   &\leq C +   \epsilon \norm{\phi_k(t)}_{H^2(\Omega)}.
   \end{aligned}
\end{align}
In particular, this yields $\big(f_k(t),g_k(t)\big) \in \LL^2$ for almost all $t\in [0,T]$.
Introducing the functions
$$
    \psi_k := \alpha\Psi_k
    \quad\text{and}\quad
    \theta_k := \alpha^{-1} \Theta_k
    \quad\text{for any $k\in\N$,}
$$
we infer that for almost all $t\in[0,T]$, the pair $\big(\phi_k(t),\psi_k(t)\big) \in L^2(0,T;\DD_1\cap \HH^2)$ is a weak solution of the elliptic system
\begin{subequations}
\label{EQ:ELPR*}
\begin{alignat}{2}
    -\Lap \phi_k(t) &= f_k(t)
    &&\quad\text{in $\Omega$,}\\
    - \Lapg \psi_k(t) + \deln \phi_k(t) &= \alpha^{-1} g_k(t)
    &&\quad\text{on $\Gamma$,}\\
    \phi_k(t)\vert_\Gamma &= \psi_k(t)
    &&\quad\text{on $\Gamma$,}
\end{alignat}
\end{subequations}
with $\big(f_k(t),\alpha^{-1} g_k(t)\big) \in \LL^2$.
Hence, by regularity theory for elliptic problems with bulk-surface coupling (see \cite[Theorem~3.3]{knopf-liu}), we deduce that $\big(\phi_k(t),\psi_k(t)\big) \in \HH^2$ with
\begin{align*}
    &\bignorm{\big(\phi_k(t),\psi_k(t)\big)}_{\HH^2}^2 
    \le C \norm{\big(f_k(t),\alpha^{-1} g_k(t)\big)}_{\LL^2}^2
    \\
    &\quad \le C \Big( \norm{\mu_k(t)}_{L^2(\Omega)}^2 
        + \bignorm{F'\big(\phi_k(t)\big)}_{L^2(\Omega)}^2 
        + \norm{\theta_k(t)}_{L^2(\Gamma)}^2 
        + \bignorm{G'\big(\psi_k(t)\big)}_{L^2(\Gamma)}^2 \Big)
    \\
    &\quad \le C + C \norm{\mu_k(t)}_{L^2(\Omega)}^2 + C\norm{\theta_k(t)}_{L^2(\Gamma)}^2 + C\epsilon^2  \bignorm{\big(\phi_k(t),\psi_k(t)\big)}_{\HH^2}^2
\end{align*}
for almost all $t\in [0,T]$. 
Choosing $\epsilon$ sufficiently small, we can absorb the $\epsilon$-dependent term on the right-hand side by the left-hand side. Integrating the resulting estimate with respect to $t$ from $0$ to $T$, we derive that
\begin{align}
\label{REG:L2H2}
\int_0^T \norm{(\phi_k(t),\psi_k(t))}_{\HH^2}^2 \dt \leq C.
\end{align}
Recalling the convergence properties established in Step~4, we now pass to the limit $k\to \infty$. By means of the Banach--Alaoglu theorem, we infer from \eqref{REG:L2H2} that $\big(\phi,\psi\big) \in L^2(0,T;\HH^2)$. We further obtain
\begin{align*}
    \big(f_k,\alpha^{-1} g_k\big)
    \to \big( \mu - F'(\phi) , \theta - G'(\psi) \big)
    \quad\text{weakly in $L^{6/5}(Q) \times L^2(\Sigma)$}
\end{align*}
as $k\to\infty$.
Hence, by passing to the limit in \eqref{EQ:ELPR*}, we eventually conclude that \eqref{EQ:MUTH:STRG} is fulfilled in the strong sense.

We now suppose that $\Omega$ is even of class $C^3$ and, if $d=3$, we further assume $p<6$. 
In the following, we merely proceed formally by directly considering the solution $(\phi,\psi)$ of \eqref{EQ:MUTH:STRG}. Nevertheless, the reasoning below can be rigorously justified by 
a suitable approximation argument.

Let $\epsilon>0$ be arbitrary and without loss of generality, we merely consider the case 
$p \in [5,6)$.
Applying the Gagliardo--Nirenberg inequality and the Agmon inequality, we find
\begin{align}
    \label{EST:F'':L2}
   \begin{aligned}   
   \bignorm{F''\big(\phi(t)\big) \nabla \phi(t)}_{L^2(\Omega)}
   &\leq C \norm{\nabla \phi(t)}_{L^2(\Omega)} 
   + C \norm{|\phi(t)|^{p-2} \nabla \phi(t)}_{L^2(\Omega)}
   \\
   &\leq C + C \norm{\phi(t)}_{L^{2(p-2)}(\Omega)}^{p-2} \norm{\nabla \phi(t)}_{L^\infty(\Omega)}
   \\
   &\leq C + C \norm{\phi(t)}_{L^6(\Omega)}^{\frac{p+1}{2}} \norm{\phi(t)}_{H^2(\Omega)}^{\frac{p-5}{2}+\frac12} 
   \norm{\phi(t)}_{H^3(\Omega)}^\frac12
   \\
   &\leq C 
   + C \norm{\phi(t)}_{H^2(\Omega)}^\frac{p-4}{2} 
   \norm{\phi(t)}_{H^3(\Omega)}^\frac12
   \\
   &\leq C
   + C \norm{\phi(t)}_{H^1(\Omega)}^\frac{p-4}{4}
   \norm{\phi(t)}_{H^3(\Omega)}^{\frac{p-4}{4}+\frac12}
   \\
   &\leq C + 
   C \norm{\phi(t)}_{H^3(\Omega)}^\frac{p-2}{4}
   \\
   &\leq C 
   + \varepsilon \norm{\phi(t)}_{H^3(\Omega)}.
   \end{aligned}
\end{align}
for almost all $t\in [0,T]$. Using Hölder's inequality and the estimates \eqref{EST:F':L2} and \eqref{EST:F'':L2}, we have
\begin{align}
\label{EST:REG:F':H1}
\begin{aligned}
	\bignorm{F'\big(\phi(t)\big)}_{H^1(\Omega)}^2
	&= \bignorm{F'\big(\phi(t)\big)}_{L^2(\Omega)}^2 
		+ \bignorm{F''\big(\phi(t)\big) \Grad \phi(t)}_{L^2(\Omega)}^2
	\\
	& \le C + C\epsilon \norm{\phi(t)}_{H^3(\Omega)}^2 
	\end{aligned}
\end{align}
for almost all $t\in [0,T]$.
Proceeding similarly, we derive the estimate
\begin{align}
\label{EST:G'':L2}
	\bignorm{G''\big(\phi(t)\big) \nabla \phi(t)}_{L^2(\Omega)}
	\leq C 
   + \epsilon \norm{\psi(t)}_{H^3(\Gamma)},
\end{align}
which leads to 
\begin{align}
\label{EST:REG:G':H1}
	\bignorm{G'\big(\phi(t)\big)}_{H^1(\Omega)}^2
	\le C + C\epsilon \norm{\psi(t)}_{H^3(\Gamma)}^2 
\end{align}
for almost all $t\in [0,T]$.
We thus obtain that $\big(f(t),g(t)\big) \in \HH^1$ 
and by regularity theory for elliptic problems with bulk-surface coupling (see \cite[Theorem~3.3]{knopf-liu}), we infer that $\big(\phi(t),\psi(t)\big) \in \HH^3$ with
\begin{align*}
    &\bignorm{\big(\phi(t),\psi(t)\big)}_{\HH^3}^2 
    \le C \bignorm{\big(f(t),g(t)\big)}_{\HH^1}^2
    \\
    &\quad \le C \Big( \norm{\mu(t)}_{H^1(\Omega)}^2 
        + \bignorm{F'\big(\phi(t)\big)}_{H^1(\Omega)}^2 
        + \norm{\theta(t)}_{H^1(\Gamma)}^2 
        + \bignorm{G'\big(\psi(t)\big)}_{H^1(\Gamma)}^2 \Big)
    \\
    &\quad \le C +C\norm{\mu(t)}_{H^1(\Omega)}^2
    + C \norm{\theta(t)}_{H^1(\Gamma)}^2 + C\epsilon^2 \bignorm{\big(\phi(t),\psi(t)\big)}_{\HH^3}^2
\end{align*}
for almost all $t\in [0,T]$. 
Choosing $\epsilon$ sufficiently small, we can absorb the $\epsilon$-dependent term on the right-hand side by the left-hand side.
Then, integrating the resulting estimate with respect to $t$ from $0$ to $T$, we eventually conclude that $(\phi,\psi) \in L^2(0,T;\HH^3)$.

This means that all assertions are established and thus, the proof is complete. 
\end{proof}

\begin{proof}[Proof of Theorem~\ref{THM:UNI}]
We recall that the existence of a weak solution already follows from Theorem~\ref{THM:WS}. As the functions $\nu$, $\mom$ and $\mga$ are assumed to be constant, we assume, without loss of generality, that $\nu\equiv 1$, $\mom\equiv 1$ and $\mga\equiv 1$. In the following, we write $C$ to denote generic positive constants depending only on $\Omega$, the parameters of the system \eqref{NSCH:MD} and the norms of the initial data.

To establish the continuous dependence estimate \eqref{cont-dep*}, we now suppose that $(\v_1,\phi_1,\psi_1,\mu_1,\theta_1)$ and $(\v_2,\phi_2,\psi_2,\mu_2,\theta_2)$ are two weak solutions corresponding to the 
initial data $(\v_{0}^1, (\phi_0^1, \psi_0^1))$, $(\v_{0}^2, (\phi_0^2, \psi_0^2)) \in \bL^2_\Div(\Omega)\times\DD_1$, respectively. Once again, we employ the change of variables \eqref{CofV} introduced in the proof of Theorem \ref{THM:WS}. Therefore, for $\alpha = \sqrt{\beta}$ and $i=1,2$, we define
\begin{alignat}{2}
\label{CofV*}
\left\{\;
\begin{aligned}
    \Psi_i &:= \alpha^{-1} \psi_i,
    &&\quad\text{i.e.,}\quad
    \psi_i = \alpha\Psi_i,
    \\
    \Psi_0^i &:= \alpha^{-1} \psi_0^i,
    &&\quad\text{i.e.,}\quad
    \psi_0^i = \alpha\Psi_0^i,
    \\
    \Theta_i &:= \alpha \theta_i
    &&\quad\text{i.e.,}\quad
    \theta_i = \alpha^{-1} \Theta_i,
    \\
    \nga(\Psi_i) &:= \alpha^{-2} \mga(\psi) \equiv \alpha^{-2}.
    \end{aligned}
\right.
\end{alignat}
We further recall that the corresponding weak solutions $(\v_i,\phi_i,\Psi_i,\mu_i,\Theta_i)$ exhibit the regularity \eqref{WS:REG:alpha} for $i=1,2$. Since $d=2$ and $\Omega$ is of class $C^3$, we further have 
$(\phi_i,\Psi_i) \in L^2(0,T;\HH^3)$ for $i=1,2$.
Now, we set
$$
(\v,\phi,\Psi,\mu,\Theta)= (\v_1-\v_2,\phi_1-\phi_2,\Psi_1-\Psi_2,\mu_1-\mu_2,\Theta_1-\Theta_2).
$$
In view of their construction in the proof of Theorem~\ref{THM:WS}, the functions $(\v_i,\phi_i,\Psi_i,\mu_i,\Theta_i)$ satisfy the weak formulation \eqref{WF:alpha} for $i=1,2$. We thus obtain
\begin{subequations}
\label{WF-U}
\begin{align}
\label{WF:NS}
&\begin{aligned}
&\ang{\delt \v}{\w}_{\bH_\Div^1(\Omega)} 
- \intO (\v_1 \otimes \v) : \Grad\w \dx
 -\intO (\v \otimes \v_2) : \Grad\w \dx 
 + \intO \Grad\v : \Grad\w \dx
\\
&\quad
= \intO \mu \Grad\phi_1 \cdot \w \dx 
+\intO \mu_2 \Grad\phi \cdot \w \dx
- \intG \gamma(\alpha \Psi_1) \v\cdot \w \dS
\\
&\qquad 
- \intG \left(\gamma(\alpha\Psi_1)-\gamma(\alpha \Psi_2)\right) \v_2\cdot \w \dS
-\intG \Psi_1\Gradg \Theta \cdot \w \dS
- \intG \Psi\Gradg\Theta_2 \cdot \w \dS,
\end{aligned}
\\[1ex]
\label{WF:CH}
&\begin{aligned}
    &\bigang{(\delt \phi,\delt \Psi)}{(\zeta,\xi)}_{\DD_\alpha}
        - \intO \phi_1\v \cdot \Grad\zeta \dx
        - \intO \phi\v_2 \cdot \Grad\zeta \dx
        - \intG \Psi_1\v \cdot \Gradg\xi \dS
        \\     
        &\quad - \intG \Psi\v_2 \cdot \Gradg\xi \dS
        = 
        - \intO  \Grad \mu \cdot \Grad \zeta \dx
        	-\intG \frac{1}{\alpha^2} \Gradg \Theta \cdot \Gradg \xi \dS
        	\end{aligned}
        \end{align}
        \end{subequations}
        for all test functions $\w\in \bH^1_\Div(\Omega)$, $(\zeta,\xi)\in\DD_\alpha$, where
 \begin{subequations}
    \label{EQ:UNIQ:STRG}
    \begin{alignat}{2}
    \label{EQ:UNIQ:STRG:A}
        \mu &= -\Lap \phi + F'(\phi_1)-F'(\phi_2) &&\quad\text{a.e.~in $Q$},\\
    \label{EQ:UNIQ:STRG:B}
        \Theta &= - \alpha^2 \Lapg \Psi + \alpha G'(\alpha \Psi_1)-\alpha G'(\alpha \Psi_2) 
        + \alpha \deln\phi &&\quad\text{a.e.~on $\Sigma$}.
    \end{alignat}
    \end{subequations}
Taking $\w=\v$ in \eqref{WF:NS}, the classical chain rule formula that yields
\begin{align}
\label{Test:NS}
\begin{aligned}
&\frac12 \ddt \norm{\v}_{L^2(\Omega)}^2 
+ \norm{\Grad \v}_{L^2(\Omega)}^2 
+\gamma_0 \norm{\v}_{L^2(\Ga)}^2
\\
&\quad \leq
 \intO (\v_1 \otimes \v) : \Grad\v \dx 
 +\intO \mu \Grad\phi_1 \cdot \v \dx 
 +\intO \mu_2 \Grad\phi \cdot \v \dx
\\
&\qquad 
 - \intG \left(\gamma(\alpha \Psi_1)-\gamma(\alpha \Psi_2)\right) \v_2\cdot \v \dS
 - \intG \Psi_1\Gradg\Theta \cdot \v \dS
 - \intG \Psi\Gradg\Theta_2 \cdot \v \dS.
 \end{aligned}
\end{align}
Here, we have used the assumption \eqref{ass:vis-fric} for the friction term $\gamma$
and the relation
$$
\intO (\v \otimes \v_2) : \Grad\v \dx
= - \frac12 \intO (\Div \, \v_2) |\v|^2 \dx 
=0.
$$
Choosing $(\zeta,\xi)=(\phi,\Psi) \in \DD_\alpha$ in \eqref{WF:CH}, 
we find
\begin{align}
\label{Test:CH:L2}
\begin{aligned}
&\ddt \left[ \frac{1}{2} \norm{\phi}_{L^2(\Omega)}^2 + \frac{1}{2} \norm{\Psi}_{L^2(\Ga)}^2 \right]
+ \intO \Grad \mu \cdot \Grad \phi \dx
+ \frac{1}{\alpha^2}\intG \Grad_\Ga \Theta \cdot \Grad_\Ga \Psi \dS 
\\
&\quad = 
 \intO \phi_1\v \cdot \Grad\phi \dx
 + \intO \phi\v_2 \cdot \Grad\phi \dx
  + \intG \Psi_1\v \cdot \Gradg\Psi \dS
  + \intG \Psi\v_2 \cdot \Gradg\Psi \dS.
  \end{aligned}
\end{align}
We observe that 
\begin{align*}
 \intO \Grad \mu \cdot \Grad \phi \dx
&= - \intO \mu \Delta \phi \dx 
+ \intG \mu \partial_\n \phi \dS 
\notag\\
&= \norm{\Delta \phi}_{L^2(\Omega)}^2 - \intO \left( F'(\phi_1)-F'(\phi_2)\right) \Delta \phi \dx 
+ \intG \alpha \Theta \partial_\n \phi \dS
\notag\\
&
\begin{aligned}
&= \norm{\Delta \phi}_{L^2(\Omega)}^2 - \intO \left( F'(\phi_1)-F'(\phi_2)\right) \Delta \phi \dx 
+ \norm{\Theta}_{L^2(\Ga)}^2 
\\
&\quad + \intG \alpha^2 \Theta \Delta_\Ga \Psi \dS
- \intG \Theta \left( \alpha G'(\alpha\Psi_1)-\alpha G'(\alpha \Psi_2) \right) \dS
\end{aligned}
\\
&= \norm{\Delta \phi}_{L^2(\Omega)}^2 - \intO \left( F'(\phi_1)-F'(\phi_2)\right) \Delta \phi \dx 
+ \norm{\Theta}_{L^2(\Ga)}^2 
\notag\\
&\quad - \alpha^2 \intG \Grad_\Ga \Theta \cdot \Grad_\Ga \Psi \dS
- \intG \Theta \left(\alpha G'(\alpha\Psi_1)-\alpha G'(\alpha\Psi_2) \right) \dS.
\notag
\end{align*}
Since $\Div\, \v_2=0$ in $\Omega$, we further deduce that
\begin{align*}
\intO \phi\v_2 \cdot \Grad\phi \dx  
= \intO \v_2 \cdot \Grad \left( \frac12 \phi^2 \right)\dx 
= 0
\end{align*}
via integration by parts.
Hence, the differential equation \eqref{Test:CH:L2} can be reformulated as 
\begin{align}
\label{Test:CH:L2-2}
\begin{aligned}
&\ddt \left[ \frac{1}{2} \norm{\phi}_{L^2(\Omega)}^2
    + \frac{1}{2} \norm{\Psi}_{L^2(\Ga)}^2 \right]
+ \norm{\Delta \phi}_{L^2(\Omega)}^2  
+ \norm{\Theta}_{L^2(\Ga)}^2 
\\
&\quad = 
 \intO \phi_1\v \cdot \Grad\phi \dx
  + \intG \Psi_1\v \cdot \Gradg \Psi \dS
  + \intG \Psi \v_2 \cdot \Gradg \Psi \dS
\\
 &\qquad +
 \left( \alpha^2-\frac{1}{\alpha^2}\right) \intG \Gradg \Theta \cdot \Gradg \Psi \dS
 + \intO \left( F'(\phi_1)-F'(\phi_2)\right) \Delta \phi \dx 
 \\
 &\qquad 
 + \intG \Theta \left( \alpha G'(\alpha \Psi_1)-\alpha G'(\alpha \Psi_2) \right) \dS.
 \end{aligned}
\end{align}
Next, choosing $(\zeta,\xi)=(\mu,\Theta) \in\DD_\alpha$ in \eqref{WF:CH}, we find
\begin{align}
\label{Test:CH:H1}
\begin{aligned}
 &\bigang{(\delt \phi,\delt \Psi)}{(\mu,\Theta)}_{\DD_\alpha}
 + \norm{\Grad \mu}_{L^2(\Omega)}^2
 +\frac{1}{\alpha^2}\norm{\Grad_\Ga \Theta}_{L^2(\Gamma)}^2
\\
 &\quad =
  \intO \phi_1\v \cdot \Grad\mu \dx
  + \intO \phi\v_2 \cdot \Grad\mu \dx
  + \intG \Psi_1\v \cdot \Gradg\Theta \dS
  + \intG \Psi\v_2 \cdot \Gradg\Theta \dS.
  \end{aligned}
\end{align}
Recalling that $(\phi_i,\psi_i) \in L^2(0,T; \HH^3)$, we infer that $F'(\phi_i) \in H^1(\Omega)$ and $G'(\alpha \Psi_i) \in H^1(\Gamma)$ a.e.~in $[0,T]$ for $i\in\{1,2\}$. We further conclude from \eqref{REL:POT} that
\begin{align*}
    F'(\phi_1) - F'(\phi_2) = F'(\alpha \Psi_1) - F'(\alpha \Psi_2) 
    = \alpha^2 \big( G'(\alpha \Psi_1) - G'(\alpha \Psi_2) \big)
    \quad\text{a.e.~on $\Sigma$},
\end{align*}
which directly entails
$\big( F'(\phi_1) - F'(\phi_2) , \alpha G'(\alpha \Psi_1) - \alpha G'(\alpha \Psi_2) \big) \in \DD_\alpha$, for almost every $t \in [0,T]$.
In view of \eqref{EQ:UNIQ:STRG}, we thus have
\begin{align*}
    ( -\Lap\phi , -\alpha^2 \Lapg\Psi + \alpha \deln\phi ) \in L^2(0,T; \DD_\alpha).
\end{align*}
Hence, exploiting Proposition \ref{chainrule} with $\kappa=\alpha^2$ and $\sigma=\alpha$,  the first term in \eqref{Test:CH:H1} can be expressed as
\begin{align}
\begin{aligned}
\bigang{(\delt \phi,\delt \Psi)}{(\mu,\Theta)}_{\DD_\alpha}
&= \bigang{(\delt \phi,\delt \Psi)}{(-\Delta \phi,-\alpha^2\Delta_\Ga \Psi+ \alpha \partial_\n \phi)}_{\DD_\alpha} 
\\[1ex]
&\quad 
+ \bigang{(\delt \phi,\delt \Psi)}{\big(F'(\phi_1)-F'(\phi_2), \alpha G'(\alpha \Psi_1)-\alpha G'(\alpha \Psi_2)\big)}_{\DD_\alpha}
\\
&= \ddt \left[ \frac12 \| \Grad \phi\|_{L^2(\Omega)}^2+ \frac{\alpha^2}{2} \| \Grad_\Ga \Psi\|_{L^2(\Ga)}^2 \right]
\\[1ex]
&\quad +
\bigang{(\delt \phi,\delt \Psi)}{\big(F'(\phi_1)-F'(\phi_2),\alpha G'(\alpha \Psi_1)- \alpha G'(\alpha \Psi_2)\big)}_{\DD_\alpha}.
\end{aligned}
\end{align}
Using this identity, \eqref{Test:CH:H1} can be rewritten as
\begin{align}
\label{Test:CH:H1-2}
\begin{aligned}
 &\ddt \left[ \frac12 \| \Grad \phi\|_{L^2(\Omega)}^2+ \frac{\alpha^2}{2} \| \Grad_\Ga \Psi\|_{L^2(\Ga)}^2 \right]
 + \norm{\Grad \mu}_{L^2(\Omega)}^2
 +\frac{1}{\alpha^2}\norm{\Grad_\Ga \Theta}_{L^2(\Gamma)}^2
\\
 &\quad =
  \intO \phi_1\v \cdot \Grad\mu \dx
  + \intO \phi\v_2 \cdot \Grad\mu \dx
  + \intG \Psi_1\v \cdot \Gradg\Theta \dS
  + \intG \Psi\v_2 \cdot \Gradg\Theta \dS
\\[1ex]
& \qquad -\bigang{(\delt \phi,\delt \Psi)}{\big(F'(\phi_1)-F'(\phi_2),\alpha G'(\alpha \Psi_1)-\alpha G'(\alpha \Psi_2)\big)}_{\DD_\alpha}.
\end{aligned}
\end{align}
Summing \eqref{Test:CH:L2-2} and \eqref{Test:CH:H1-2}, we find
\begin{align}
\label{Test:CH}
\begin{aligned}
 &\ddt \left[ \frac12 \| \phi\|_{H^1(\Omega)}^2 + \frac{1}{2} \| \Psi\|_{L^2(\Ga)}^2 
 + \frac{\alpha^2}{2} 
 \| \Gradg\Psi\|_{L^2(\Ga)}^2 \right]
 \\
 &\qquad + \norm{\Grad \mu}_{L^2(\Omega)}^2
 +\frac{1}{\alpha^2}
 \norm{\Gradg \Theta}_{L^2(\Gamma)}^2
 +\norm{\Delta \phi}_{L^2(\Omega)}^2  
 +\norm{\Theta}_{L^2(\Ga)}^2 
\\
&\quad = 
 \intO \phi_1\v \cdot \Grad\phi \dx
  +  \intG \Psi_1\v \cdot \Gradg\Psi \dS
  + \intG \Psi \v_2 \cdot \Gradg\Psi \dS
 \\
 &\qquad +
 \left( \alpha^2-\frac{1}{\alpha^2}\right) \intG \Gradg \Theta \cdot \Gradg \Psi \dS
 + \intO \left( F'(\phi_1)-F'(\phi_2)\right) \Delta \phi \dx 
 \\
 &\qquad 
 + \intG \Theta \left( \alpha G'(\alpha \Psi_1)-\alpha G'(\alpha \Psi_2) \right) \dS
  +\intO \phi_1\v \cdot \Grad\mu \dx
  \\
 &\qquad 
 + \intO \phi\v_2 \cdot \Grad\mu \dx
  + \intG \Psi_1\v \cdot \Gradg\Theta \dS
  + \intG \Psi\v_2 \cdot \Gradg\Theta \dS
 \\[1ex]
& \qquad-\bigang{(\delt \phi,\delt \Psi)}{\big(F'(\phi_1)-F'(\phi_2), \alpha G'(\alpha \Psi_1)- \alpha G'(\alpha \Psi_2)\big)}_{\DD_\alpha}.
\end{aligned}
\end{align}
Observing that
$$
\intO \mu \Grad\phi_1 \cdot \v \dx +
 \intO \phi_1\v \cdot \Grad\mu \dx=0,
$$
which follows via integration by parts since $\Div \, \v=0$ in $\Omega$,
and collecting \eqref{Test:NS} and \eqref{Test:CH} together, we end up with
\begin{align}
\label{Test}
\begin{aligned}
 &\ddt \left[ \frac12 \norm{ \v}_{L^2(\Omega)}^2 + 
 \frac12 \| \phi\|_{H^1(\Omega)}^2 +\frac12 \| \Psi\|_{L^2(\Ga)}^2 
 + \frac{\alpha^2}{2} \| \Gradg\Psi\|_{L^2(\Ga)}^2 \right]
\\
&\quad 
+ \norm{\Grad \v}_{L^2(\Omega)}^2 
+\gamma_0 \norm{\v}_{L^2(\Ga)}^2
 + \norm{\Grad \mu}_{L^2(\Omega)}^2
 +\frac{1}{\alpha^2}
 \norm{\Grad_\Ga \Theta}_{L^2(\Gamma)}^2
 +\norm{\Delta \phi}_{L^2(\Omega)}^2  
 +\norm{\Theta}_{L^2(\Ga)}^2 
\\
&\quad 
\leq \intO (\v_1 \otimes \v) : \Grad\v \dx
+\intO \mu_2 \Grad\phi \cdot \v \dx
- \intG \left(\gamma(\alpha \Psi_1)-\gamma(\alpha \Psi_2)\right) \v_2\cdot \v \dS
\\
&\qquad 
- \intG \Psi\Gradg\Theta_2 \cdot \v \dS
+ \intO \phi_1\v \cdot \Grad\phi \dx
  +  \intG \Psi_1\v \cdot \Gradg\Psi \dS
  + \intG \Psi\v_2 \cdot \Gradg\Psi \dS
 \\
 &\qquad
 + \intO \phi\v_2 \cdot \Grad\mu \dx
 + \intG \Psi\v_2 \cdot \Gradg\Theta \dS
 +\left( \alpha^2-\frac{1}{\alpha^2}\right) \intG \Gradg \Theta \cdot \Gradg \Psi \dS
\\
& \qquad
 + \intO \left( F'(\phi_1)-F'(\phi_2)\right) \Delta \phi \dx
 + \intG \Theta \left( \alpha G'(\alpha \Psi_1)
 -G'(\alpha \Psi_2) \right) \dS 
\\[1ex]
& \qquad
-\bigang{(\delt \phi,\delt \Psi)}{\big(F'(\phi_1)-F'(\phi_2), \alpha G'(\alpha \Psi_1)-\alpha G'(\alpha \Psi_2)\big)}_{\DD_\alpha}.
\end{aligned}
\end{align}

In order to estimate the right-hand side in \eqref{Test}, we recall that
\begin{align}
\label{WS:bound}
\norm{\v_i}_{L^\infty(0,T;L^2(\Omega))}\leq C, \quad 
\norm{(\phi_i,\Psi_i)}_{L^\infty(0,T;\mathcal{H}^1)}\leq C
\end{align}
for $i\in\{1,2\}$. By the Ladyzhenskaya inequality and the Poincaré inequality \eqref{bs-Poincare:2}, we have 
\begin{align}
\label{T1}
\begin{aligned}
\left| \intO (\v_1 \otimes \v) : \Grad\v \dx
\right|  
&\leq \norm{\v_1}_{L^4(\Omega)} \norm{\v}_{L^4(\Omega)} \norm{\Grad \v}_{L^2(\Omega)}
\\
&\leq C \norm{\v_1}_{L^4(\Omega)}
\norm{\v}_{L^2(\Omega)}^\frac12
\norm{\v}_{H^1(\Omega)}^\frac12  \norm{\Grad \v}_{L^2(\Omega)} 
\\
&\leq C 
\norm{\v_1}_{L^4(\Omega)} 
\norm{\v}_{L^2(\Omega)}^\frac12
 \left(  \norm{\Grad \v}_{L^2(\Omega)} 
 + \gamma_0 \norm{\v}_{L^2(\Gamma)}
 \right)^\frac32
\\
 &\leq \frac18 \norm{\Grad \v}_{L^2(\Omega)}^2
 + \frac{\gamma_0}{12} \norm{\v}_{L^2(\Gamma)}^2 
 + C \norm{\v_1}_{L^4(\Omega)}^4 \norm{\v}_{L^2(\Omega)}^2.
 \end{aligned}
\end{align}
Moreover, using Sobolev's embedding theorem as well as Young's inequality, we obtain
\begin{align}
\label{T2}
\begin{aligned}
\left| 
\intO \mu_2 \Grad\phi \cdot \v \dx \right|
&\leq \norm{\mu_2}_{L^6(\Omega)} 
\norm{\Grad \phi}_{L^3(\Omega)} 
\norm{\v}_{L^2(\Omega)}
\\
&\leq \frac{\varpi}{10} \norm{\phi}_{H^2(\Omega)}^2
+ C \norm{\mu_2}_{H^1(\Omega)}^2 
\norm{\v}_{L^2(\Omega)}^2,
\end{aligned}
\end{align}
where $\varpi$ is a small constant which will be chosen later on.
Since $d=2$, we have the continuous embeddings $H^1(\Gamma)\hookrightarrow L^\infty(\Gamma)$ and $H^{1/2}(\Gamma)\emb L^p(\Gamma)$ for every $p\in[1,\infty)$. Due to the trace theorem, it holds $H^1(\Omega)\emb H^{1/2}(\Gamma)$ and consequently $H^1(\Omega) \emb L^p(\Gamma)$ for every $p\in[1,\infty)$.
Using \eqref{WS:bound} as well as the assumption that $\gamma'$ is locally bounded, we thus get
\begin{align}
\label{T3}
\begin{aligned}
\left| - \intG \left(\gamma(\alpha \Psi_1)-\gamma(\alpha \Psi_2)\right) \v_2\cdot \v \dS
\right|    
&\leq C \norm{\Psi}_{L^6(\Gamma)} \norm{\v_2}_{L^3(\Gamma)} \norm{\v}_{L^2(\Gamma)}
\\
&\leq \frac{\gamma_0}{12} \norm{\v}_{L^2(\Gamma)}^2 
+C \norm{\v_2}_{H^1(\Omega)}^2
\norm{\Psi}_{H^1(\Gamma)}^2.
\end{aligned}
\end{align}
Exploiting the bulk-surface Poincar\'e inequality \eqref{bs-Poincare:2},
the trace theorem and Sobolev's embedding theorem, we obtain
\begin{align}
\label{T4}
\begin{aligned}
\left| - \intG \Psi\Gradg\Theta_2 \cdot \v \dS \right|    
& \leq \norm{\Psi}_{L^6(\Ga)} 
\norm{\Theta_2}_{H^1(\Ga)} 
\norm{\v}_{L^3(\Ga)}
\\
&\leq C \norm{\Psi}_{H^1(\Ga)} \norm{\Theta_2}_{H^1(\Ga)} 
\norm{\v}_{H^1(\Omega)}
\\
&\leq \frac18 \norm{\Grad \v}_{L^2(\Omega)}^2 
 + \frac{\gamma_0}{12} \norm{\v}_{L^2(\Gamma)}^2 
 + C \norm{\Theta_2}_{H^1(\Ga)}^2 \norm{\Psi}_{H^1(\Ga)}^2.
 \end{aligned}
\end{align}
Proceeding similarly, we derive the estimates
\begin{align}
\label{T5}
&\begin{aligned}
\left| \intO \phi_1\v \cdot \Grad\phi \dx\right|   
&\leq \norm{\phi_1}_{L^6(\Omega)} 
\norm{\v}_{L^3(\Omega)}
\norm{\Grad \phi}_{L^2(\Omega)}
\\
&\leq 
\frac18 \norm{\Grad \v}_{L^2(\Omega)}^2 
 + \frac{\gamma_0}{12} \norm{\v}_{L^2(\Gamma)}^2 
 + C \norm{\phi}_{H^1(\Omega)}^2,
 \end{aligned}
\\[1ex]
\label{T6}
&\begin{aligned}
\left|  \intG \Psi_1\v \cdot \Gradg\Psi \dS\right|
&\leq 
\norm{\Psi_1}_{L^6(\Ga)}
\norm{\v}_{L^3(\Ga)} 
\norm{\Grad_\Ga \Psi}_{L^2(\Ga)}
\\
&\leq 
C\norm{\v}_{H^1(\Omega)} 
\norm{\Psi}_{H^1(\Ga)}
\\
&\leq \frac18 \norm{\Grad \v}_{L^2(\Omega)}^2 
 + \frac{\gamma_0}{12} \norm{\v}_{L^2(\Gamma)}^2 
 + C \norm{\Psi}_{H^1(\Ga)}^2,
 \end{aligned}
\\[1ex]
\label{T7}
&\begin{aligned}
\left|  \intG \Psi \v_2 \cdot \Gradg \Psi \dS \right|    
&
\leq \norm{\Psi}_{L^6(\Ga)} \norm{\v_2}_{L^3(\Ga)}
\norm{\Grad_\Ga \Psi}_{L^2(\Ga)}
\\
&\leq 
C \norm{\v_2}_{H^1(\Omega)} \norm{\Psi}_{H^1(\Ga)}^2,
\end{aligned}
\\[1ex]
\label{T8}
&\begin{aligned}
\left| \intO \phi\v_2 \cdot \Grad\mu \dx\right|    
&\leq \norm{\phi}_{L^6(\Omega)} \norm{\v_2}_{L^3(\Omega)} \norm{\Grad \mu}_{L^2(\Omega)}
\\
&\leq 
\frac14 \norm{\Grad \mu}_{L^2(\Omega)}^2
+ C \norm{\v_2}_{H^1(\Omega)}^2 \norm{\phi}_{H^1(\Omega)}^2,
\end{aligned}
\\[1ex]
\label{T9}
&\begin{aligned}
\left|  \intG \Psi \v_2 \cdot \Gradg \Theta \dS \right|
&\leq \norm{\Psi}_{L^6(\Ga)} \norm{\v_2}_{L^3(\Ga)} \norm{\Grad_\Ga \Theta}_{L^2(\Ga)}
\\
&\leq \frac{1}{8\alpha^2} \norm{\Grad_\Ga \Theta}_{L^2(\Ga)}^2 
+C \norm{\v_2}_{H^1(\Omega)}^2 \norm{\Psi}_{H^1(\Ga)}^2
\end{aligned}
\end{align}
and
\begin{align}
\begin{split}
\left| \left( \alpha^2-\frac{1}{\alpha^2}\right) \intG \Gradg \Theta \cdot \Gradg \Psi \dS\right|
    &\leq 
\frac{1}{8\alpha^2} \norm{\Grad_\Ga \Theta}_{L^2(\Ga)}^2 
+C \norm{\Psi}_{H^1(\Ga)}^2.
\end{split}
\end{align}
Recalling \eqref{ass:pot} as well as \eqref{WS:bound}, we deduce that
\begin{equation}
\label{T10}
\begin{split}
&\left| \intO \left( F'(\phi_1)-F'(\phi_2)\right) \Delta \phi \dx \right|  
\\
&\quad \leq C \intO \left( 1+|\phi_1|^{p-2}+ |\phi_2|^{p-2}\right) |\phi| |\Delta \phi| \dx
\\
& \quad \leq C \left( 1+ \norm{\phi_1}_{L^{3(p-2)}(\Omega)}^{p-2}+
\norm{\phi_2}_{L^{3(p-2)}(\Omega)}^{p-2}
\right) \norm{\phi}_{L^6(\Omega)} \norm{\Delta \phi}_{L^2(\Omega)}
\\
& \quad \leq C \norm{\phi}_{H^1(\Omega)} 
\norm{\Delta \phi}_{L^2(\Omega)}
\\
&\quad \leq \frac12 \norm{\Delta \phi}_{L^2(\Omega)}^2 
+ C \norm{\phi}_{H^1(\Omega)}^2.
\end{split}
\end{equation}
Recalling that $H^1(\Gamma)\hookrightarrow L^\infty(\Gamma)$ (since $d=2$), we also infer that
\begin{equation}
\label{T11}
\begin{split}
\left| \intG \Theta \left( \alpha G'(\alpha \Psi_1)- \alpha G'(\alpha \Psi_2) \right) \dS \right|
&\leq C \norm{\Theta}_{L^2(\Ga)} \norm{\Psi}_{L^2(\Ga)}
\\
&\leq \frac{1}{2} \norm{\Theta}_{L^2(\Ga)}^2
+ C\norm{\Psi}_{H^1(\Ga)}^2.
\end{split}
\end{equation}
Regarding the final term, it follows by duality that
\begin{align}
\label{T12}
\begin{aligned}
&\left| -\bigang{(\delt \phi,\delt \Psi)}{\big(F'(\phi_1)-F'(\phi_2),\alpha G'(\alpha \Psi_1)-\alpha G'(\alpha \Psi_2)\big)}_{\DD_\alpha}
\right|
\\
&\quad \leq \norm{(\delt \phi,\delt \Psi)}_{\mathcal{D}_\alpha '} 
\bignorm{\big(F'(\phi_1)-F'(\phi_2), \alpha G'(\alpha \Psi_1)-\alpha G'(\alpha \Psi_2)\big)}_{\mathcal{D}_\alpha}
\\
&\quad 
\leq C \norm{(\delt \phi,\delt \Psi)}_{\mathcal{D}_\alpha '} 
\left( \bignorm{F'(\phi_1)-F'(\phi_2)}_{H^1(\Omega)}+ \bignorm{G'(\alpha \Psi_1)-G'(\alpha \Psi_2)}_{H^1(\Ga)}\right).
\end{aligned}
\end{align}
By comparison in \eqref{WF:CH}, we observe that
\begin{align}
\label{T12-11}
\begin{aligned}
\norm{(\delt \phi,\delt \Psi)}_{\mathcal{D}_\alpha '} 
&
\leq
C \left( \norm{\phi_1}_{L^\infty(\Omega)} \norm{\v}_{L^2(\Omega)} + \norm{\phi}_{L^6(\Omega)} \norm{\v_2}_{L^3(\Omega)}
+\norm{\Psi_1}_{L^\infty(\Gamma)} \norm{\v}_{L^2(\Gamma)} \right.
\\
&\quad \left. + \norm{\Psi}_{L^6(\Gamma)}\norm{\v_2}_{L^3(\Gamma)} + \norm{\nabla \mu}_{L^2(\Omega)}
+\norm{\Grad_\Ga \Theta}_{L^2(\Gamma)}
\right)
\\
&
\leq
C \left( \norm{\phi_1}_{L^\infty(\Omega)} \norm{\v}_{L^2(\Omega)} + \norm{\phi}_{L^6(\Omega)} \norm{\v_2}_{L^3(\Omega)}
+\norm{\v}_{L^2(\Gamma)}
 \right.
\\
&\quad \left. 
+ \norm{\Psi}_{L^6(\Gamma)}\norm{\v_2}_{L^3(\Gamma)}  + \norm{\nabla \mu}_{L^2(\Omega)}
+\norm{\Grad_\Ga \Theta}_{L^2(\Gamma)}
\right).
\end{aligned}
\end{align}
Moreover, recalling \eqref{ass:pot}, we obtain
\begin{align}
\label{T12-12}
\begin{aligned}
&\bignorm{F'(\phi_1)-F'(\phi_2)}_{H^1(\Omega)}
\\
&\leq \bignorm{F'(\phi_1)-F'(\phi_2)}_{L^2(\Omega)}
+
\bignorm{F''(\phi_1) \Grad \phi_1 -F''(\phi_2) \Grad \phi_2}_{L^2(\Omega)}
\\
&\leq \bignorm{F'(\phi_1)-F'(\phi_2)}_{L^2(\Omega)}
+ \bignorm{\left( F''(\phi_1)-F''(\phi_2)\right) \Grad \phi_1}_{L^2(\Om)} +
\bignorm{F''(\phi_2) \Grad \phi}_{L^2(\Om)}
\\
&\leq \left( \int_{\Om} \left| \int_0^1 F''(\tau \phi_1 + (1-\tau)\phi_2) \dtau  \right|^2 \phi^2 \dx \right)^\frac12
\\
&\quad +\left( \int_{\Om} \left| \int_0^1 F'''(\tau \phi_1 + (1-\tau)\phi_2) \dtau  \right|^2 \phi^2 |\Grad \phi_1|^2 \dx \right)^\frac12
+ \bignorm{F''(\phi_2)}_{L^6(\Om)} \|\Grad \phi\|_{L^3(\Om)}
\\
&\leq 
\left( 1+ \norm{\phi_1}_{L^{3(p-2)}(\Omega)}^{p-2}+
\norm{\phi_2}_{L^{3(p-2)}(\Omega)}^{p-2} \right) \norm{\phi}_{L^6(\Om)} 
\\
&\quad +
C \left( 1+ \norm{\phi_1}_{L^{12(p-3)}(\Omega)}^{p-3}+
\norm{\phi_2}_{L^{12(p-3)}(\Omega)}^{p-3} \right) \norm{\phi}_{L^6(\Om)} \norm{\nabla \phi_1}_{L^4(\Om)}
+C \norm{\Grad \phi}_{L^3(\Om)}
\\
&\leq C 
\left(1+ \norm{\phi_1}_{W^{1,4}(\Omega)}\right)\norm{\phi}_{H^1(\Omega)} +
C \norm{\phi}_{W^{1,3}(\Omega)}.
\end{aligned}
\end{align}
Similarly, we find
\begin{align}
\label{T12-13}
\begin{aligned}
&\bignorm{G'(\alpha \Psi_1)-G'(\alpha \Psi_2)}_{H^1(\Ga)}
\\
&\leq 
\bignorm{G'(\alpha \Psi_1)-G'(\alpha \Psi_2)}_{L^2(\Ga)}
+\alpha \bignorm{G''(\alpha\Psi_1) \Grad_\Ga \Psi_1 -G''(\alpha \Psi_2) \Grad_\Ga \Psi_2}_{L^2(\Ga)}
\\
&\leq C \norm{\Psi}_{L^2(\Ga)}
+ C \bignorm{\left( G''(\alpha \Psi_1)-G''(\alpha \Psi_2)\right)\Grad_\Ga \Psi_1}_{L^2(\Ga)} 
 + C \bignorm{G''(\alpha \Psi_2)\Grad_\Ga \Psi}_{L^2(\Ga)}
\\
&\leq C \norm{\Psi}_{L^2(\Ga)}
+ C \norm{\Psi}_{L^\infty(\Ga)} \norm{\Grad_\Ga \Psi_1}_{L^2(\Ga)}
+ C \norm{\Grad_\Ga \Psi}_{L^2(\Ga)}
\\
&\leq C \norm{\Psi}_{H^1(\Ga)}.
\end{aligned}
\end{align}
Combining the above 
estimates \eqref{T12-11}--\eqref{T12-13},  we obtain 
\begin{align}
\label{T12-2}
\begin{aligned}
&\left| -\bigang{(\delt \phi,\delt \Psi)}{\big(F'(\phi_1)-F'(\phi_2),\alpha G'(\alpha \Psi_1)-\alpha G'(\alpha \Psi_2)\big)}_{\DD_\alpha}
\right|
\\
&\leq  C \norm{\phi_1}_{L^\infty(\Omega)} \norm{\v}_{L^2(\Omega)} \left( \left(1+ \norm{\phi_1}_{W^{1,4}(\Omega)}\right)\norm{\phi}_{H^1(\Omega)} + \norm{\phi}_{W^{1,3}(\Omega)} + 
\norm{\Psi}_{H^1(\Ga)}\right)
\\
&\quad + C\norm{\v_2}_{L^3(\Omega)} \norm{\phi}_{L^6(\Omega)} 
\left( \left(1+ \norm{\phi_1}_{W^{1,4}(\Omega)}\right)\norm{\phi}_{H^1(\Omega)} + \norm{\phi}_{W^{1,3}(\Omega)} + 
\norm{\Psi}_{H^1(\Ga)}\right)
\\
&\quad +C\norm{\v_2}_{L^3(\Gamma)}  \norm{\Psi}_{L^6(\Gamma)}
\left( \left(1+ \norm{\phi_1}_{W^{1,4}(\Omega)}\right)\norm{\phi}_{H^1(\Omega)} + \norm{\phi}_{W^{1,3}(\Omega)} + 
\norm{\Psi}_{H^1(\Ga)}\right)
\\
&\quad +C\left( \norm{\v}_{L^2(\Ga)} + \norm{\nabla \mu}_{L^2(\Omega)}
+\norm{\Grad_\Ga \Theta}_{L^2(\Gamma)}
\right) 
\\
&\qquad \times \left( \left(1+ \norm{\phi_1}_{W^{1,4}(\Omega)}\right)\norm{\phi}_{H^1(\Omega)} + \norm{\phi}_{W^{1,3}(\Omega)} + 
\norm{\Psi}_{H^1(\Ga)}\right)
\\
&=: R_1+R_2+R_3+R_4.
\end{aligned}
\end{align}
Using Sobolev's embedding theorem, the Gagliardo--Nirenberg interpolation
inequality, the trace theorem and Young's inequality, we infer that
\begin{align*}
    R_1 
    &\leq 
     C\norm{\phi_1}_{L^\infty(\Omega)}
     \left(1+\norm{\phi_1}_{W^{1,4}(\Omega)}
     \right) \left(
     \norm{\v}_{L^2(\Omega)}^2+ \norm{\phi}_{H^1(\Omega)}^2\right)
\notag\\
     &\quad+
     C\norm{\phi_1}_{L^\infty(\Omega)} \norm{\v}_{L^2(\Omega)}
     \norm{\phi}_{W^{1,3}(\Omega)} + 
     C\norm{\phi_1}_{L^\infty(\Omega)} \left( \norm{\v}_{L^2(\Omega)}^2+
\norm{\Psi}_{H^1(\Ga)}^2\right)
\notag\\
    &\leq
    \frac{\varpi}{10}\norm{\phi}_{H^2(\Omega)}^2 + C\norm{\phi_1}_{L^\infty(\Omega)}
     \left(1+\norm{\phi_1}_{W^{1,4}(\Omega)}
     \right) \left(
     \norm{\v}_{L^2(\Omega)}^2+ \norm{\phi}_{H^1(\Omega)}^2\right)
\notag\\
     &\quad
     + C \left(\norm{\phi_1}_{L^\infty(\Omega)}+\norm{\phi_1}_{L^\infty(\Omega)}^2 \right) \left( \norm{\v}_{L^2(\Omega)}^2+
\norm{\Psi}_{H^1(\Ga)}^2\right)
\notag\\
&\leq
    \frac{\varpi}{10}\norm{\phi}_{H^2(\Omega)}^2
 + C \left( 1+
\norm{\phi_1}_{L^\infty(\Omega)}^2
+ \norm{\phi_1}_{W^{1,4}(\Omega)}^2
     \right) \left(
     \norm{\v}_{L^2(\Omega)}^2+ \norm{\phi}_{H^1(\Omega)}^2+
\norm{\Psi}_{H^1(\Ga)}^2\right),
\\[1ex]
     R_2
&\leq C \norm{\v_2}_{L^3(\Omega)} \left(1+ \norm{\phi_1}_{W^{1,4}(\Omega)}\right) \left( \norm{\phi}_{H^1(\Om)}^2+ \norm{\Psi}_{H^1(\Ga)}^2 \right) 
\notag\\
&\quad 
+ C\norm{\v_2}_{L^3(\Omega)} \norm{\phi}_{L^6(\Omega)}
\norm{\phi}_{W^{1,3}(\Omega)}
\notag\\
&\leq 
\frac{\varpi}{10}\norm{\phi}_{H^2(\Omega)}^2
+ C \left( 1+  \norm{\phi_1}_{W^{1,4}(\Omega)}^2 + \norm{\v_2}_{L^3(\Omega)}^2\right) \left( \norm{\phi}_{H^1(\Om)}^2+ \norm{\Psi}_{H^1(\Ga)}^2 \right), 
\\[1ex]
    R_3 
&\leq 
C\left( \norm{\v_2}_{L^3(\Gamma)} +
\norm{\v_2}_{L^3(\Gamma)} \norm{\phi_1}_{W^{1,4}(\Omega)} \right)
\left( \norm{\phi}_{H^1(\Omega)}^2+
\norm{\Psi}_{H^1(\Ga)}^2 \right)
\notag\\
&\quad +C\norm{\v_2}_{L^3(\Gamma)}  \norm{\Psi}_{L^6(\Gamma)}
\norm{\phi}_{W^{1,3}(\Omega)} 
\notag\\
&\leq 
C\left( 1+\norm{\v_2}_{H^1(\Omega)}^2 +
 \norm{\phi_1}_{W^{1,4}(\Omega)}^2 \right)
\left( \norm{\phi}_{H^1(\Omega)}^2+
\norm{\Psi}_{H^1(\Ga)}^2 \right)
\notag\\
&\quad +C\norm{\v_2}_{H^1(\Omega)}  \norm{\Psi}_{H^1(\Gamma)}
\norm{\phi}_{H^1(\Omega)}^\frac23 
\norm{\phi}_{H^2(\Omega)}^\frac13 
\notag\\
&\leq 
\frac{\varpi}{8}\norm{\phi}_{H^2(\Omega)}^2
+ C\left( 1+\norm{\v_2}_{H^1(\Omega)}^2 +
 \norm{\phi_1}_{W^{1,4}(\Omega)}^2 \right)
\left( \norm{\phi}_{H^1(\Omega)}^2+
\norm{\Psi}_{H^1(\Ga)}^2 \right)
\notag\\
&\quad 
+ 
C\norm{\v_2}_{H^1(\Omega)}^\frac65  \norm{\Psi}_{H^1(\Gamma)}^\frac65
\norm{\phi}_{H^1(\Omega)}^\frac45 
\notag\\
&\leq 
\frac{\varpi}{10}\norm{\phi}_{H^2(\Omega)}^2
+ C\left( 1+\norm{\v_2}_{H^1(\Omega)}^2 +
 \norm{\phi_1}_{W^{1,4}(\Omega)}^2 \right)
\left( \norm{\phi}_{H^1(\Omega)}^2
+\norm{\Psi}_{H^1(\Ga)}^2 \right), 
\\[1ex]
    R_4
    &\leq 
    \frac{\gamma_0}{12} \norm{\v}_{L^2(\Gamma)}^2 
    + \frac14 \norm{\nabla \mu}_{L^2(\Omega)}^2 
    +\frac{1}{4\alpha^2} \norm{\Grad_\Ga \Theta}_{L^2(\Gamma)}^2
\notag\\
&\qquad + C  \left(1+ \norm{\phi_1}_{W^{1,4}(\Omega)}^2\right)
\left( \norm{\phi}_{H^1(\Omega)}^2 
+ \norm{\Psi}_{H^1(\Ga)}^2\right) 
+C \norm{\phi}_{H^1(\Omega)}^\frac43
\norm{\phi}_{H^2(\Omega)}^\frac23
\notag\\
&\leq \frac{\gamma_0}{12} \norm{\v}_{L^2(\Gamma)}^2 
    + \frac14 \norm{\nabla \mu}_{L^2(\Omega)}^2 
    +\frac{1}{4\alpha^2} \norm{\Grad_\Ga \Theta}_{L^2(\Gamma)}^2
    +\frac{\varpi}{10}\norm{\phi}_{H^2(\Omega)}^2
\notag\\
&\qquad + C  \left(1+ \norm{\phi_1}_{W^{1,4}(\Omega)}^2\right)
\left( \norm{\phi}_{H^1(\Omega)}^2 
+ \norm{\Psi}_{H^1(\Ga)}^2\right). 
\end{align*}
Plugging these estimates into \eqref{T12-2} and exploiting \eqref{T1}--\eqref{T11}, we eventually obtain from \eqref{Test} the differential inequality
\begin{align}
\label{Diff-IN}
\begin{aligned}
 &\ddt \left[ \frac12 \norm{ \v}_{L^2(\Omega)}^2 + 
 \frac12 \| \phi\|_{H^1(\Omega)}^2
 + \frac12 \| \Psi\|_{L^2(\Ga)}^2 
 + \frac{\alpha^2}{2} \| \Gradg\Psi\|_{L^2(\Ga)}^2  \right]
 + \frac12 \norm{\Grad \v}_{L^2(\Omega)}^2 
\\
&\quad 
 + \frac{\gamma_0}{2} \norm{\v}_{L^2(\Ga)}^2
 + \frac12\norm{\Grad \mu}_{L^2(\Omega)}^2
 + \frac{1}{2\alpha^2}
 \norm{\Grad_\Ga \Theta}_{L^2(\Gamma)}^2
 + \frac12\norm{\Delta \phi}_{L^2(\Omega)}^2  
 + \frac{1}{2} \norm{\Theta}_{L^2(\Ga)}^2 
\\
 &\leq \frac{\varpi}{2}\norm{\phi}_{H^2(\Om)}^2
 + C\mathcal{F}(t) \left[ \frac12 \norm{ \v}_{L^2(\Omega)}^2 + 
 \frac12 \norm{\phi}_{H^1(\Omega)}^2
 + \frac{1}{2} \| \Psi\|_{L^2(\Ga)}^2 + \frac{\alpha^2}{2} \| \Gradg\Psi\|_{L^2(\Ga)}^2 \right],
\end{aligned}
\end{align}
where
\begin{equation*}
       \mathcal{F}(t) := 
       1+ \norm{\v_1(t)}_{L^4(\Omega)}^4+
       \norm{\v_2(t)}_{H^1(\Omega)}^2 +
       \norm{\mu_2(t)}_{H^1(\Omega)}^2+
       \norm{\theta_2(t)}_{H^1(\Ga)}^2+
       \norm{\phi_1(t)}_{W^{1,4}(\Omega)}^2.
\end{equation*}
for almost all $t\in[0,T]$.

Now, reversing the change of variables \eqref{CofV*}, we infer that the original variables $(\v,\phi,\psi,\mu,\theta)$ satisfy the estimate
\begin{align}
\label{Diff-IN*}
\begin{aligned}
 &\ddt \left[ \frac12 \norm{ \v}_{L^2(\Omega)}^2 + 
 \frac12 \| \phi\|_{H^1(\Omega)}^2
 + \frac{1}{2\alpha^2} \| \psi\|_{L^2(\Ga)}^2 
 + \frac{1}{2} \| \Gradg\psi\|_{L^2(\Ga)}^2  \right]
 + \frac12 \norm{\Grad \v}_{L^2(\Omega)}^2 
\\
&\quad 
 + \frac{\gamma_0}{2} \norm{\v}_{L^2(\Ga)}^2
 + \frac12\norm{\Grad \mu}_{L^2(\Omega)}^2
 + \frac{1}{2}
 \norm{\Grad_\Ga \theta}_{L^2(\Gamma)}^2
 + \frac12\norm{\Delta \phi}_{L^2(\Omega)}^2  
 + \frac{\alpha^2}{2} \norm{\theta}_{L^2(\Ga)}^2 
\\
 &\leq \frac{\varpi}{2}\norm{\phi}_{H^2(\Om)}^2
 + C\mathcal{F}(t) \left[ \frac12 \norm{ \v}_{L^2(\Omega)}^2 + 
 \frac12 \norm{\phi}_{H^1(\Omega)}^2
 + \frac{1}{2\alpha^2} \| \psi\|_{L^2(\Ga)}^2 + \frac{1}{2} \| \Gradg\psi\|_{L^2(\Ga)}^2 \right].
\end{aligned}
\end{align}
In order to complete the proof, it remains to absorb the $H^2(\Omega)$-norm of $\phi$. 
Since
 \begin{subequations}
    \label{EQ:UNIQ:STRG*}
    \begin{alignat}{2}
    \label{EQ:UNIQ:STRG:A*}
        \mu &= -\Lap \phi + F'(\phi_1)-F'(\phi_2) &&\quad\text{a.e.~in $Q$},\\
    \label{EQ:UNIQ:STRG:B*}
        \theta &= - \Lapg \psi +  G'(\psi_1)-\alpha G'( \psi_2) 
        +  \deln\phi &&\quad\text{a.e.~on $\Sigma$},
    \end{alignat}
    \end{subequations}
 we use the regularity theory for elliptic problems with bulk-surface coupling (see \cite[Theorem~3.3]{knopf-liu}) to derive that
\begin{align}
\begin{aligned}
\norm{\phi}_{H^2(\Omega)}^2 
&\le \norm{(\phi,\psi)}_{\HH^2}^2 
\le C \norm{(-\Lap\phi , -\Lapg\psi + \deln\phi)}_{\LL^2}^2 
\\
& \le C \Big( \norm{\Lap\phi}_{L^2(\Omega)}^2 
    + \norm{\theta}_{L^2(\Gamma)}^2
    + \bignorm{G'(\psi_1) - G'(\psi_2)}_{L^2(\Gamma)}^2 \Big)
\\
&\leq C_0 \left( \norm{\Delta \phi}_{L^2(\Omega)}^2 + \alpha^2 \norm{\theta}_{L^2(\Gamma)}^2 \right) 
+ \frac{C}{\alpha^2} \norm{\psi}_{L^2(\Ga)}^2.
\end{aligned}
\end{align}
Multiplying the above estimate by $\frac{1}{4C_0}$, and adding the resulting inequality to \eqref{Diff-IN}, we obtain
\begin{align*}
\begin{aligned}
 &\ddt \left[ \frac12 \norm{ \v}_{L^2(\Omega)}^2 + 
 \frac12 \| \phi\|_{H^1(\Omega)}^2
 + \frac{1}{2\alpha^2} \| \psi\|_{L^2(\Ga)}^2 + \frac 12 \| \Gradg\psi\|_{L^2(\Ga)}^2 \right]
 + \frac12 \norm{\Grad \v}_{L^2(\Omega)}^2 
+\frac{\gamma_0}{2} \norm{\v}_{L^2(\Ga)}^2
\\
&\quad 
 + \frac12\norm{\Grad \mu}_{L^2(\Omega)}^2
 +\frac12\norm{\Grad_\Ga \theta}_{L^2(\Gamma)}^2
 +\frac14\norm{\Delta \phi}_{L^2(\Omega)}^2  +\frac{\alpha^2}{4} \norm{\theta}_{L^2(\Ga)}^2 
 +\left( \frac{1}{4C_0} - \frac{\varpi}{2}\right) \norm{\phi}_{H^2(\Omega)}^2
\\
 &\leq C\mathcal{F}(t) \left[ \frac12 \norm{ \v}_{L^2(\Omega)}^2 + 
 \frac12 \| \phi\|_{H^1(\Omega)}^2
 + \frac{1}{2\alpha^2} \| \psi\|_{L^2(\Ga)}^2 + \frac 12 \| \Gradg\psi\|_{L^2(\Ga)}^2 \right].
 \end{aligned}
 \end{align*}
 Thus, choosing $\varpi\leq \frac{1}{2C_0}$, we conclude that 
 \begin{equation}
\label{Diff-IN3}
\begin{split}
&\ddt \left[ \frac12 \norm{ \v}_{L^2(\Omega)}^2 + 
 \frac12 \| \phi\|_{H^1(\Omega)}^2
 + \frac{1}{2\alpha^2} \| \psi\|_{L^2(\Ga)}^2 + \frac 12 \| \Gradg\psi\|_{L^2(\Ga)}^2 \right]
 \\
 &\quad \leq C\mathcal{F}(t) \left[ \frac12 \norm{ \v}_{L^2(\Omega)}^2 + 
 \frac12 \| \phi\|_{H^1(\Omega)}^2
 + \frac{1}{2\alpha^2} \| \psi\|_{L^2(\Ga)}^2 + \frac 12 \| \Gradg\psi\|_{L^2(\Ga)}^2 \right].
 \end{split}
 \end{equation}
 Since $\mathcal{F}\in L^1(0,T)$, Gronwall's lemma 
  directly implies that the stability estimate \eqref{cont-dep*} holds for all $t\in [0,T]$. As a consequence, the uniqueness of weak solutions immediately follows. Thus, the proof is complete.
\end{proof}

\appendix
\section{A chain rule formula.}
\label{A}
\setcounter{equation}{0}
Let $\kappa$ and $\sigma$ be two positive constants. We define the function $\Phi_\kappa^{\sigma}: \LL^2 \rightarrow [0,\infty]$ as
\begin{align}
    \Phi_\kappa^\sigma(\phi,\psi)=
    \left\{
    \begin{aligned}
    &\int_{\Omega} \frac12 |\nabla \phi|^2 \dx + 
    \int_{\Ga} \frac{\kappa}{2} |\nabla_\Ga \psi|^2 \dS, \quad &&\text{if } (\phi,\psi)\in \DD_\sigma,\\
    &+\infty, \quad &&\text{otherwise}.
    \end{aligned}
    \right.
\end{align}

\begin{proposition}
For any $\kappa \in (0,\infty)$ and $\sigma \in (0,\infty)$, the function $\Phi_\kappa^\sigma$ is convex, lower semicontinuous and proper. The subgradient $\mathcal{A}_\kappa^\sigma=\partial \Phi_\kappa^\sigma$ is given by 
\begin{equation}
\label{subgradient}
\mathcal{A}_\kappa^\sigma(\phi,\psi)=(-\Delta \phi, - \kappa \Delta_\Ga \psi+ \sigma \partial_\n \phi), \quad \text{for all $(\phi,\psi)\in D(\mathcal{A}_\kappa^\sigma)$},
\end{equation}
where
\begin{equation}
\label{subgradient-dom}
D(\mathcal{A}_\kappa)=  \HH^2 \cap \DD_\sigma.
\end{equation}
\end{proposition}

\begin{proof}
It is easy to see that $\Phi_\kappa^\sigma$ is convex and proper. The lower semicontinuity of $\Phi_\kappa^\sigma$ immediately follows from the weak lower semicontinuity of the norm with respect to the weak convergence. 

Let us now consider the operator $\mathcal{A}_\kappa^\sigma: D(\mathcal{A}_\kappa^\sigma)\subset \LL^2 \rightarrow \LL^2$ as defined in \eqref{subgradient}--\eqref{subgradient-dom}. First of all, we observe that $\mathcal{A}_\kappa^\sigma$ is well defined since $\partial_\n \phi \in H^{\frac12}(\Ga)$ for any $\phi \in H^2(\Omega)$.
For any $(\phi,\psi)\in D(\mathcal{A}_\kappa^\sigma)$ and $(\zeta,\xi) \in \DD_\sigma$, we have 
\begin{align*}
    (\mathcal{A}_\kappa^\sigma(\phi,\psi), (\phi,\psi)-(\zeta,\xi))_{\LL^2}
    &=\intO -\Delta \phi (\phi-\zeta)\dx + \intG (- \kappa \Delta_\Ga \psi+ \sigma \partial_\n \phi)(\psi-\xi) \dS\\
    &=\intO \nabla \phi \cdot \nabla(\phi-\zeta) \dx 
      - \intG \partial_\n \phi (\phi-\zeta) \dS\\
    &\quad  + \intG \kappa \nabla_\Ga \psi \cdot\nabla_\Ga(\psi-\xi)\dS + \intG  \sigma \partial_\n \phi(\psi-\xi) \dS\\
    &=\intO \nabla \phi \cdot \nabla(\phi-\zeta) \dx 
      - \sigma \intG \partial_\n \phi (\psi-\xi) \dS\\
    &\quad + \intG \kappa \nabla_\Ga \psi \cdot\nabla_\Ga(\psi-\xi)\dS + \sigma \intG  \partial_\n \phi(\psi-\xi) \dS\\
    &=\intO \nabla \phi \cdot \nabla(\phi-\zeta) \dx 
    + \intG \kappa \nabla_\Ga \psi \cdot\nabla_\Ga(\psi-\xi)\dS\\
    &\geq \int_{\Omega} \frac12 |\nabla \phi|^2 \dx - \int_{\Omega} \frac12 |\nabla \zeta|^2 \dx\\
    &\quad +\int_{\Ga} \frac{\kappa}{2} |\nabla_\Ga \psi|^2 \dS
    - \int_{\Ga} \frac{\kappa}{2} |\nabla_\Ga \xi|^2 \dS.
\end{align*}
Thus, we found
$$
(\mathcal{A}_\kappa^\sigma(\phi,\psi), (\phi,\psi)-(\zeta,\xi))_{\LL^2}
\geq \Phi_\kappa^\sigma(\phi,\psi)-\Phi_\kappa^\sigma(\zeta,\xi), \quad \text{for all $(\phi,\psi)\in D(\mathcal{A}_\kappa^\sigma), \ (\zeta,\xi) \in \DD_\sigma$.}
$$
The above inequality can be easily extended to 
$$
(\mathcal{A}_\kappa^\sigma(\phi,\psi), (\phi,\psi)-(\zeta,\xi))_{\LL^2}
\geq \Phi_\kappa^\sigma(\phi,\psi)-\Phi_\kappa^\sigma(\zeta,\xi), \quad \text{for all $(\phi,\psi)\in D(\mathcal{A}_\kappa^\sigma), \ (\zeta,\xi) \in \LL^2$.}
$$
This implies that $\mathcal{A}_\kappa\subset \partial \Phi_\kappa$, i.e., $D(\mathcal{A}_\kappa)\subset D(\partial \Phi_\kappa)$ and $\mathcal{A}_\kappa(\phi,\psi)\in \partial \Phi_\kappa(\phi,\psi)$ for all $(\phi,\psi)\in D(\mathcal{A}_\kappa)$. 

In order to conclude that $\mathcal{A}_\kappa^\sigma=\partial \Phi_\kappa^\sigma$, we are left to show that $\mathcal{A}_\kappa^\sigma$ is maximal monotone in $\LL^2\times \LL^2$, namely $R(I+A_\kappa^\sigma)= \LL^2$. For this purpose, let us fix $(f,g) \in \LL^2$ and consider the equation $(\phi,\psi)+\mathcal{A}_\kappa^\sigma(\phi,\psi)=(f,g)$, which reads as
\begin{subequations}
	\label{EQ:subgradient}
	\begin{align}
		\phi-\Lap \phi &= f && \text{ in } \Omega, \\
		\psi- \kappa \Lapg \psi + \sigma \deln \phi &= g && \text{ on } \Gamma, \\
			\label{EQ:subgradient-bc}
			 \phi \vert_\Ga &= \sigma \psi && \text{ on } \Gamma.
		\end{align}
	\end{subequations}
	A pair $(\phi, \psi)\in \DD_\sigma$ is a corresponding weak solution if
	\begin{equation}
	\label{Eqweak:subgradient}
	    \bigscp{(\phi,\psi)}{(\zeta,\xi)}_{\LL^2}+
	    \intO \nabla \phi \cdot \nabla \zeta \dx 
        + \intG  \kappa \nabla_\Ga \psi \cdot\nabla_\Ga\xi \dS =
	    \bigscp{(f,g)}{(\zeta,\xi)}_{\LL^2}
	\end{equation}
for all $(\zeta,\xi) \in \DD_\sigma$.
The existence and uniqueness of $(\phi,\psi)\in \DD_\sigma$ solving \eqref{Eqweak:subgradient} follows directly from the Lax-Milgram theorem. Next, using regularity theory for elliptic problems with bulk-surface coupling (see \cite[Theorem 3.3]{knopf-liu}), where $(f-\phi,g-\psi)\in\LL^2$ is interpreted as the source term,
we infer that $(\phi,\psi) \in \HH^2$ and \eqref{EQ:subgradient} holds almost everywhere. This entails the desired conclusion. 	
\end{proof}

\begin{proposition}
\label{chainrule}
Let $\kappa\in (0,\infty)$ and $\sigma\in (0,\infty)$. Assume that $(\phi,\psi) \in L^2(0,T; \DD_\sigma)\cap H^1(0,T;\DD_\sigma')$ such that $(-\Delta \phi,- \kappa \Delta_\Ga \psi+ \sigma \partial_\n \phi)\in L^2(0,T; \DD_\sigma)$. Then, the chain rule formula 
\begin{equation}
\label{CRformula}
\ddt \left[ \frac12 \| \Grad \phi\|_{L^2(\Omega)}^2+ \frac{\kappa}{2} \| \Grad_\Ga \psi\|_{L^2(\Ga)}^2 \right]
=\bigang{(\delt \phi,\delt \psi)}{(-\Delta \phi,-\kappa \Delta_\Ga \psi+ \sigma \partial_\n \phi)}_{\DD_\sigma}
\end{equation}
holds for almost every $t \in [0,T]$.
\end{proposition}

\begin{proof}
We define the function $\widetilde{\Phi}: \LL^2 \rightarrow [0,\infty]$ as
\begin{align}
    \widetilde{\Phi}(\phi,\psi)=
    \left\{
    \begin{aligned}
    &\int_{\Omega} \frac12 |\phi|^2 \dx + 
    \int_{\Ga} \frac12 |\psi|^2 \dS, \quad &&\text{if } (\phi,\psi)\in \LL^2,\\
    &+\infty, \quad &&\text{otherwise}.
    \end{aligned}
    \right.
\end{align}
It is easily seen that $\widetilde{\Phi}$ is convex, lower semicontinuous and proper. In addition, the subgradient $\mathcal{B}: \LL^2\rightarrow \LL^2$, $\mathcal{B}(\phi,\psi)= (\phi,\psi)$ is Lipschitz in $\HH^0$. Let us now consider 
$\Phi^\star=\Phi_\kappa^\sigma+\widetilde{\Phi}: \LL^2\rightarrow [0,\infty]$, which clearly is a convex, lower semicontinuous, proper functional. Moreover, we have 
$$
    \Phi^\star(\phi,\psi) \geq \widetilde{\Phi}(\phi,\psi)= \norm{(\phi,\psi)}_{\LL^2}^2
    \quad\text{for all $(\phi,\psi) \in\LL^2$.}
$$
Thanks to \cite[Lemma 2.1, Chapter IV]{showalter2013monotone}  , we infer that $\mathcal{A}_\kappa^\sigma+\mathcal{B}$ is a maximal monotone operator with domain 
$D(\mathcal{A}_\kappa^\sigma+\mathcal{B})=D(\mathcal{A}_\kappa^\sigma)$, which coincides with $ \partial \Phi^\star$. In light of our assumptions, we recall that 
$(\phi,\psi) \in L^2(0,T; \DD_\sigma)\cap H^1(0,T;\DD_\sigma')$. Then, since $\phi|_\Ga = \sigma \psi$ on $\Ga$, we observe that 
$$
\partial \Phi^\star(\phi,\psi)= 
(-\Delta \phi+\phi, -\kappa \Delta_\Ga \psi+ \sigma \partial_\n \phi +\psi)
\in L^2(0,T;\DD_\sigma).
$$
Therefore, we are in position to apply \cite[Lemma 4.1]{rocca2004universal}, which entails that $t\rightarrow \Phi^\star(\phi(t),\psi(t))$ is absolutely continuous on $[0,T]$ and the chain rule
\begin{equation}
\label{CR-prel}
\begin{split}
&\ddt \left[ \frac12 \| \phi\|_{H^1(\Omega)}^2+ \frac12 \|  \psi\|_{L^2(\Ga)}^2 +\frac{\kappa}{2} \| \Grad_\Ga \psi\|_{L^2(\Ga)}^2 \right]\\
&\quad =\bigang{(\delt \phi,\delt \psi)}{(-\Delta \phi +\phi,- \kappa \Delta_\Ga \psi+ \sigma \partial_\n \phi +\psi)}_{\DD_\sigma}
\end{split}
\end{equation}
holds almost everywhere in $[0,T]$.
On the other hand, by the classical chain rule formula (or, alternatively, by using \cite[Lemma 4.1]{rocca2004universal} with $\mathcal{J}= \widetilde{\Phi}$), we  know that 
\begin{equation}
\label{CR-prel2}
\ddt \left[ \frac12 \| \phi\|_{L^2(\Omega)}^2+ \frac12 \|  \psi\|_{L^2(\Ga)}^2 \right]
=\bigang{(\delt \phi,\delt \psi)}{( \phi,\psi)}_{\DD_\sigma}
\end{equation}
almost everywhere in $[0,T]$. Finally, subtracting \eqref{CR-prel2} from \eqref{CR-prel}, we reach the claimed \eqref{CRformula}.
\end{proof}

\section*{Acknowledgement}
The authors want to thank Harald Garcke for helpful discussions. Part of this work was done while the first author was visiting the Department of Mathematics at University of Regensburg whose hospitality is gratefully acknowledged.
Andrea Giorgini was supported by the MUR grant Dipartimento di Eccellenza 2023-2027 and by Gruppo Nazionale per l'Analisi Matematica, la Probabilità e le loro Applicazioni (GNAMPA) of Istituto Nazionale per l'Alta Matematica (INdAM).
Patrik Knopf was partially supported by the Deutsche Forschungsgemeinschaft (DFG, German Research Foundation): on the one hand by the DFG-project 524694286, and on the other hand by the RTG 2339 ``Interfaces, Complex Structures, and Singular Limits''.
Their support is gratefully acknowledged.

\section*{Competing Interests and funding}
The authors do not have any financial or non-financial interests that are directly or indirectly related to the work submitted for publication.

\section*{Data availability statement}
No further data is used in this manuscript.


\footnotesize

\bibliographystyle{abbrv}
\bibliography{GK}

\begin{thebibliography}{10}

\bibitem{Abels2009}
H.~Abels.
\newblock On a diffuse interface model for two-phase flows of viscous,
  incompressible fluids with matched densities.
\newblock {\em Arch. Ration. Mech. Anal.}, 194(2):463--506, 2009.

\bibitem{Abels2012}
H.~Abels.
\newblock Strong well-posedness of a diffuse interface model for a viscous,
  quasi-incompressible two-phase flow.
\newblock {\em SIAM J. Math. Anal.}, 44(1):316--340, 2012.

\bibitem{abels2013existence}
H.~Abels, D.~Depner, and H.~Garcke.
\newblock Existence of weak solutions for a diffuse interface model for
  two-phase flows of incompressible fluids with different densities.
\newblock {\em J. Math. Fluid Mech.}, 15(3):453--480, 2013.

\bibitem{abels2013incompressible}
H.~Abels, D.~Depner, and H.~Garcke.
\newblock On an incompressible {N}avier-{S}tokes/{C}ahn-{H}illiard system with
  degenerate mobility.
\newblock {\em Ann. Inst. H. Poincar\'{e} C Anal. Non Lin\'{e}aire},
  30(6):1175--1190, 2013.

\bibitem{AbelsGarckeReview}
H.~Abels and H.~Garcke.
\newblock Weak solutions and diffuse interface models for incompressible
  two-phase flows.
\newblock In {\em Handbook of mathematical analysis in mechanics of viscous
  fluids}, pages 1267--1327. Springer, Cham, 2018.

\bibitem{AGG}
H.~Abels, H.~Garcke, and G.~Gr\"{u}n.
\newblock Thermodynamically consistent, frame indifferent diffuse interface
  models for incompressible two-phase flows with different densities.
\newblock {\em Math. Models Methods Appl. Sci.}, 22(3):1150013, 40, 2012.

\bibitem{AbelsWeber2021}
H.~Abels and J.~Weber.
\newblock Local well-posedness of a quasi-incompressible two-phase flow.
\newblock {\em J. Evol. Equ.}, 21(3):3477--3502, 2021.

\bibitem{Alt}
H.~W. Alt.
\newblock {\em Linear functional analysis}.
\newblock Universitext. Springer-Verlag London, Ltd., London, 2016.
\newblock An application-oriented introduction, Translated from the German
  edition by Robert N\"{u}rnberg.

\bibitem{BGN-book}
J.~W. Barrett, H.~Garcke, and R.~N\"{u}rnberg.
\newblock Parametric finite element approximations of curvature-driven
  interface evolutions.
\newblock In {\em Geometric partial differential equations. {P}art {I}},
  volume~21 of {\em Handb. Numer. Anal.}, pages 275--423.
  Elsevier/North-Holland, Amsterdam, 2020.

\bibitem{binder-frisch}
K.~Binder and H.~L. Frisch.
\newblock {Dynamics of surface enrichment: A theory based on the Kawasaki
  spin-exchange model in the presence of a wall}.
\newblock {\em Z. Phys. B}, 84:403–--418, 1991.

\bibitem{boyer1999}
F.~Boyer.
\newblock Mathematical study of multi-phase flow under shear through order
  parameter formulation.
\newblock {\em Asymptot. Anal.}, 20(2):175--212, 1999.

\bibitem{boyer2002theoretical}
F.~Boyer.
\newblock A theoretical and numerical model for the study of incompressible
  mixture flows.
\newblock {\em Comput. \& Fluids}, 31(1):41--68, 2002.

\bibitem{boyer}
F.~Boyer and P.~Fabrie.
\newblock {\em Mathematical tools for the study of the incompressible
  {N}avier--{S}tokes equations and related models}, volume 183 of {\em Applied
  Mathematical Sciences}.
\newblock Springer, New York, 2013.

\bibitem{cherfils2019compressible}
L.~Cherfils, E.~Feireisl, M.~Mich\'{a}lek, A.~Miranville, M.~Petcu, and
  D.~{Praz\'{a}k}.
\newblock The compressible {N}avier-{S}tokes-{C}ahn-{H}illiard equations with
  dynamic boundary conditions.
\newblock {\em Math. Models Methods Appl. Sci.}, 29(14):2557--2584, 2019.

\bibitem{CGM}
L.~Cherfils, S.~Gatti, and A.~Miranville.
\newblock A variational approach to a {C}ahn-{H}illiard model in a domain with
  nonpermeable walls.
\newblock {\em J. Math. Sci. (N.Y.)}, 189(4):604--636, 2013.
\newblock Problems in mathematical analysis. No. 69.

\bibitem{CFL}
P.~Colli, T.~Fukao, and K.~F. Lam.
\newblock On a coupled bulk-surface {A}llen--{C}ahn system with an affine
  linear transmission condition and its approximation by a {R}obin boundary
  condition.
\newblock {\em Nonlinear Anal.}, 184:116--147, 2019.

\bibitem{DiBenedetto}
E.~DiBenedetto.
\newblock {\em Real analysis}.
\newblock Birkh\"{a}user Advanced Texts: Basler Lehrb\"{u}cher. [Birkh\"{a}user
  Advanced Texts: Basel Textbooks]. Birkh\"{a}user Boston, Inc., Boston, MA,
  2002.

\bibitem{ding2007diffuse}
H.~Ding, P.~D. Spelt, and C.~Shu.
\newblock Diffuse interface model for incompressible two-phase flows with large
  density ratios.
\newblock {\em J. Comput. Phys.}, 226(2):2078--2095, 2007.

\bibitem{du2020phase}
Q.~Du and X.~Feng.
\newblock The phase field method for geometric moving interfaces and their
  numerical approximations.
\newblock In {\em Geometric partial differential equations. {P}art {I}},
  volume~21 of {\em Handb. Numer. Anal.}, pages 425--508.
  Elsevier/North-Holland, Amsterdam, [2020] \copyright 2020.

\bibitem{dziuk2013finite}
G.~Dziuk and C.~M. Elliott.
\newblock Finite element methods for surface {PDE}s.
\newblock {\em Acta Numer.}, 22:289--396, 2013.

\bibitem{freistuhler2017phase}
H.~Freist\"{u}hler and M.~Kotschote.
\newblock Phase-field and {K}orteweg-type models for the time-dependent flow of
  compressible two-phase fluids.
\newblock {\em Arch. Ration. Mech. Anal.}, 224(1):1--20, 2017.

\bibitem{GalGrasselli2010}
C.~Gal and M.~Grasselli.
\newblock Asymptotic behavior of a {C}ahn--{H}illiard--{N}avier--{S}tokes
  system in 2{D}.
\newblock {\em Ann. Inst. H. Poincar\'{e} C Anal. Non Lin\'{e}aire},
  27(1):401--436, 2010.

\bibitem{GGP2021}
C.~Gal, M.~Grasselli, and A.~Poiatti.
\newblock {A}llen--{C}ahn--{N}avier--{S}tokes--{V}oigt systems with moving
  contact lines.
\newblock {\em ResearchGate preprint 10.13140/RG.2.2.16038.86086}, 06 2021.

\bibitem{Gal}
C.~G. Gal.
\newblock A {C}ahn--{H}illiard model in bounded domains with permeable walls.
\newblock {\em Math. Methods Appl. Sci.}, 29(17):2009--2036, 2006.

\bibitem{GGM2016}
C.~G. Gal, M.~Grasselli, and A.~Miranville.
\newblock {C}ahn--{H}illiard--{N}avier--{S}tokes systems with moving contact
  lines.
\newblock {\em Calc. Var. Partial Differential Equations}, 55(3):Art. 50, 47,
  2016.

\bibitem{GGW2019}
C.~G. Gal, M.~Grasselli, and H.~Wu.
\newblock Global weak solutions to a diffuse interface model for incompressible
  two-phase flows with moving contact lines and different densities.
\newblock {\em Arch. Ration. Mech. Anal.}, 234(1):1--56, 2019.

\bibitem{Qiu-Wang}
H.~Gao, Z.~Qiu, and H.~Wang.
\newblock {Stochastic Cahn--Hilliard--Navier--Stokes equations with the dynamic
  boundary: Martingale weak solution, Markov selection}.
\newblock {\em preprint: arXiv:2205.12759}, 05 2022.

\bibitem{GK}
H.~Garcke and P.~Knopf.
\newblock Weak solutions of the {C}ahn--{H}illiard system with dynamic boundary
  conditions: a gradient flow approach.
\newblock {\em SIAM J. Math. Anal.}, 52(1):340--369, 2020.

\bibitem{GKY2022}
H.~Garcke, P.~Knopf, and S.~Yayla.
\newblock Long-time dynamics of the {C}ahn--{H}illiard equation with kinetic
  rate dependent dynamic boundary conditions.
\newblock {\em Nonlinear Anal.}, 215:Paper No. 112619, 44, 2022.

\bibitem{giga2017variational}
M.-H. Giga, A.~Kirshtein, and C.~Liu.
\newblock Variational modeling and complex fluids.
\newblock In {\em Handbook of mathematical analysis in mechanics of viscous
  fluids}, pages 73--113. Springer, Cham, 2018.

\bibitem{giorgini2021well}
A.~Giorgini.
\newblock Well-posedness of the two-dimensional {A}bels-{G}arcke-{G}r\"{u}n
  model for two-phase flows with unmatched densities.
\newblock {\em Calc. Var. Partial Differential Equations}, 60(3):Paper No. 100,
  40, 2021.

\bibitem{giorgini2022-3D}
A.~Giorgini.
\newblock Existence and stability of strong solutions to the
  {A}bels-{G}arcke-{G}r{\"u}n model in three dimensions.
\newblock {\em Interfaces Free Bound.}, 24(4):565--608, 2022.

\bibitem{GMT2019}
A.~Giorgini, A.~Miranville, and R.~Temam.
\newblock Uniqueness and regularity for the
  {N}avier--{S}tokes--{C}ahn--{H}illiard system.
\newblock {\em SIAM J. Math. Anal.}, 51(3):2535--2574, 2019.

\bibitem{GMS2011}
G.~Goldstein, A.~Miranville, and G.~Schimperna.
\newblock A {C}ahn--{H}illiard model in a domain with non-permeable walls.
\newblock {\em Phys. D}, 240(8):754--766, 2011.

\bibitem{GurtinPolignoneVinals}
M.~Gurtin, D.~Polignone, and J.~Vi\~{n}als.
\newblock Two-phase binary fluids and immiscible fluids described by an order
  parameter.
\newblock {\em Math. Models Methods Appl. Sci.}, 6(6):815--831, 1996.

\bibitem{heida2012development}
M.~Heida, J.~M\'{a}lek, and K.~R. Rajagopal.
\newblock On the development and generalizations of {C}ahn-{H}illiard equations
  within a thermodynamic framework.
\newblock {\em Z. Angew. Math. Phys.}, 63(1):145--169, 2012.

\bibitem{kenzler}
R.~Kenzler, F.~Eurich, P.~Maass, B.~Rinn, J.~Schropp, E.~Bohl, and
  W.~Dieterich.
\newblock Phase separation in confined geometries: solving the
  {C}ahn-{H}illiard equation with generic boundary conditions.
\newblock {\em Comput. Phys. Comm.}, 133(2-3):139--157, 2001.

\bibitem{KLLM2021}
P.~Knopf, K.~Lam, C.~Liu, and S.~Metzger.
\newblock Phase-field dynamics with transfer of materials: the
  {C}ahn-{H}illiard equation with reaction rate dependent dynamic boundary
  conditions.
\newblock {\em ESAIM Math. Model. Numer. Anal.}, 55(1):229--282, 2021.

\bibitem{knopf-lam}
P.~Knopf and K.~F. Lam.
\newblock Convergence of a {R}obin boundary approximation for a
  {C}ahn--{H}illiard system with dynamic boundary conditions.
\newblock {\em Nonlinearity}, 33(8):4191--4235, 2020.

\bibitem{knopf-liu}
P.~Knopf and C.~Liu.
\newblock On second-order and fourth-order elliptic systems consisting of bulk
  and surface {PDE}s: well-posedness, regularity theory and eigenvalue
  problems.
\newblock {\em Interfaces Free Bound.}, 23(4):507--533, 2021.

\bibitem{KS}
P.~Knopf and A.~Signori.
\newblock On the nonlocal {C}ahn--{H}illiard equation with nonlocal dynamic
  boundary condition and boundary penalization.
\newblock {\em J. Differential Equations}, 280:236--291, 2021.

\bibitem{Liu-Wu2019}
C.~Liu and H.~Wu.
\newblock An energetic variational approach for the {C}ahn--{H}illiard equation
  with dynamic boundary condition: model derivation and mathematical analysis.
\newblock {\em Arch. Ration. Mech. Anal.}, 233(1):167--247, 2019.

\bibitem{lowengrub1998quasi}
J.~Lowengrub and L.~Truskinovsky.
\newblock Quasi-incompressible {C}ahn-{H}illiard fluids and topological
  transitions.
\newblock {\em R. Soc. Lond. Proc. Ser. A Math. Phys. Eng. Sci.},
  454(1978):2617--2654, 1998.

\bibitem{Miranville-Book}
A.~Miranville.
\newblock {\em {The Cahn--Hilliard equation: Recent advances and
  applications}}, volume~95 of {\em CBMS-NSF Regional Conference Series in
  Applied Mathematics}.
\newblock Society for Industrial and Applied Mathematics (SIAM), Philadelphia,
  PA, 2019.

\bibitem{MW}
A.~Miranville and H.~Wu.
\newblock Long-time behavior of the {C}ahn-{H}illiard equation with dynamic
  boundary condition.
\newblock {\em J. Elliptic Parabol. Equ.}, 6(1):283--309, 2020.

\bibitem{pruss2016moving}
J.~Pr\"{u}ss and G.~Simonett.
\newblock {\em Moving interfaces and quasilinear parabolic evolution
  equations}, volume 105 of {\em Monographs in Mathematics}.
\newblock Birkh\"{a}user/Springer, [Cham], 2016.

\bibitem{Qian-Wang-Sheng}
T.~Qian, X.-P. Wang, and P.~Sheng.
\newblock A variational approach to moving contact line hydrodynamics.
\newblock {\em J. Fluid Mech.}, 564:333--360, 2006.

\bibitem{rocca2004universal}
E.~Rocca and G.~Schimperna.
\newblock Universal attractor for some singular phase transition systems.
\newblock {\em Phys. D}, 192(3-4):279--307, 2004.

\bibitem{shen2013mass}
J.~Shen, X.~Yang, and Q.~Wang.
\newblock Mass and volume conservation in phase field models for binary fluids.
\newblock {\em Commun. Comput. Phys.}, 13(4):1045--1065, 2013.

\bibitem{shokrpour2018diffuse}
M.~Shokrpour~Roudbari, G.~\c{S}im\c{s}ek, E.~H. van Brummelen, and K.~G.
  van~der Zee.
\newblock Diffuse-interface two-phase flow models with different densities: a
  new quasi-incompressible form and a linear energy-stable method.
\newblock {\em Math. Models Methods Appl. Sci.}, 28(4):733--770, 2018.

\bibitem{showalter2013monotone}
R.~E. Showalter.
\newblock {\em Monotone operators in {B}anach space and nonlinear partial
  differential equations}, volume~49 of {\em Mathematical Surveys and
  Monographs}.
\newblock American Mathematical Society, Providence, RI, 1997.

\bibitem{Strauss}
W.~A. Strauss.
\newblock On continuity of functions with values in various {B}anach spaces.
\newblock {\em Pacific J. Math.}, 19:543--551, 1966.

\bibitem{temam2012infinite}
R.~Temam.
\newblock {\em Infinite-dimensional dynamical systems in mechanics and
  physics}, volume~68 of {\em Applied Mathematical Sciences}.
\newblock Springer-Verlag, New York, second edition, 1997.

\bibitem{Wu-Review}
H.~Wu.
\newblock A review on the {C}ahn-{H}illiard equation: classical results and
  recent advances in dynamic boundary conditions.
\newblock {\em Electron. Res. Arch.}, 30(8):2788--2832, 2022.

\end{thebibliography}

\end{document}